\documentclass[11pt]{article}

\usepackage{float}

\usepackage{amsmath,latexsym,amsthm,amsfonts,amssymb,%
ifthen,color,wrapfig,hyperref,cite,fullpage,rotate,lmodern,aliascnt,%
datetime,graphicx,algorithmic,algorithm,enumerate,enumitem,%
todonotes,cleveref,subcaption,tabularx,xcolor,colortbl%
}

\usepackage{wasysym}

\usepackage{tikz}
\usetikzlibrary{calc,3d}
\usetikzlibrary{intersections,decorations.pathmorphing,shapes,decorations.pathreplacing,fit,backgrounds,decorations.pathreplacing}
\tikzset{black node/.style={draw, circle, fill = black, minimum size = 4pt, inner sep = 0pt}}
\tikzset{hblack node/.style={draw, circle, fill = black, minimum size = 5pt, inner sep = 0pt}}

\usepackage[greek,english]{babel}
\usepackage[utf8]{inputenc}
\usepackage{alphabeta}
\usepackage{relsize}

\hypersetup{
	pdftitle = {Minor-Obstructions for Apex Pseudoforests},
	pdfauthor=  {Alexandros Leivaditis, Alexandros Singh, Giannos Stamoulis, Dimitrios Thilikos, and Konstantinos Tsatsanis},
	colorlinks = true,
	urlcolor = black!50!red,
	linkcolor = black!50!red,
	citecolor = black!50!green,
	menucolor = {red}
	filecolor = {cyan},
	anchorcolor = {black}
	urlcolor = {magenta},
	runcolor = {cyan},
	linkbordercolor = {blue}
}

\newcommand{\remove}[1]{}

\definecolor{Gray}{gray}{0.90}
\definecolor{Black}{rgb}{0,0, 0}
\definecolor{Blue}{rgb}{0, 0 ,1}
\definecolor{Red}{rgb}{1, 0 ,0}
\definecolor{RRed}{rgb}{.5, 0 ,0}
\definecolor{White}{rgb}{1, 1, 1}
\definecolor{Yellow}{rgb}{.55,.55,0}
\definecolor{mustard}{rgb}{1.0, 0.86, 0.35}
\definecolor{applegreen}{rgb}{0.55, 0.71, 0.0}

\newcommand{\mynewtheorem}[2]{
\newaliascnt{#1}{dummy}
\newtheorem{#1}[#1]{#2}
\aliascntresetthe{#1}
\expandafter\def\csname #1autorefname\endcsname{#2}
}

\theoremstyle{plain}
\mynewtheorem{theorem}{Theorem}
\mynewtheorem{proposition}{Proposition}
\mynewtheorem{corollary}{Corollary}
\mynewtheorem{lemma}{Lemma}

\theoremstyle{definition}
\mynewtheorem{definition}{Definition}
\theoremstyle{remark}
\mynewtheorem{sublemma}{Sublemma}
\mynewtheorem{claim}{Claim}
\mynewtheorem{remark}{Remark}
\mynewtheorem{fact}{Fact}
\mynewtheorem{conjecture}{Conjecture}
\mynewtheorem{question}{Question}       
\mynewtheorem{observation}{Observation}
\mynewtheorem{problem}{Problem}
\mynewtheorem{note}{Note}
\newcommand{\cupall}{\pmb{\pmb{\bigcup}}}

\usepackage{cancel}
\usepackage{titling}
\newcommand*\samethanks[1][\value{footnote}]{\footnotemark[#1]}
\thanksmarkseries{arabic}

\newcommand{\cactobs}[1]{{\bf obs}({\cal A}_{#1}({\cal S}))}%Μην το ξανά αλλάξεις αυτό!
\newcommand{\kapexmo}[1]{{\cal A}_{#1}({\cal S})}

\title{{Minor-Obstructions for Apex Sub-unicyclic  Graphs\thanks{\href{http://www.cs.upc.edu/~sedthilk/oapf/apexpse_foto.jpg}{}Emails:\! {\scriptsize  \sf \href{mailto:livaditisalex@gmail.com}{livaditisalex@gmail.com},\!
\href{mailto:alexsingh@di.uoa.gr}{alexsingh@di.uoa.gr},\!
\href{mailto:giannosstam@di.uoa.gr}{giannosstam@di.uoa.gr},\!
\href{mailto:sedthilk@thilikos.info}{sedthilk@thilikos.info},\!
\href{mailto:kostistsatsanis@gmail.com}{kostistsatsanis@gmail.com},\!
\href{mailto:vasiliki.velona@upf.edu}{vasiliki.velona@upf.edu}.}
}}}

\author{\bigskip 
Alexandros Leivaditis%
\thanks{Department of Mathematics, National and Kapodistrian University of Athens, Athens, Greece.}% 
\and 
Alexandros Singh%
\thanks{Inter-university Postgraduate Programme ``Algorithms, Logic, and Discrete Mathematics'' (ALMA).}
\and 
Giannos Stamoulis%
\samethanks[3]  
\and
Dimitrios  M. Thilikos\thanks{AlGCo project-team, LIRMM, CNRS, Universit\'e de Montpellier, Montpellier, France.}$\ ^,$\samethanks[2]$\ ^,$\samethanks[3]$\ ^,$\thanks{Supported by projects DEMOGRAPH (ANR-16-CE40-0028) and ESIGMA (ANR-17-CE23-0010).} \and Konstantinos Tsatsanis\samethanks[3]\and
Vasiliki Velona\thanks{Department of Economics,  Universitat Pompeu Fabra, Barcelona, Spain.}$\ ^,$\thanks{Department of Mathematics, Universitat Politècnica  de Catalunya,  Barcelona, Spain.}$\ ^,$\thanks{Supported under an FPI grant from the MINECO research project MTM2015-67304-PI.}}

\date{}

\begin{document}
 
 \maketitle

 \vspace{-1cm}
 \begin{abstract}
  \noindent A graph is {\em sub-unicyclic} if it contains at most one cycle. We also say that a graph $G$ is  {\em $k$-apex sub-unicyclic} if it can become sub-unicyclic by removing $k$ of its vertices.
  We identify 29 graphs that are the minor-obstructions of the class of {$1$-apex} sub-unicyclic graphs,  i.e., the set of all minor minimal graphs that do not belong in this class. 
  For bigger values of $k$, we give an exact structural characterization of all the cactus graphs that are minor-obstructions  of {$k$-apex} sub-unicyclic graphs and we enumerate them. This implies that, for every $k$,   the class of $k$-apex sub-unicyclic graphs has  at least $0.34\cdot k^{-2.5}(6.278)^{k}$ minor-obstructions.
 \end{abstract}
 
 \medskip
 
  
 \section{Introduction}
% 
 %\removed{%
 
% \end{document}
 A graph is called {\em unicyclic}
 \cite{Harary91Grap} if it contains exactly one cycle and is called {\em sub-unicyclic} if it contains at 
 most one cycle.  Notice that sub-unicyclic graphs are exactly the subgraphs of unicyclic graphs.\medskip 
 
A graph $H$ is a minor of a graph $G$ if a graph isomorphic to $H$
can be obtained by some subgraph of $G$ after a series of contractions.
We say that a graph class ${\cal G}$ is {\em minor-closed}
if every minor of every graph in ${\cal G}$ also belongs in ${\cal G}$.
We also define ${\bf obs}({\cal G})$, called the {\em minor-obstruction set} of ${\cal G}$,  as the set of minor-minimal 
graphs not in ${\cal G}$. It is easy to verify that if ${\cal G}$ is minor-closed, then $G\in{\cal G}$ iff $G$ excludes all graphs in ${\bf obs}({\cal G})$  as a minor.  Because of Roberson and Seymour theorem~\cite{RobertsonSXX}, ${\bf obs}({\cal G})$ is finite for every minor-closed graph class. That way,
${\bf obs}({\cal G})$ can be seen as a {\em complete characterization}
of ${\cal G}$ via a finite set of forbidden graphs.   The identification of 
${\bf obs}({\cal G})$ for distinct minor-closed classes has attracted a lot of attention in Graph Theory (see~\cite{Mattman16forb,Adler08open} for related surveys). 

There are several ways to construct minor-closed graph classes from others (see~\cite{Mattman16forb}). A popular one is to consider the set of all {\em $k$-apices} of a graph class ${\cal G}$, denoted by 
${\cal A}_{k}({\cal G})$, that contains all graphs that can give a graph in ${\cal G}$, after the removal 
of at most $k$ vertices. It is easy to verify that if ${\cal G}$ is minor closed, then the same holds 
for ${\cal A}_{k}({\cal G})$ as well, for every non-negative integer $k$. It was also proved in~\cite{AdlerGK08comp}  that the construction of  ${\bf obs}({\cal A}_{k}({\cal G}))$, given ${\bf obs}({\cal G})$ and $k$, is a computable problem.

A lot of research has been oriented to the (partial) identification of the minor-obstructions of the $k$-apices, of several minor-closed graph classes.
For instance, 
 ${\bf obs}({\cal A}_{k}({\cal G}))$ has been identified for $k\in\{1,\ldots,7\}$ when ${\cal G}$ is the set of edgeless  graphs~\cite{CattellDDFL00onco,DinneenX02mino,DinneenV12obst}, 
and for $k\in\{1,2\}$ when ${\cal G}$ is the set of acyclic graphs~\cite{DinneenCF01forb}.
Recently, ${\bf obs}({\cal A}_{1}({\cal G}))$ was identified 
when ${\cal G}$ is the class of outerplanar graphs~\cite{Ding2016} and 
when ${\cal G}$ is the class of cactus graphs (as  announced in~\cite{DziobiakD2013}). 
A particularly popular problem is  identification of 
${\bf obs}({\cal A}_{k}({\cal G}))$ when ${\cal G}$ is the class of planar graphs  (see e.g.,~\cite{LiptonMMPRT16sixv,Mattman16forb,Yu06more}). The best  advance 
on this question was done recently by Jobson and Kézdy \cite{JobsonK18allm} who identified all 2-connected  minor-obstructions of 1-apex planar graphs (see also~\cite{MattmanP16thea,Pierce14PhDThesis}).
Another recent result  is the identification of  ${\bf obs}({\cal A}_{1}({\cal P}))$  where  ${\cal P}$
is the class of all pseudoforests, i.e., graphs where all connected components are sub-unicyclic~\cite{Leivaditis18mino}.

A different direction is to upper-bound the size 
of the graphs ${\bf obs}({\cal A}_{k}({\cal G}))$ by some function of $k$. In this direction,
it was proved in~\cite{FominLMS12plan} that the size of the graphs in ${\bf obs}({\cal A}_{k}({\cal G}))$ is bounded by a  polynomial on $k$ in the case where the ${\bf obs}({\cal G})$ contains some planar graph (see also~\cite{ZorosPhD17}). Another line of research is  to prove lower bounds to the size of  ${\bf obs}({\cal A}_{k}({\cal G}))$. In this direction Michael Dinneen  proved in~\cite{Dinneen97}
that, if all graphs in ${\bf obs}({\cal G})$ are connected, then  $|{\bf obs}({\cal A}_{k}({\cal G}))|$ is exponentially big. To show this,  Dinneen proved a more general structural theorem
claiming that, under the former connectivity assumption, every connected component of a non-connected graph
%\aliv{Εδώ δεν πρέπει να πούμε "..of a disconnected graph..";}
in  ${\bf obs}({\cal A}_{k}({\cal G}))$ is a graph in ${\bf obs}({\cal A}_{k’}({\cal G}))$, for some $k’<k$. Another way to prove lower bounds to $|{\bf obs}({\cal A}_{k}({\cal G}))|$
is to completely characterize, for every $k$, the set 
${\bf obs}({\cal A}_{k}({\cal G}))\cap {\cal H}$, for some graph class ${\cal H}$, 
and then lower bound $|{\bf obs}({\cal A}_{k}({\cal G}))|$ by counting (asymptotically or exactly) all the graphs in ${\bf obs}({\cal A}_{k}({\cal G}))\cap {\cal H}$. This last approach has been applied in~\cite{RueST12oute} 
when ${\cal G}$ is the class of acyclic graphs and ${\cal H}$ is the class of outerplanar graphs (see also~\cite{GiannopoulouDT12forb,KoutsonasTY14oute}).
\medskip

\paragraph{Our results.}
In this paper we study the set ${\bf obs}({\cal A}_{k}({\cal S}))$ where  ${\cal S}$ is the class of sub-unicyclic graphs. Certainly the class ${\cal S}$ is minor-closed (while this is not the case for unicyclic graphs). It is easy to see that ${\bf obs}({\cal S})=\{2K_{3},K_{4}^{-},Z\}$, where 
$2K_{3}$ is the disjoint union of two triangles, $K_{4}^{-}$ is the complete graph on $4$ vertices minus an edge, and $Z$ the  {\em butterfly graph}, obtained by $2K_{3}$ after identifying 
two vertices of its triangles (we call the result of this identification {\em central vertex} of $Z$). 

Our first result is the identification of  ${\bf obs}({\cal A}_{1}({\cal S}))$, i.e., the minor-obstruction set of 
all $1$-apices of sub-unicyclic graphs (Section~\ref{asos}). This set contains 29 graphs 
that is the union of two sets ${\cal L}_{0}$ and ${\cal L}_{1}$, depicted in Figures~\ref{granting} and~\ref{bringing} respectively. An important 
ingredient of our proof is the notion of a {\sl nearly-biconnected graph}, that is any graph
that is either biconnected or it contains only one cut-vertex joining two blocks where one of them is a triangle. We first prove that ${\cal L}_{0}$ is the set of minor-obstructions in ${\bf obs}({\cal A}_{k}({\cal S}))$  that are not nearly-biconnected. The proof is completed by proving that the nearly-biconnected graphs in ${\bf obs}({\cal A}_{1}({\cal S}))$ 
are also minor-obstructions for 1-apex pseudoforests, i.e., members of ${\bf obs}({\cal A}_{1}({\cal P}))$. As this set is known from~\cite{Leivaditis18mino}, we can identify the remaining obstructions 
in  ${\bf obs}({\cal A}_{1}({\cal S}))$, that is the set ${\cal L}_{1}$, by exhaustive search.
 
Our second result  is an exponential  lower bound on  the size ${\bf obs}({\cal A}_{k}({\cal S}))$ (Section~\ref{opera}).
For this we completely characterize, for every $k$, the set  ${\bf obs}({\cal A}_{k}({\cal S}))\cap {\cal K}$  where ${\cal K}$ is the set of all cacti (graphs whose all blocks are either edges or cycles).
In particular, we first prove that each connected cactus obstruction in ${\bf obs}({\cal A}_{k}({\cal S}))$ 
can be obtained by identifying non-central vertices of $k+1$ butterfly graphs
and then we give a characterization of disconnected cacti in ${\bf obs}({\cal A}_{k}({\cal S}))$
in terms of obstructions in ${\bf obs}({\cal A}_{k’}({\cal S}))$ for $k’<k$ (we stress that 
here the result of Dinneen in~\cite{Dinneen97} does not apply immediately, as not all graphs in ${\bf obs}({\cal S})$ are connected). 

After identifying ${\bf obs}({\cal A}_{k}({\cal S}))\cap {\cal K}$,
the next step is to count the number of its elements (Section~\ref{stripop}). To that end, we employ the framework of the \emph{Symbolic Method} and the corresponding techniques of \emph{singularity analysis}, as they were presented in~\cite{flajolet2009analytic}. The combinatorial construction that we devise relies critically on the \emph{Dissymmetry Theorem for Trees}, by which one can move from the enumeration of rooted tree structures to unrooted ones (see~\cite{bergeron1998combinatorial} for a comprehensive account of these techniques, in the context of the \emph{Theory of Species}).
%\sed{μπλα μπλα για θεώρημα δισυμμετρίας κ.λ.π.}  Our estimation is
%\vv{must pick a symbol between $\approx$ and $\approx$. Also, maybe $\rho $ instead of $x$?}
 \begin{equation*}
|{\bf obs}({\cal A}_{k}({\cal S}))\cap {\cal K}| \sim c{\cdot k^{-5/2}}\cdot x^{k},
\end{equation*}
where $c ~{\approx}~  0.33995$ and $x  ~{\approx}~  6.27888$.
 This provides an exponential lower bound for $|{\bf obs}({\cal A}_{k}({\cal S}))|$.
 \section{Preliminaries}

 \textbf{Sets, integers, and functions.}
 We denote by $\Bbb{N}$ the set of all non-negative integers and we
 set $\Bbb{N}^+=\Bbb{N}\setminus\{0\}.$
 Given two integers $p$ and $q,$  we set $[p,q]=\{p,\ldots,q\}$
 and given a $k\in\Bbb{N}^+$ we denote $[k]=[1,k].$ Given a set $A,$ we denote by $2^{A}$ the set of all its subsets and we define 
 $\binom{A}{2}:=\{e\mid e\in 2^{A}\wedge |e|=2\}.$
If ${\cal S}$ is a collection of objects where the operation $\cup$ is defined, then we  denote $\cupall{\cal S}=\bigcup_{X\in{\cal S}}X.$

 \bigskip
 \noindent
 \textbf{Graphs}. All the graphs in this paper are finite, undirected, and without  loops or multiple edges.  Given a graph $G,$ we denote by $V(G)$ the set of vertices of  $G$ and by $E(G)$ the set of the edges of $G.$ We refer to the quantity $|V(G)|$ as the {\em size} of $G$.
 For an edge $e=\{x,y\}\in E(G),$ we use instead the notation $e=xy,$ that is equivalent to  $e=yx.$
 {Given a vertex $v \in V(G),$ we define the \emph{neighborhood} of
  $v$ as $N_G(v) = \{u \mid u \in V(G), uv \in E(G)\}$. If $X\subseteq V(G),$ then we write $N_{G}(X)=(\bigcup_{v\in X}N_{G}(v))\setminus X.$
 The {\em  degree} of a vertex $v$ in $G$ is the quantity $|N_{G}(v)|.$
 Given two graphs $G_{1},G_{2}$, we define the \emph{union} of $G_{1},G_{2}$ as
 the graph $G_{1}\cup G_{2}=(V(G_{1})\cup V(G_{2}),E(G_{1})\cup E(G_{2}))$.

 A \emph{subgraph} of a graph $G=(V,E)$ is every graph $H$
 where $V(H)\subseteq V(G)$ and $E(H)\subseteq E(G).$
 If $S \subseteq V(G),$ the subgraph of $G$ \emph{induced by} $S,$ denoted by $G[S],$ is the graph $(S, E(G) \cap \binom{S}{2}).$
 We also define $G \setminus S$ to be the subgraph of $G$ induced by $V(G) \setminus S.$
 If $S \subseteq E(G),$ we denote by $G \setminus S$ the graph $(V(G), E(G) \setminus S).$
 Given a vertex $x\in V(G)$ we define $G\setminus x=G\setminus \{x\}$ and given an edge $e\in E(G)$ we define $G\setminus e=G\setminus \{e\}.$

An edge $e\in E(G)$ is a {\em bridge} of $G$ if $G$ has less connected components than $G\setminus e$.
Given a set $S\subseteq V(G),$ we say that $S$ is a {\em  separator} of $G$ if $G$ has less connected components than $G\setminus S.$ 
Let $G$ be a graph and $S\subseteq V(G)$ and let $V_{1},\ldots,V_{q}$ be  the vertex sets of the connected components of $G\setminus S.$ We define ${\cal C}(G,S)=\{G_{1},\ldots,G_{q}\}$ where, for $i\in [q],$
$G_{i}$ is the graph obtained from $G[V_{i}\cup S]$ if we add all edges  between vertices in $S.$ %We call members of the set ${\cal C}(G,S)$ {\em augmented connected components}. 
Given a vertex $x\in V(G)$ we define ${\cal C}(G,x) = {\cal C}(G,\{x\}).$

 A vertex $v\in V(G)$ is a {\em  cut-vertex} of $G$ if $\{v\}$ is a separator of $G.$ 
 A {\em  block} of a graph $G$ is a maximal biconnected subgraph of $G$.

 By $K_{r}$ we denote the complete graph on $r$ vertices, also known as  {\em $r$-clique}. Similarly,  by $K_{r_{1},r_{2}}$ we denote the complete bipartite graph of which one part has $r_{1}$ vertices and the other $r_{2}.$ We denote by $K_{r}^{-}$ the graph obtained by $K_{r}$ after removing any edge. 
 For an $r\geq 3,$  we denote by  $C_{r}$ the cycle on $r$ vertices.
 Given a graph $G$ and an $r\geq 1$ we denote by $rG$ the graph with $r$ connected components, each isomorphic to $G.$\medskip

Given a graph  class ${\cal G}$ and a graph $G$, a vertex $v\in V(G)$ is a 
 {\em ${\cal G}$-apex of $G$} if $G\setminus v\in{\cal G}$.

 \paragraph{Minors.} We define $G / e,$ the graph obtained from the graph $G$ by {\em contracting} an edge $e = xy \in E(G),$ to be the graph obtained by replacing the edge $e$ by a new vertex $v_{e}$ which becomes adjacent to all neighbors of $x,y$ (apart from $y$ and $x$).  
 Given two graphs $H$ and $G$ we say that $H$ is a {\em  minor} of $G,$ denoted by $H\leq G,$ if $H$ can be obtained by some subgraph of $G$ after contracting edges.
 We say that $H$ is a {\em  proper minor} of $G$ if it is a minor of $G$ but is not isomorphic to $G.$ Given a set ${\cal H}$ of graphs, we write ${\cal H}\leq G$  to denote that $\exists H\in{\cal H}: H\leq G.$

  \paragraph{Sub-unicyclic Graphs.}
  We now resume  some basic concepts that we already mentioned in the introduction.
 A {\em sub-unicyclic} graph is a graph that contains at most one cycle. A graph is a {\em pseudoforest} if all its connected components are sub-unicyclic.    We denote by ${\cal S}$ (resp. ${\cal P}$)  the set of all sub-unicyclic graphs (resp. pseudoforests).
The study of the class ${\cal P}$ dates back in~\cite{Dantzig63,PicardQ82anet}. 
 Clearly, ${\cal S}\subseteq {\cal P}$ and therefore ${\cal A}_{k}({\cal S})\subseteq {\cal A}_{k}({\cal P})$, for every $k\in\Bbb{N}$.
 For simplicity, instead of saying that a graph is  1-apex sub-unicyclic/pseudoforest/acyclic 
 we just say {\em apex sub-unicyclic/pseudoforest/acyclic}.

%\aliv{Πρόσθεσα τον παρακάτω ορισμό, όπου μάλλον είχε σβηστεί.}
Given a graph $G$ and a set $S\subseteq V(G)$ we say that $S$ is an {\em apex sub-unicyclic set} (resp. {\em apex forest set}) of $G$ if $G\setminus S$ is sub-unicyclic (resp. forest). If $|S|\leq k$, for some $k\in\mathbb{N}$, then we say that $S$ a {\em $k$-apex sub-unicyclic set} (resp. {\em $k$-apex forest set}) of $G$ if $G\setminus S$ is sub-unicyclic (resp. forest).

 \section{Minor-obstructions for apex sub-unicyclic graphs}
 \label{asos}

In this section we will identify the set ${\bf obs}({\cal A}_{1}({\cal S}))$. Part of it will be the set  
${\cal L}_{0}$  containing the graphs depicted in \autoref{granting}.

 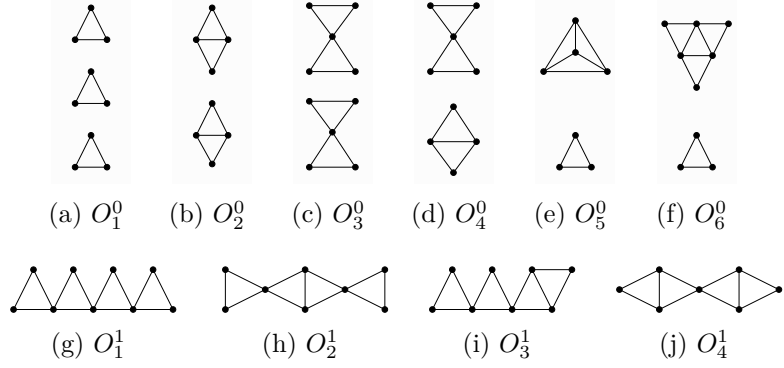
\begin{figure}[!h]
  \centering
  \begin{tabular}{ c c c c c c }
   \subcaptionbox{${O}_{1}^{0}$\label{writings}}{
    \scalebox{.53}{\begin{tikzpicture}[myNode/.style = black node, scale=0.8]
     
     \begin{scope}[yshift=-1cm]
     
     \node[myNode] (n1) at (1.5,1) {}; %1.1
     \node[myNode] (n1) at (2.5,1) {};%2.1
     \node[myNode] (n1) at (2,2) {};%3.1
     \node[myNode] (n1) at (1.5,3) {}; %1.2
     \node[myNode] (n1) at (2.5,3) {};%2.2
     \node[myNode] (n1) at (2,4) {};%3.2
     \node[myNode] (n1) at (1.5,5) {}; %1.3
     \node[myNode] (n1) at (2.5,5) {};%2.3
     \node[myNode] (n1) at (2,6) {};%3.3
     
     \draw (1.5,1) -- (2.5,1);
     \draw (2.5,1) -- (2,2);
     \draw (2,2) -- (1.5,1);
     \draw (1.5,3) -- (2.5,3);
     \draw (2.5,3) -- (2,4);
     \draw (2,4) -- (1.5,3);
     \draw (1.5,5) -- (2.5,5);
     \draw (2.5,5) -- (2,6);
     \draw (2,6) -- (1.5,5);
     
     \end{scope}
     \begin{scope}[on background layer]
     \fill[gray!2] (0.75,5.5) -- (3.25,5.5) -- (3.25,-0.5) -- (0.75,-0.5) -- cycle;
     \end{scope}
     \end{tikzpicture}}} &
   \subcaptionbox{${O}_{2}^{0}$\label{magnetic}}{
    \scalebox{.53}{\begin{tikzpicture}[myNode/.style = black node, scale=0.8]
     %Α_{2}
     \node[myNode] (n1) at (1.5,1) {}; 
     \node[myNode] (n1) at (2.5,1) {};
     \node[myNode] (n1) at (2,2) {};
     \node[myNode] (n1) at (2,0.1) {};
     
     \node[myNode] (n1) at (1.5,4) {}; 
     \node[myNode] (n1) at (2.5,4) {};
     \node[myNode] (n1) at (2,5) {};
     \node[myNode] (n1) at (2,3) {};
     
     \draw (1.5,1) -- (2.5,1);
     \draw (2.5,1) -- (2,2);
     \draw (2,2) -- (1.5,1);
     \draw (2,0.1) -- (1.5,1);
     \draw (2,0.1) -- (2.5,1);
     
     \draw (1.5,4) -- (2.5,4);
     \draw (2.5,4) -- (2,5);
     \draw (2,5) -- (1.5,4);
     \draw (2,3) -- (1.5,4);
     \draw (2,3) -- (2.5,4);
     
     \begin{scope}[on background layer]
     \fill[gray!2] (0.75,5.5) -- (3.25,5.5) -- (3.25,-0.5) -- (0.75,-0.5) -- cycle;
     \end{scope}
     
     \end{tikzpicture}}} &
   \subcaptionbox{${O}_{3}^{0}$\label{consiste}}{
    \scalebox{.53}{\begin{tikzpicture}[myNode/.style = black node, scale=0.8]
     
     \begin{scope}[xshift=0.25cm]
     
     \node[myNode] (n1) at (1,0) {}; 
     \node[myNode] (n2) [right =of n1] {};
     \coordinate (Middle) at ($(n1)!0.5!(n2)$);
     \node[myNode] (n3) [above =0.8cm of Middle] {};
     \node[myNode] (n4) [above =1.5cm of n1] {};
     \node[myNode] (n5) [above =1.5cm of n2] {};
     
     \node[myNode] (n6) at (1,3) {}; 
     \node[myNode] (n7) [right =of n6] {};
     \coordinate (Middle2) at ($(n6)!0.5!(n7)$);
     \node[myNode] (n8) [above =0.8cm of Middle2] {};
     \node[myNode] (n9) [above =1.5cm of n6] {};
     \node[myNode] (n10) [above =1.5cm of n7] {};

     \foreach \from/\to in               {n1/n2,n1/n3,n2/n3,n4/n5,n4/n3,n5/n3}
     \draw (\from) -- (\to);
     
     \foreach \from/\to in               {n6/n7,n6/n8,n7/n8,n10/n9,n10/n8,n9/n8}
     \draw (\from) -- (\to);
     \end{scope}
     
     \begin{scope}[on background layer]
     \fill[gray!2] (0.75,5.5) -- (3.25,5.5) -- (3.25,-0.5) -- (0.75,-0.5) -- cycle;
     \end{scope}
     
     \end{tikzpicture}}} &
   \subcaptionbox{${O}_{4}^{0}$\label{narcotic}}{
    \scalebox{.53}{\begin{tikzpicture}[myNode/.style = black node, scale=0.8]
     
     \begin{scope}[xshift=0.25cm, yshift=0.8cm]
     
     \node[myNode] (n1) at (1,0) {}; 
     \node[myNode] (n2) [right =of n1] {};
     \coordinate (Middle) at ($(n1)!0.5!(n2)$);
     \node[myNode] (n3) [below =0.7cm of Middle] {};
     \node[myNode] (n4) [above =0.8cm of Middle] {};

     \foreach \from/\to in               {n1/n2,n1/n3,n2/n3,n4/n1,n4/n2}
     \draw (\from) -- (\to);
     
     \node[myNode] (n6) at (1,2.2) {}; 
     \node[myNode] (n7) [right =of n6] {};
     \coordinate (Middle2) at ($(n6)!0.5!(n7)$);
     \node[myNode] (n8) [above =0.8cm of Middle2] {};
     \node[myNode] (n9) [above =1.5cm of n6] {};
     \node[myNode] (n10) [above =1.5cm of n7] {};
     
     \foreach \from/\to in               {n6/n7,n6/n8,n7/n8,n10/n9,n10/n8,n9/n8}
     \draw (\from) -- (\to);
     
     \end{scope}
     
     \begin{scope}[on background layer]
     \fill[gray!2] (0.75,5.5) -- (3.25,5.5) -- (3.25,-0.5) -- (0.75,-0.5) -- cycle;
     \end{scope}
     \end{tikzpicture}}} &
   \subcaptionbox{${O}_{5}^{0}$\label{unguided}}{
    \scalebox{.53}{\begin{tikzpicture}[myNode/.style = black node, scale=0.8]
     \begin{scope}[yshift=-3cm]
     
     \node[myNode] (n1) at (1,6) {}; 
     \node[myNode] (n1) at (3,6) {};
     \node[myNode] (n1) at (2,6.6) {};
     \node[myNode] (n1) at (2,7.6) {};
     
     \node[myNode] (n1) at (1.5,3) {}; %1.2
     \node[myNode] (n1) at (2.5,3) {};%2.2
     \node[myNode] (n1) at (2,4) {};%3.2
     
     \draw (1,6) -- (3,6);
     \draw (1,6) -- (2,6.6);
     \draw (1,6) -- (2,7.6);
     \draw (3,6) -- (2,6.6);
     \draw (3,6) -- (2,7.6);
     \draw (2,6.6) -- (2,7.6);
     
     \draw (1.5,3) -- (2.5,3);
     \draw (2.5,3) -- (2,4);
     \draw (2,4) -- (1.5,3);
     
     \end{scope}
     
     \begin{scope}[on background layer]
     \fill[gray!2] (0.75,5.5) -- (3.25,5.5) -- (3.25,-0.5) -- (0.75,-0.5) -- cycle;
     \end{scope}
     \end{tikzpicture}}} &
   \subcaptionbox{${O}_{6}^{0}$\label{consists}}{
    \scalebox{.53}{\begin{tikzpicture}[myNode/.style = black node, scale=0.8]
     \begin{scope}[yshift=-3cm]
     
     \node[myNode] (n1) at (1,7.5) {};
     \node[myNode] (n1) at (3,7.5) {};
     \node[myNode] (n1) at (2,7.5) {};
     \node[myNode] (n1) at (1.5,6.5) {};
     \node[myNode] (n1) at (2.5,6.5) {};
     \node[myNode] (n1) at (2,5.5) {};
     
     \node[myNode] (n1) at (1.5,3) {}; %1.2
     \node[myNode] (n1) at (2.5,3) {};%2.2
     \node[myNode] (n1) at (2,4) {};%3.2
     
     \draw (1.5,3) -- (2.5,3);
     \draw (2.5,3) -- (2,4);
     \draw (2,4) -- (1.5,3);
     
     \draw (1,7.5) -- (2,7.5);
     \draw (3,7.5) -- (2,7.5);
     \draw (1,7.5) -- (1.5,6.5);
     \draw (2.5,6.5) -- (3,7.5);
     \draw (2.5,6.5) -- (2,7.5);
     \draw (1.5,6.5) -- (2,7.5);
     \draw (2.5,6.5) -- (1.5,6.5);
     \draw (2.5,6.5) -- (2,5.5);
     \draw (1.5,6.5) -- (2,5.5);
     
     \end{scope}    
     \begin{scope}[on background layer]
     \fill[gray!2] (0.75,5.5) -- (3.25,5.5) -- (3.25,-0.5) -- (0.75,-0.5) -- cycle;
     \end{scope}
     
     \end{tikzpicture}}} \\
  \end{tabular}\bigskip
  
  \begin{tabular}{ c c c c c c c c c }
   
   \subcaptionbox{${O}_{1}^{1}$\label{appetite}}{
    \scalebox{.53}{\begin{tikzpicture}[myNode/.style = black node]
     \node[myNode] (n1) at (0,0) {};
     \node[myNode] (n1) at (1,0) {};
     \node[myNode] (n1) at (2,0) {};
     \node[myNode] (n1) at (3,0) {};
     \node[myNode] (n1) at (4,0) {};
     \node[myNode] (n1) at (0.5,1) {};
     \node[myNode] (n1) at (1.5,1) {};
     \node[myNode] (n1) at (2.5,1) {};
     \node[myNode] (n1) at (3.5,1) {};
     
     \draw (0,0) -- (1,0);
     \draw (1,0) -- (2,0);
     \draw (2,0) -- (3,0);
     \draw (3,0) -- (4,0);
     \draw (0,0) -- (0.5,1);
     \draw (1,0) -- (0.5,1);
     \draw (1,0) -- (1.5,1);
     \draw (2,0) -- (1.5,1);
     \draw (2,0) -- (2.5,1);   
     \draw (3,0) -- (2.5,1);
     \draw (3,0) -- (3.5,1);
     \draw (4,0) -- (3.5,1);
     \end{tikzpicture}} }&
   \subcaptionbox{${O}_{2}^{1}$\label{adorning}}{
    \scalebox{.53}{\begin{tikzpicture}[myNode/.style = black node]
     
     \node[myNode] (n1) at (1,1) {};
     \node[myNode] (n1) at (1,2) {};
     \node[myNode] (n1) at (2,1.5) {};
     \node[myNode] (n1) at (3,1) {};
     \node[myNode] (n1) at (3,2) {};
     \node[myNode] (n1) at (4,1.5) {};
     \node[myNode] (n1) at (5,1) {};
     \node[myNode] (n1) at (5,2) {};
     
     \draw (1,1) -- (1,2);
     \draw (1,1) -- (2,1.5);
     \draw (1,2) -- (2,1.5);
     \draw (2,1.5) -- (3,1);
     \draw (2,1.5) -- (3,2);
     \draw (3,1) -- (3,2);
     \draw (5,1) -- (4,1.5);
     \draw (5,2) -- (4,1.5);
     \draw (4,1.5) -- (3,1);
     \draw (4,1.5) -- (3,2);
     \draw (5,1) -- (5,2);
     
     \end{tikzpicture}}}&
   \subcaptionbox{${O}_{3}^{1}$\label{skeleton}}{
    \scalebox{.53}{\begin{tikzpicture}[myNode/.style = black node]
     \node[myNode] (n1) at (0,0) {};
     \node[myNode] (n1) at (1,0) {};
     \node[myNode] (n1) at (2,0) {};
     \node[myNode] (n1) at (3,0) {};
     \node[myNode] (n1) at (0.5,1) {};
     \node[myNode] (n1) at (1.5,1) {};
     \node[myNode] (n1) at (2.5,1) {};
     \node[myNode] (n1) at (3.5,1) {};
     
     \draw (0,0) -- (1,0);
     \draw (1,0) -- (2,0);
     \draw (2,0) -- (3,0);
     \draw (0,0) -- (0.5,1);
     \draw (1,0) -- (0.5,1);
     \draw (1,0) -- (1.5,1);
     \draw (2,0) -- (1.5,1);
     \draw (2,0) -- (2.5,1);
     \draw (3,0) -- (2.5,1);
     \draw (3,0) -- (3.5,1);
     \draw (2.5,1) -- (3.5,1);     \end{tikzpicture}}} 
  &
   
   \subcaptionbox{${O}_{4}^{1}$\label{friesian}}{
    \scalebox{.53}{\begin{tikzpicture}[myNode/.style = black node]
     
     \node[myNode] (n1) at (1,1.5) {};
     \node[myNode] (n1) at (2,2) {};
     \node[myNode] (n1) at (2,1) {};
     \node[myNode] (n1) at (3,1.5) {};
     \node[myNode] (n1) at (4,2) {};
     \node[myNode] (n1) at (4,1) {};
     \node[myNode] (n1) at (5,1.5) {};
     
     \draw (1,1.5) -- (2,1);
     \draw (1,1.5) -- (2,2);
     \draw (2,1) -- (2,2);
     \draw (2,1) -- (3,1.5);
     \draw (2,2) -- (3,1.5);
     \draw (3,1.5) -- (4,1);
     \draw (3,1.5) -- (4,2);
     \draw (4,1) -- (4,2);
     \draw (4,1) -- (5,1.5);
     \draw (4,2) -- (5,1.5);
     
     \end{tikzpicture}}}

     \\
   
  \end{tabular}
  \caption{The set ${\cal L}_{0}$  of obstructions for ${\cal A}_{1}({\cal S})$ that are not nearly-biconnected.}
  \label{granting}
 \end{figure}

 \subsection{Structure for general obstructions}
 
 We need the following lemma on the general structure of  the obstructions of ${\cal A}_{k}({\cal S})$.
\begin{lemma} 
\label{segments}
Let $G \in {\bf obs}({\cal A}_{k}({\cal S})), k\geq 0$.  Then 
the following hold:
\begin{enumerate}
\item The minimum degree of a vertex in $G$ is at least 2.
\item $G$ has no bridges.
\item All of its vertices of degree 2 have adjacent neighbors.
\end{enumerate}
\end{lemma}

\begin{proof}
	It is clear that every vertex and every edge of $G$ participates in a cycle.
	Thus, we get (1) and (2). \smallskip
	Regarding (3), suppose, to the contrary, that there exists a  vertex  $v  \in V(G)$ of degree 2 whose neighbors are no adjacent,
	and let $e \in E(G)$ be an edge incident to $v,$ i.e. $e= uv$ for some $u\in V(G).$
	As $G\in{\bf obs}({\cal A}_{k}({\cal S}))$ we have that $G' := G / e \in {\cal A}_{k}({\cal S}).$ 
	Let $S$ be a $k$-apex sub-unicyclic set of $G'$ and $v_{e}$ the vertex formed by contracting $e.$
	Observe that, every cycle in $G$ that contains $v$ also contains $u$
	and so if $v_{e}\in S$ then $(S\setminus \{v_{e}\})\cup \{u\}$ is a $k$-apex sub-unicyclic set of $G,$ a contradiction.
	Therefore, $v_{e}\notin S$ and so $S\subseteq V(G).$
	Since the neighbors of $v$ are not adjacent, the contraction of $e$ can only shorten cycles and not destroy them.
	Hence, $S$ is a $k$-apex sub-unicyclic set of $G,$ a contradiction.
\end{proof}

 \subsection{The disconnected case}

  We set ${\cal O}^{0}=\{{{O}_{1}^{0}},\ldots,{{O}_{6}^{0}}\}$. 
  We begin with an easy observation:
  
  \begin{observation}\label{pensions}
  Let $G$ be a connected graph such that ${\bf obs}({\cal S})\leq G$. Then, ${\bf obs}({\cal S})\setminus \{2K_{3}\}\leq G$.
  \end{observation}

 \begin{lemma}
  \label{presents}
  If $G\in {\bf obs}({\cal A}_{1}({\cal S}))$ and $G$ is not connected, then $G\in {\cal O}^{0}$.
 \end{lemma}

 \begin{proof}
Notice first that ${O}^{0}\subseteq {\bf obs}({\cal A}_{1}({\cal S}))$.
  Suppose, to the contrary, that there exists some disconnected 
  graph $G\in {\bf obs}({\cal A}_{1}({\cal S}))\setminus {\cal O}^{0}$.
  Note that due to  \autoref{segments} each connected component of $G$ contains
  at least one cycle and so if $G$ has more than two connected components,
  it follows that ${{O}_{1}^{0}} \leq G,$ a contradiction.
  Therefore, $G$ has exactly two connected components, namely $G_{1}, G_{2}$.

  \medskip
  
  \noindent{\em Claim 1:} One of $G_{1}, G_{2}$ is isomorphic to $K_{3}.$
  
  \noindent{\em Proof of Claim 1:}
  Suppose, towards a contradiction, that $G_{1}, G_{2}\not\in {\cal S}$.
  Then, since both $G_{1},G_{2}$ are connected, \autoref{pensions} implies that
  ${\bf obs}({\cal S})\setminus\{2K_{3}\}\leq G_{1},G_{2}$ and therefore
  $\{{{O}_{2}^{0}}, {{O}_{3}^{0}}, {{O}_{4}^{0}}\}\leq G$,
   a contradiction. Hence, one of $G_{1},G_{2}$ is sub-unicyclic
  and therefore, by \autoref{segments}, isomorphic to $K_{3}.$ Claim 1 follows.
  
  \medskip
  
  By Claim 1, we can assume, without loss of generality, that $G_{2}\cong K_{3}.$

  \medskip
  
  \noindent{\em Claim 2:} $G_{1}$ biconnected but not triconnected.
  
  \noindent{\em Proof of Claim 2:}
  If $G_{1}$ is triconnected then $K_{4}\leq G_{1},$ and therefore
%   , since $G_{2}\cong K_{3},$ we have
  ${{O}_{5}^{0}} \leq G,$ a contradiction.
  Now, suppose that there exists a cut-vertex $x$ of $G_{1}$.
  Note that, by \autoref{segments}, it follows that every $H\in {\cal C}(G_{1},x)$
  contains at least one cycle. If $x$ belongs to every cycle of $G_{1},$
  then $x$ is an ${\cal S}$-apex vertex of $G,$ which is a contradiction.
  Therefore, there exists an $H\in {\cal C}(G_{1},x)$ such that $H\setminus x$
  contains a cycle $C$ which together with a cycle in some
  $H' \in {\cal C}(G_{1},x)\setminus \{H\}$ and $G_{2}$ form
  ${{O}_{1}^{0}}$ as a minor of $G,$ a contradiction. Therefore, $G$ is biconnected. Claim 2 follows.
  
  \medskip
  
  Claim 2 implies that there exists a $2$-separator $S=\{x,y\}$ of $G_{1}$ such that
  every $H\in {\cal C}(G_{1},S)$ is a biconnected graph.
  
  \smallskip
  
  \noindent{\em Observation:} Every cycle in $G_{1}$ contains either $x$ or $y$.
  Indeed, suppose to the contrary that there exists a cycle
  $C\subseteq G_{1}$ disjoint to both $x,y$ and consider an
  $H\in{\cal C}(G_{1},S)$ such that $C\not\subseteq H$. Then,
  due to \autoref{segments}, $G[V(H)]$ contains a cycle
  and together with $C$ and $G_{2}$ form ${{O}_{1}^{0}}$
  as a minor of $G$, a contradiction.
  
  \smallskip

  \begin{figure}[!h]
  	\centering
  	\vspace{-1cm}
  	\begin{tikzpicture}[myNode/.style = black node, scale=0.7]
  	\begin{scope}
  	
  	\node[myNode, label=above:$x$] (A) at (0,2) {};
  	\node[myNode, label=below:$y$] (B) at (0,0) {};
  	
  	\begin{scope}[on background layer]
  	\filldraw[fill=green!70!blue,  opacity=0.7, draw=white]
  	(B.center) to [bend right =110, min distance =5cm] (A.center) to [bend left=20] (B.center);
  	\end{scope}  
  	
  	\begin{scope}[on background layer]
  	\filldraw[fill=mustard,  opacity=0.7, draw=white]
  	(B.center) to [bend left=110, min distance =5cm] (A.center) to [bend right=20] (B.center);
  	
  	\draw[-, name path=path1] (A.center) .. controls  ++(180:3) and ++(240:3) .. (A.center);
  	\draw[-, name path=path2] (B.center) .. controls  ++(180:3) and ++(120:3) .. (B.center);
  	\draw [name intersections={of= path1 and path2}]
  	\foreach \s in {1,2}{
  		(intersection-\s) node[myNode] {}
  	};
  	
  	\node[label=above:$C_{1}$] () at (intersection-2) {};
  	\node[label=below:$C_{2}$] () at (intersection-2) {};
  	
  	\end{scope}
  	
  	\node[label= right:$H'$] () at (2,1) {};
  	\node[label= left:$H$] () at (-2,1) {};
  	
  	%   \draw[very thin] ($(A)!0.5!(B)$) ellipse (0.2 and 1.2);
  	
  	\end{scope}
  	
  	\begin{scope}[xshift=8cm]
  	
  	\node[myNode, label=above:$x$] (A) at (0,2) {};
  	\node[myNode, label=below:$y$] (B) at (0,0) {};
  	
  	\begin{scope}[on background layer]
  	\filldraw[fill=green!70!blue,  opacity=0.7, draw=white]
  	(B.center) to [bend right =110, min distance =5cm] (A.center) to [bend left=20] (B.center);
  	\end{scope}  
  	
  	\begin{scope}[on background layer]
  	\filldraw[fill=mustard,  opacity=0.7, draw=white]
  	(B.center) to [bend left=110, min distance =5cm] (A.center) to [bend right=20] (B.center);
  	
  	\node[myNode] (u) at (-1.5,1) {};
  	
  	\draw[-] (A.center) to [out=180, in=180, min distance=1.3cm] (u.center) to [out=0,in=240] (A.center);
  	\draw[-] (B.center) to [out=180, in=180, min distance=1.3cm] (u.center) to [out=0,in=120] (B.center);
  	
  	\node[label=above right:$C_{1}$] () at (u.center) {};
  	\node[label=below right:$C_{2}$] () at (u) {};
  	
  	\end{scope}
  	
  	\node[label= right:$H'$] () at (2,1) {};
  	\node[label= left:$H$] () at (-2,1) {};

  	%   \draw[very thin] ($(A)!0.5!(B)$) ellipse (0.2 and 1.2);

  	\end{scope}
  	\end{tikzpicture}
  	\vspace{-.7cm}
  	\caption{The cycles $C_{1}, C_{2}$ in the last part of the proof of \autoref{presents}}\label{teaching}
  \end{figure}
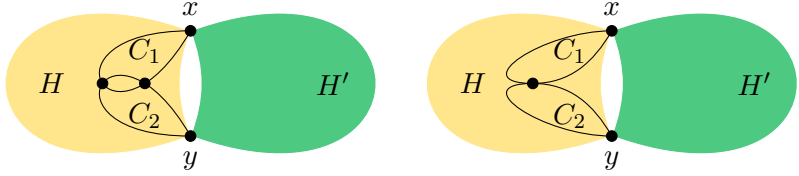

  Since $x,y$ are not ${\cal S}$-apex vertices of $G$, then, apart from $G_{2}$, there exist
  two cycles $C_{1},C_{2}$ in $G_{1}$ such that $y\notin V(C_{1})$ and
  $x\notin V(C_{2})$. The above Observation implies that $x\in V(C_{1})$
  and $y\in V(C_{2})$. Due to ${{O}_{1}^{0}}$-freeness of $G$,
  we have that $V(C_{1})\cap V(C_{2})\neq \emptyset$ and therefore there exists an
  $H\in {\cal C}(G_{1},S)$ such that $C_{1}\cup C_{2}\subseteq H$.
  Consider, now, an $H'\in {\cal C}(G_{1},S)$ different from $H$ and observe
  that by \autoref{segments}, $G[V(H')]$ contains a cycle.
  Then, if $C_{1},C_{2}$ share more than one vertex,
  ${{O}_{5}^{0}}\leq G$, while if they share only one vertex,
  ${{O}_{6}^{0}}\leq G$, a contradiction in both cases (see \autoref{teaching}). Lemma follows.
 \end{proof}

  \subsection{The connected cases}
  
 \begin{lemma}
  \label{immature}
  If $G$ is a connected graph in ${\bf obs}({\cal A}_{1}({\cal S}))$, with  at least three cut-vertices,  then $G\cong {{O}_{1}^{1}}$.
 \end{lemma}
 
 \begin{proof}
 Consider a connected graph $G\in{\bf obs}({\cal A}_{1}({\cal S}))$ with  at least three cut-vertices.
 We first exclude the case where there is a block $B$ of $G$ containing three cut-vertices $x,y,z.$ Indeed,  due to \autoref{segments},
  each block of $G$ contains a cycle and this holds for $B$ and the blocks of $G$ that share a cut-vertex with $B.$
 % \lsst{Δεν μας νοιάζει ότι το $B$ περιέχει κύκλο. Πειράζει που το γράφουμε;}
  This implies the existence of ${{O}_{1}^{0}}$ as a proper minor of $G,$ a contradiction as ${{O}_{1}^{0}}\in  {\bf obs}({\cal A}_{1}({\cal S}))$.
   \smallskip 
  
We just proved that $G$ contains 4 blocks $B_{1},$ $B_{2},$ $B_{3},$ and $B_{4}$ such that
  $V(B_{1})\cap V(B_{2})=\{x\}$, $V(B_{2})\cap V(B_{3})=\{y\},$ $V(B_{3})\cap V(B_{4})=\{z\}$  are singletons each consisting of a cut-vertex. 
  In this case, again by \autoref{segments}, each block in $\{B_{1},B_{2},B_{3},B_{4}\}$ contains a cycle, which implies that ${{O}_{1}^{1}}\leq G$. As ${{O}_{1}^{1}}\in{{\bf obs}}({\cal A}_{1}({\cal S}))$, it follows that $G\cong {{O}_{1}^{1}}$.
 \end{proof}

 \begin{lemma}
  \label{daylight}
    If $G$ is a connected graph in ${\bf obs}({\cal A}_{1}({\cal S}))$ with exactly two cut-vertices,  then $G
    \in \{ {{O}_{2}^{1}},{{O}_{3}^{1}}\}$.
 \end{lemma}
 
 \begin{proof} 	
Observe first that $\{{{O}_{2}^{1}},{{O}_{3}^{1}}\}\subseteq {\bf obs}({\cal A}_{1}({\cal S}))$.%\gstam{Αυτό νομίζω δεν χρειάζεται.}
 Suppose, to the contrary, that there is a graph $G\in {\bf obs}({\cal A}_{1}({\cal S}))\setminus\{ {{O}_{2}^{1}},{{O}_{3}^{1}}\}$ that has  exactly two cut-vertices, namely $u_{1}$ and $u_{2}$.
  Let $B$ be the (unique) block containing the two cut-vertices $u_{1},u_{2}$ and let 
  $${H}_{1} = \cupall \{ H \in {\cal C}(G,u_{1}) : u_{2} \notin V(H) \}\mbox{~~and~~}{H}_{2} = \cupall \{ H \in {\cal C}(G,u_{2}) : u_{1} \notin V(H) \}.$$ 
 Keep in mind that $G$ cannot contain any graph in ${\cal O}^{0}$  as a minor because, due to the connectivity of $G$,
it would contain it as a proper minor, a contradiction to the fact that ${\cal O}^{0}\subseteq {\bf obs}({\cal A}_{1}({\cal S}))$.

  We  prove a series of claims:
  
  \medskip
  \noindent{\em Claim 1:} Every cycle in $G$ contains either $u_{1}$ or $u_{2}.$
  
  \noindent{\em Proof of Claim 1:}
  Suppose, to the contrary, that there exists a cycle $C$ not containing any of the cut-vertices. We distinguish two cases:
  
  \smallskip
  
  \noindent {\em Case 1:} $C$ is in $B.$
  Note that, due to  \autoref{segments}, each block contains a cycle.
  Therefore the cycle $C$ along with two more cycles, one from  $H_{1}$ and one from  $H_{2},$ form ${{O}_{1}^{0}}$ as a proper minor of $G,$ a contradiction.
  
  \smallskip
  
  \noindent {\em Case 2:} $C$ is either in $H_{1}$ or $H_{2}.$
  Suppose, without loss of generality, that $C$ is in some block of $H_{1}.$ Then, due to Menger's theorem, there exist two paths from $C$ to $u_{1}$ that intersect only in $u_{1}.$
  Since each block of $G$ contains at least one cycle, we have that $C$ together with the aforementioned paths,
  the block $B$ and any block in $H_{2},$ form ${{O}_{3}^{1}}$ as minor of $G,$ a contradiction (See \autoref{dispense}). Claim 1 follows.

  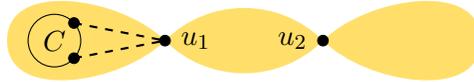
\begin{figure}[!h]
   \centering
   \begin{tikzpicture}[myNode/.style = black node, scale=.7]
   
   \node[myNode, label=right:$u_1$] (u1) at (0,1) {};
   \coordinate (u1t) at (-3,1) {};
   \node[myNode, label=left:$u_2$]  (u2) at (3,1) {};
   \coordinate (u2t) at (6,1) {};

   \begin{scope}[on background layer]
   \fill[mustard!90] (u1.center) to [out=135, in=90] (u1t.center) to [out=270, in=225] (u1.center);
   \path[-, name path= path1] (u1.center).. controls ++(150:3) and ++(210:3) .. (u1.center);
   \node (C) at ($(u1t)!0.3!(u1)$) {$C$};
   \draw[-, name path= path2] (C) circle (0.5cm);
   \draw [name intersections={of= path1 and path2}][dashed, thick]
   \foreach \s in {1,2}{
    (intersection-\s) -- (u1.center)
   };
   \draw [name intersections={of= path1 and path2}]
   \foreach \s in {1,2}{
    (intersection-\s) node[myNode] {}
   };
   \end{scope}
   
   \begin{scope}[on background layer]
   \fill[mustard!90] (u1.center) to [bend left = 45] (u2.center) to [bend left=45] (u1.center);
   \end{scope}
   
   \begin{scope}[on background layer]
   \fill[mustard!90] (u2.center) to [out=45, in=90] (u2t.center) to [out=270, in=315] (u2.center);
   \end{scope}

   \end{tikzpicture}
   \caption{The cycle $C$ in Case 2 of the proof of Claim 1.}\label{dispense}
  \end{figure}

  \noindent{\em Claim 2:} Both $H_{1}$ and $H_{2}$ are isomorphic to $K_{3}.$
  
  \noindent{\em Proof of Claim 2:}
  Suppose, towards a contradiction, that one of $H_{1}, H_{2}$, say $H_{1}$,
  is not sub-unicyclic (we will use \autoref{segments}). Since $H_{1}$ is connected, then, by
  \autoref{pensions}, we have that ${\bf obs}({\cal S})\setminus\{2K_{3}\} \leq H_{1}.$
  Also, since $u_{1}$ is not an {${\cal S}$-apex} vertex of $G$ and since,
  due to Claim 1, all cycles of $H_{1}$ contain $u_{1}$,
  then $G\setminus V(H_{1})\not\in {\cal S}$. Now, since
  $G\setminus V(H_{1})$ is connected then \autoref{pensions} implies
  that ${\bf obs}({\cal S})\setminus\{2K_{3}\} \leq G\setminus V(H_{1}).$
  Hence, $\{{{O}_{2}^{0}}, {{O}_{3}^{0}},
  {{O}_{4}^{0}}\}\leq G$, a contradiction. Therefore 
  $H_{1}, H_{2}\in {\cal S}$ and, by \autoref{segments}, Claim 2 follows.
  
  \medskip

  Since $u_{1}$ is not a ${\cal S}$-apex of $G$, then Claim 2 implies that (apart from $H_{2}$) there exists a cycle $C_{2}$ in $B\setminus u_{1}$,
  which by Claim 1, contains $u_{2}$. The same holds for $u_{2}$, i.e.
  there exists a cycle $C_{1}$ in $B\setminus u_{2}$ that contains $u_{1}$.
   Then ${{O}_{3}^{0}}\leq G,$ if $C_{1}, C_{2}$ are disjoint,
   ${{O}_{1}^{1}}\leq G,$ if $C_{1}$ and $C_{2}$ share exactly one vertex,
  and ${{O}_{2}^{1}}\leq G,$ if $C_{1}, C_{2}$ share at least $2$ vertices,
  a contradiction in all cases (see \autoref{connects}). Lemma follows.
  
  \begin{figure}[!h]
  	\centering
  	\vspace{-0.5cm}
  	\begin{tikzpicture}[myNode/.style = black node, scale=.7]
  	
  	\begin{scope}
  	\node[myNode, label=above:$u_1$] (u1) at (0,0) {};
  	\node[myNode, label=above:$u_2$]  (u2) at (4,0) {};
  	\draw[-] (u1) to (150:1) node[myNode] {} to (210:1) node[myNode] {} to (u1);
  	\begin{scope}[xscale=-1, xshift=-4cm]
  	\draw[-] (0,0) to (150:1) node[myNode] {} to (210:1) node[myNode] {} to (0,0);
  	\end{scope}
  	
  	\draw[-] (u1.center) .. controls ++(45:3) and ++(-45:3).. (u1.center);
  	\draw[-] (u2.center) .. controls ++(135:3) and ++(-135:3).. (u2.center);
  	
  	\begin{scope}[on background layer]
  	\fill[mustard!90] (u1.center) to [bend left = 60] (u2.center) to [bend left=60] (u1.center);
  	\end{scope}
  	
  	\end{scope}
  	
  	\begin{scope}[xshift=7cm]
  	\node[myNode, label=above:$u_1$] (u1) at (0,0) {};
  	\node[myNode, label=above:$u_2$]  (u2) at (4,0) {};
  	\draw[-] (u1) to (150:1) node[myNode] {} to (210:1) node[myNode] {} to (u1);
  	\begin{scope}[xscale=-1, xshift=-4cm]
  	\draw[-] (0,0) to (150:1) node[myNode] {} to (210:1) node[myNode] {} to (0,0);
  	\end{scope}
  	
  	\node[myNode] (M) at (2,0) {};
  	
  	\draw[-] (u1.center) to [out=20, in=90] (M.center) to [out=-90, in=-20] (u1.center);
  	\draw[-] (u2.center) to [out=160, in=90] (M.center) to [out=-90, in=-160] (u2.center);
  	
  	\begin{scope}[on background layer]
  	\fill[mustard!90] (u1.center) to [bend left = 60] (u2.center) to [bend left=60] (u1.center);
  	\end{scope}
  	
  	\end{scope}
  	
  	\begin{scope}[xshift=14cm]
  	\node[myNode, label=above:$u_1$] (u1) at (0,0) {};
  	\node[myNode, label=above:$u_2$]  (u2) at (4,0) {};
  	\draw[-] (u1) to (150:1) node[myNode] {} to (210:1) node[myNode] {} to (u1);
  	\begin{scope}[xscale=-1, xshift=-4cm]
  	\draw[-] (0,0) to (150:1) node[myNode] {} to (210:1) node[myNode] {} to (0,0);
  	\end{scope}
  	
  	\draw[-, name path=path1] (u1.center) .. controls ++(30:3.5) and ++(-30:3.5).. (u1.center);
  	\draw[-, name path=path2] (u2.center) .. controls ++(150:3.5) and ++(-150:3.5).. (u2.center);
  	
  	\draw [name intersections={of= path1 and path2}]
  	\foreach \s in {1,2}{
  		(intersection-\s) node[myNode] {}
  	};
  	
  	\begin{scope}[on background layer]
  	\fill[mustard!90] (u1.center) to [bend left = 60] (u2.center) to [bend left=60] (u1.center);
  	\end{scope}
  	
  	\end{scope}
  	\end{tikzpicture}
  	\caption{The ways the cycles $C_{1}, C_{2}$ may intersect in the last part of the proof of \autoref{daylight}.}\label{connects}
  \end{figure}
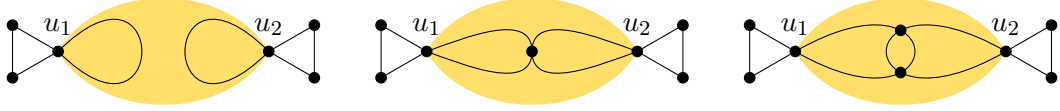
  
  \end{proof}

\paragraph{Nearly-biconnected graphs.}We say that a graph $G$ is {\em nearly-biconnected} if  it is either biconnected or it contains exactly one cut-vertex $x$ 
and ${\cal C}(G,x)=\{H,K_{3}\}$ where $H$ is a biconnected graph.

 \begin{lemma}\label{handling}
  Let $G\in {\bf obs}({\cal A}_{1}({\cal S}))$ be a connected graph that contains exactly one cut-vertex. Then either $G\cong{{O}_{4}^{1}}$  or  $G$ is nearly-biconnected.
 \end{lemma}

 \begin{proof}
 	 Notice that ${{O}_{4}^{1}}\in {\bf obs}({\cal A}_{1}({\cal S}))$.
 	 Let $G$ be a graph in ${\bf obs}({\cal A}_{1}({\cal S}))\setminus\{ {{O}_{4}^{1}}\}$. As in the the proof of \autoref{daylight}, keep in mind that  $G$ does not contain any graph in ${\cal O}^{0}$ as a minor.

 We first prove the following claim.
  
  \medskip
  
  \noindent{\em Claim 1:} There exists a unique component in ${\cal C}(G,x)$ which contains a cycle disjoint from $x.$\smallskip
  
  \noindent{\em Proof of Claim 1:}
  First, we easily observe that there exists such a component in ${\cal C}(G,x).$
  Suppose that there exist two different components in ${\cal C}(G,x),$
  each of which contains a cycle disjoint to $x.$
  Then, due to Menger's theorem, for each of said cycles there exist two paths from $x$ to the cycle being considered, intersecting only in $x.$
  But then ${{O}_{4}^{1}}$ is formed as a minor of $G,$ a contradiction (see \autoref{softened}). Claim follows.

  \begin{figure}[!h]
   \centering
   \vspace{-.5cm}
   \begin{tikzpicture}[myNode/.style = black node, scale=.7]
   
   \node[myNode, label=above:$x$] (u1) at (0,1) {};
   \coordinate (u1t) at (-3,1) {};
   \coordinate (u1s) at (3,1) {};

   \begin{scope}[on background layer]
   \fill[mustard!90] (u1.center) to [out=135, in=90] (u1t.center) to [out=270, in=225] (u1.center);
   \path[-, name path= path1] (u1.center).. controls ++(150:3) and ++(210:3) .. (u1.center);
   \node (C) at ($(u1t)!0.3!(u1)$) {};
   \draw[-, name path= path2] (C) circle (0.5cm);
   \draw [name intersections={of= path1 and path2}][dashed, thick]
   \foreach \s in {1,2}{
    (intersection-\s) -- (u1.center)
   };
   \draw [name intersections={of= path1 and path2}]
   \foreach \s in {1,2}{
    (intersection-\s) node[myNode] {}
   };
   \end{scope}
   
   \begin{scope}[on background layer]
   \fill[mustard!90] (u1.center) to [out=45, in=90] (u1s.center) to [out=270, in=315] (u1.center);
   \path[-, name path= path1] (u1.center).. controls ++(30:3) and ++(-30:3) .. (u1.center);
   \node (C2) at ($(u1s)!0.3!(u1)$) {};
   \draw[-, name path= path2] (C2) circle (0.5cm);
   \draw [name intersections={of= path1 and path2}][dashed, thick]
   \foreach \s in {1,2}{
    (intersection-\s) -- (u1.center)
   };
   \draw [name intersections={of= path1 and path2}]
   \foreach \s in {1,2}{
    (intersection-\s) node[myNode] {}
   };
   \end{scope}
   
   \end{tikzpicture}
   \caption{The configuration in the proof of Claim 1.}\label{softened}
  \end{figure}
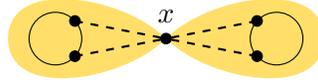
  
  Let $H\in {\cal C}(G,x)$ be the unique, by Claim 1, component in  ${\cal C}(G,x)$ which contains a cycle 
  disjoint to $x.$ Also, let $D = \cupall \{ H^{'} \in {\cal C}(G,x) :H^{'}\not= H \}.$
  
  \medskip
  
  \noindent{\em Claim 2:} $D\cong K_{3}$
  
  \noindent{\em Proof of Claim 2:}
  We argue that $D\in {\cal S},$ which, due to \autoref{segments}, implies that $D\cong K_{3}.$
  Suppose, to the contrary, that $D\not\in {\cal S}.$ Then, since $D$ is connected, \autoref{pensions} implies that ${\bf obs}({\cal S})\setminus\{2K_{3}\} \leq D.$
  By Claim 1, every cycle in $D$ contains $x$ and therefore, since $x$ is not an {${\cal S}$-apex} vertex of $G,$ $H\setminus x$ contains at least two cycles. Then, taking into account the connectivity of $H\setminus x$, \autoref{pensions} implies that ${\bf obs}({\cal S})\setminus\{2K_{3}\} \leq H\setminus x$. 
  Hence, $\{{{O}_{2}^{0}}, {{O}_{3}^{0}}, {{O}_{4}^{0}}\}\leq G,$ a contradiction. Claim 2 follows.
  
  \medskip
  
Claim 2 implies that $G$ is nearly-biconnected, as required.
 \end{proof}

\subsection{Borrowing obstructions from apex-presudoforests}
%\aliv{Είναι ο τίτλος της ενότητας καλός?}

We need the following fact:

\begin{fact}
\label{begrudge}
The graphs in ${\bf obs}({\cal A}_{1}({\cal P}))$ that are nearly-biconnected and
belong in  ${\bf obs}({\cal A}_{1}({\cal S}))$ are the graphs in Figure~\ref{bringing}.
\end{fact}

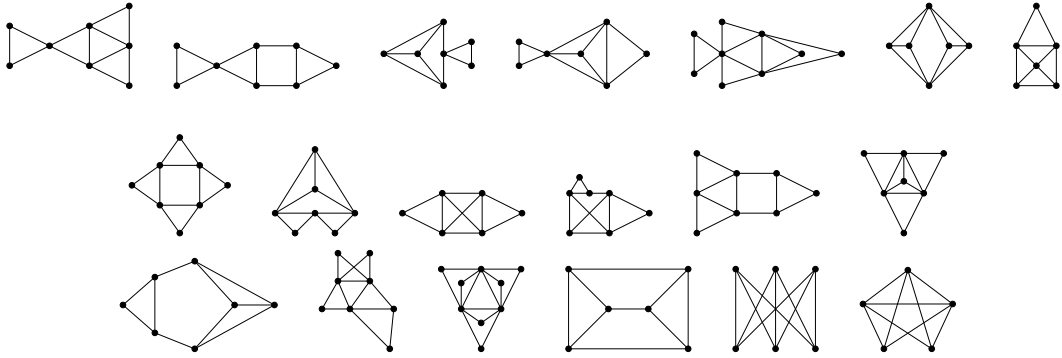
\begin{figure}[!h]
	\centering

  \begin{tabular}{ c c c c c c c }
   %   \subcaptionbox{${O}_{6}^{1}$\label{OS_61}}
   {
    \scalebox{.53}{\begin{tikzpicture}[myNode/.style = black node]
     \node[myNode] (n1) at (1,1) {};
     \node[myNode] (n1) at (1,2) {};
     \node[myNode] (n1) at (2,1.5) {};
     \node[myNode] (n1) at (3,1) {};
     \node[myNode] (n1) at (3,2) {};
     \node[myNode] (n1) at (4,1.5) {};
     \node[myNode] (n1) at (4,0.5) {};
     \node[myNode] (n1) at (4,2.5) {};
     
     \draw (1,1) -- (1,2);
     \draw (1,1) -- (2,1.5);
     \draw (1,2) -- (2,1.5);
     \draw (2,1.5) -- (3,1);
     \draw (2,1.5) -- (3,2);
     \draw (3,1) -- (3,2);
     \draw (4,0.5) -- (4,1.5);
     \draw (4,2.5) -- (4,1.5);
     \draw (4,1.5) -- (3,1);
     \draw (4,1.5) -- (3,2);
     \draw (4,0.5) -- (3,1);
     \draw (4,2.5) -- (3,2);
     
     \end{tikzpicture}}} &
%   \subcaptionbox{${O}_{2}^{1}$\label{adorning}}
   {
    \scalebox{.53}{\begin{tikzpicture}[myNode/.style = black node]
     
     \node[myNode] (n1) at (1,1) {};
     \node[myNode] (n1) at (1,2) {};
     \node[myNode] (n1) at (2,1.5) {};
     \node[myNode] (n1) at (3,1) {};
     \node[myNode] (n1) at (4,1) {};
     \node[myNode] (n1) at (5,1.5) {};
     \node[myNode] (n1) at (3,2) {};
     \node[myNode] (n1) at (4,2) {};

     \draw (1,1) -- (1,2);
     \draw (1,1) -- (2,1.5);
     \draw (1,2) -- (2,1.5);
     \draw (2,1.5) -- (3,1);
     \draw (2,1.5) -- (3,2);
     \draw (3,1) -- (4,1);
     \draw (3,2) -- (4,2);
     \draw (4,1) -- (5,1.5);
     \draw (4,2) -- (5,1.5);
     \draw (3,1) -- (3,2);
     \draw (4,1) -- (4,2);
     
     \end{tikzpicture}}}&
    %  \subcaptionbox{${O}_{7}^{1}$\label{OS_71}}
 {
    \scalebox{.53}{\begin{tikzpicture}[myNode/.style = black node]
     \node[myNode] (n1) at (3.2,1.2) {};
     \node[myNode] (n1) at (3.2,1.8) {};
     \node[myNode] (n1) at (1,1.5) {};
     \node[myNode] (n1) at (1.85,1.5) {};
     \node[myNode] (n1) at (2.5,0.7) {};
     \node[myNode] (n1) at (2.5,2.3) {};
     \node[myNode] (n1) at (2.5,1.5) {};
     
     \draw (3.2,1.2) -- (3.2,1.8);
     \draw (3.2,1.2) -- (2.5,1.5);
     \draw (3.2,1.8) -- (2.5,1.5);
     \draw (1,1.5) -- (1.85,1.5) ;
     \draw (1,1.5) -- (2.5,0.7);
     \draw (1,1.5) -- (2.5,2.3);
     \draw (1.85,1.5) -- (2.5,0.7);
     \draw (1.85,1.5) -- (2.5,2.3);
     \draw (2.5,0.7) -- (2.5,2.3);
     
     \end{tikzpicture}}} &
 %  \subcaptionbox{${O}_{8}^{1}$\label{OS_81}}
 {
    \scalebox{.53}{\begin{tikzpicture}[myNode/.style = black node]
     \node[myNode] (n1) at (0.3,1.2) {};
     \node[myNode] (n1) at (0.3,1.8) {};
     \node[myNode] (n1) at (1,1.5) {};
     \node[myNode] (n1) at (1.85,1.5) {};
     \node[myNode] (n1) at (2.5,0.7) {};
     \node[myNode] (n1) at (2.5,2.3) {};
     \node[myNode] (n1) at (3.5,1.5) {};
     
     \draw (0.3,1.2) -- (0.3,1.8);
     \draw (0.3,1.2) -- (1,1.5);
     \draw (0.3,1.8) -- (1,1.5);
     \draw (1,1.5) -- (1.85,1.5) ;
     \draw (1,1.5) -- (2.5,0.7);
     \draw (1,1.5) -- (2.5,2.3);
     \draw (1.85,1.5) -- (2.5,0.7);
     \draw (1.85,1.5) -- (2.5,2.3);
     \draw (2.5,0.7) -- (2.5,2.3);
     \draw (2.5,0.7) -- (3.5,1.5);
     \draw (2.5,2.3) -- (3.5,1.5);
     \end{tikzpicture}}} &
 %  \subcaptionbox{${O}_{9}^{1}$\label{OS_91}}
 {
    \scalebox{.53}{\begin{tikzpicture}[myNode/.style = black node]
     
     \node[myNode] (n1) at (1.3,1) {};
     \node[myNode] (n1) at (1.3,2) {};
     \node[myNode] (n1) at (2,1.5) {};
     \node[myNode] (n1) at (3,1) {};
     \node[myNode] (n1) at (3,2) {};
     \node[myNode] (n1) at (4,1.5) {};
     \node[myNode] (n1) at (2,2.3) {};
     \node[myNode] (n1) at (2,0.7) {};
     \node[myNode] (n1) at (5,1.5) {};
     
     \draw (1.3,1) -- (1.3,2);
     \draw (1.3,1) -- (2,1.5);
     \draw (1.3,2) -- (2,1.5);
     \draw (2,1.5) -- (3,1);
     \draw (2,1.5) -- (3,2);
     \draw (3,1) -- (3,2);
     \draw (4,1.5) -- (3,1);
     \draw (4,1.5) -- (3,2);
     \draw (2,0.7) -- (3,1);
     \draw (2,2.3) -- (3,2);
     \draw (2,2.3) -- (2,1.5);
     \draw (2,0.7) -- (2,1.5);
     \draw (3,1) -- (5,1.5);
     \draw (3,2) -- (5,1.5);
     
     \end{tikzpicture}}} &
 %\subcaptionbox{${O}_{1}^{2}$\label{OS_12}}
 {
 	\scalebox{.53}{\begin{tikzpicture}[myNode/.style = black node]
 		
 		%D1
 		\node[myNode] (n1) at (2,3) {};
 		\node[myNode] (n1) at (1,2) {};
 		\node[myNode] (n1) at (3,2) {};
 		\node[myNode] (n1) at (2,1) {};
 		\node[myNode] (n1) at (1.5,2) {};
 		\node[myNode] (n1) at (2.5,2) {};
 		
 		\draw (2,3) -- (1,2);
 		\draw (2,3) -- (3,2);
 		\draw (2,3) -- (1.5,2);
 		\draw (2,3) -- (2.5,2);
 		\draw (2,1) -- (1,2);
 		\draw (2,1) -- (3,2);
 		\draw (2,1) -- (1.5,2);
 		\draw (2,1) -- (2.5,2);
 		\draw (1,2) -- (1.5,2);
 		\draw (2.5,2) -- (3,2);
 		\end{tikzpicture}}} &
 %  \subcaptionbox{${O}_{2}^{2}$\label{OS_22}}
 {
 	\scalebox{.53}{\begin{tikzpicture}[myNode/.style = black node]
 		
 		\node[myNode] (n1) at (1.5,3) {};
 		\node[myNode] (n1) at (1,2) {};
 		\node[myNode] (n1) at (2,2) {};
 		\node[myNode] (n1) at (1.5,1.5) {};
 		\node[myNode] (n1) at (1,1) {};
 		\node[myNode] (n1) at (2,1) {};
 		
 		\draw (1.5,3) -- (1,2);
 		\draw (1.5,3) -- (2,2);
 		\draw (1,2) -- (2,2);
 		\draw (1,2) -- (1,1);
 		\draw (1,2) -- (2,1);
 		\draw (2,2) -- (1,1);
 		\draw (2,2) -- (2,1);
 		\draw (1,1) -- (2,1);
 		\end{tikzpicture}}} \\
  \end{tabular}
  \bigskip
  
    \centering
  \begin{tabular}{ c c c c c c c c c c c c c c c c c }
  
   %\subcaptionbox{${O}_{3}^{2}$\label{OS_32}}
   {
    \scalebox{.53}{\begin{tikzpicture}[myNode/.style = black node]
     \node[myNode] (n1) at (1.5,2.7) {};
     \node[myNode] (n1) at (1,2) {};
     \node[myNode] (n1) at (2,2) {};
     \node[myNode] (n1) at (1.5,0.3) {};
     \node[myNode] (n1) at (1,1) {};
     \node[myNode] (n1) at (2,1) {};
     \node[myNode] (n1) at (0.3,1.5) {};
     \node[myNode] (n1) at (2.7,1.5) {};
     
     \draw (1.5,2.7) -- (1,2);
     \draw (1.5,2.7) -- (2,2);
     \draw (1.5,0.3) -- (1,1);
     \draw (1.5,0.3) -- (2,1);
     \draw (0.3,1.5) -- (1,1);
     \draw (0.3,1.5) -- (1,2);
     \draw (2.7,1.5) -- (2,1);
     \draw (2.7,1.5) -- (2,2);
     \draw (1,2) -- (2,2);
     \draw (1,2) -- (1,1);
     \draw (2,2) -- (2,1);
     \draw (1,1) -- (2,1);
     \end{tikzpicture}}} &
   %\subcaptionbox{${O}_{4}^{2}$\label{OS_42}}
   %
   {
    \scalebox{.53}{\begin{tikzpicture}[myNode/.style = black node]
     
     %D4
     \node[myNode] (n1) at (1,5) {}; 
     \node[myNode] (n1) at (3,5) {};
     \node[myNode] (n1) at (2,5.6) {};
     \node[myNode] (n1) at (2,6.6) {};
     \node[myNode] (n1) at (2,5) {}; 
     \node[myNode] (n1) at (1.5,4.5) {};
     \node[myNode] (n1) at (2.5,4.5) {};
     
     \draw (1,5) -- (3,5);
     \draw (1,5) -- (2,5.6);
     \draw (1,5) -- (2,6.6);
     \draw (3,5) -- (2,5.6);
     \draw (3,5) -- (2,6.6);
     \draw (2,5.6) -- (2,6.6);
     \draw (1.5,4.5) -- (1,5); 
     \draw (1.5,4.5) -- (2,5);
     \draw (2.5,4.5) -- (2,5);
     \draw (2.5,4.5) -- (3,5);
     
     \end{tikzpicture}}} &
   %   \subcaptionbox{${O}_{5}^{2}$\label{OS_52}}
   %
   {
    \scalebox{.53}{\begin{tikzpicture}[myNode/.style = black node]
     
     %D5
     \node[myNode] (n1) at (2,1.5) {};
     \node[myNode] (n1) at (3,1) {};
     \node[myNode] (n1) at (4,1) {};
     \node[myNode] (n1) at (5,1.5) {};
     \node[myNode] (n1) at (3,2) {};
     \node[myNode] (n1) at (4,2) {};
     
     \draw (2,1.5) -- (3,1);
     \draw (2,1.5) -- (3,2);
     \draw (3,1) -- (4,1);
     \draw (3,2) -- (4,2);
     \draw (4,1) -- (5,1.5);
     \draw (4,2) -- (5,1.5);
     \draw (3,1) -- (3,2);
     \draw (4,1) -- (4,2);
     \draw (3,1) -- (4,2);
     \draw (3,2) -- (4,1);
     \end{tikzpicture}}} &
   %\subcaptionbox{${O}_{6}^{2}$\label{OS_62}}
   {
    \scalebox{.53}{\begin{tikzpicture}[myNode/.style = black node]
     
     %D6
     \node[myNode] (n1) at (3.5,2) {};
     \node[myNode] (n1) at (3.25,2.4) {};
     \node[myNode] (n1) at (3,1) {};
     \node[myNode] (n1) at (4,1) {};
     \node[myNode] (n1) at (5,1.5) {};
     \node[myNode] (n1) at (3,2) {};
     \node[myNode] (n1) at (4,2) {};
     
     \draw (3,1) -- (4,1);
     \draw (3,2) -- (4,2);
     \draw (4,1) -- (5,1.5);
     \draw (4,2) -- (5,1.5);
     \draw (3,1) -- (3,2);
     \draw (4,1) -- (4,2);
     \draw (3,1) -- (4,2);
     \draw (3,2) -- (4,1);
     \draw (3.25,2.4) -- (3,2);
     \draw (3.25,2.4) -- (3.5,2);
     
     \end{tikzpicture}}} &
   %\subcaptionbox{${O}_{7}^{2}$\label{OS_72}}
   {
    \scalebox{.53}{\begin{tikzpicture}[myNode/.style = black node]
     
     \node[myNode] (n1) at (2,0.5) {};
     \node[myNode] (n1) at (2,2.5) {};
     \node[myNode] (n1) at (2,1.5) {};
     \node[myNode] (n1) at (3,1) {};
     \node[myNode] (n1) at (4,1) {};
     \node[myNode] (n1) at (5,1.5) {};
     \node[myNode] (n1) at (3,2) {};
     \node[myNode] (n1) at (4,2) {};
     
     \draw (2,0.5) -- (3,1);
     \draw (2,2.5) -- (3,2);
     \draw (2,2.5) -- (2,1.5);
     \draw (2,0.5) -- (2,1.5);
     \draw (2,1.5) -- (3,1);
     \draw (2,1.5) -- (3,2);
     \draw (3,1) -- (4,1);
     \draw (3,2) -- (4,2);
     \draw (4,1) -- (5,1.5);
     \draw (4,2) -- (5,1.5);
     \draw (3,1) -- (3,2);
     \draw (4,1) -- (4,2);
     \end{tikzpicture}}} &
   %\subcaptionbox{${O}_{8}^{2}$\label{OS_82}}
   %
   {
    \scalebox{.53}{\begin{tikzpicture}[myNode/.style = black node]
     
     \node[myNode] (n1) at (1,7) {};
     \node[myNode] (n1) at (3,7) {};
     \node[myNode] (n1) at (2,7) {};
     \node[myNode] (n1) at (1.5,6) {};
     \node[myNode] (n1) at (2.5,6) {};
     \node[myNode] (n1) at (2,5) {};
     \node[myNode] (n1) at (2,6.3) {};
     
     \draw (1,7) -- (2,7);
     \draw (3,7) -- (2,7);
     \draw (1,7) -- (1.5,6);
     \draw (2.5,6) -- (3,7);
     \draw (2.5,6) -- (2,7);
     \draw (1.5,6) -- (2,7);
     \draw (2.5,6) -- (1.5,6);
     \draw (2.5,6) -- (2,5);
     \draw (1.5,6) -- (2,5);
     \draw (2,6.3) -- (1.5,6);
     \draw (2,6.3) -- (2.5,6);
     \draw (2,6.3) -- (2,7);
     \end{tikzpicture}}} \\
 \end{tabular}

\begin{tabular}{ c c c c c c c c c c c c c c c c c }
   
%   \subcaptionbox{${O}_{9}^{2}$\label{OS_92}}
   {
    \scalebox{.53}{\begin{tikzpicture}[myNode/.style = black node]
     
     \node[myNode] (n1) at (0.2,2) {};
     \node[myNode] (n1) at (1,2.7) {};
     \node[myNode] (n1) at (1,1.3) {};
     \node[myNode] (n1) at (2,3.1) {};
     \node[myNode] (n1) at (2,0.9) {};
     \node[myNode] (n1) at (3,2) {};
     \node[myNode] (n1) at (4,2) {};
     
     \draw (0.2,2) -- (1,2.7);
     \draw (0.2,2) -- (1,1.3);
     \draw (1,2.7) -- (1,1.3);
     \draw (1,2.7) -- (2,3.1);
     \draw (1,1.3) -- (2,0.9);
     \draw (2,3.1) -- (3,2);
     \draw (2,3.1) -- (4,2);
     \draw (2,0.9) -- (3,2);
     \draw (2,0.9) -- (4,2);
     \draw (3,2) -- (4,2);
     \end{tikzpicture}}} &
   %\subcaptionbox{${O}_{10}^{2}$\label{OS_102}}
   {
    \scalebox{.53}{\begin{tikzpicture}[myNode/.style = black node]
     
     \node[myNode] (n1) at (0.3,0.3) {};
     \node[myNode] (n1) at (1,0.3) {};
     \node[myNode] (n1) at (2.1,0.3) {};
     \node[myNode] (n1) at (0.7,1) {};
     \node[myNode] (n1) at (0.7,1.7) {};
     \node[myNode] (n1) at (1.5,1) {};
     \node[myNode] (n1) at (1.5,1.7) {};
     \node[myNode] (n1) at (2,-0.7) {};
     
     \draw (0.3,0.3) -- (1,0.3);
     \draw (0.3,0.3) -- (0.7,1);
     \draw (1,0.3) -- (0.7,1);
     \draw (1,0.3) -- (1.5,1);
     \draw (1,0.3) -- (2.1,0.3);
     \draw (1,0.3) -- (2,-0.7);
     \draw (2.1,0.3) -- (1.5,1);
     \draw (2.1,0.3) -- (2,-0.7);
     \draw (0.7,1) -- (0.7,1.7);
     \draw (0.7,1) -- (1.5,1);
     \draw (0.7,1) -- (1.5,1.7);
     \draw (0.7,1.7) -- (1.5,1);
     \draw (1.5,1) -- (1.5,1.7);
     \end{tikzpicture}}} &
   %\subcaptionbox{${O}_{11}^{2}$\label{OS_112}}
   {
    \scalebox{.53}{\begin{tikzpicture}[myNode/.style = black node]
     
     \node[myNode] (n1) at (1,7) {};
     \node[myNode] (n1) at (3,7) {};
     \node[myNode] (n1) at (2,7) {};
     \node[myNode] (n1) at (1.5,6) {};
     \node[myNode] (n1) at (2.5,6) {};
     \node[myNode] (n1) at (2,5) {};
     \node[myNode] (n1) at (2.5,6.65) {};
     \node[myNode] (n1) at (1.5,6.65) {};
     \node[myNode] (n1) at (2,5.65) {};

     \draw (1,7) -- (2,7);
     \draw (3,7) -- (2,7);
     \draw (1,7) -- (1.5,6);
     \draw (2.5,6) -- (3,7);
     \draw (2.5,6) -- (2,7);
     \draw (1.5,6) -- (2,7);
     \draw (2.5,6) -- (1.5,6);
     \draw (2.5,6) -- (2,5);
     \draw (1.5,6) -- (2,5);
     \draw (2,5.65) -- (1.5,6);
     \draw (2,5.65) -- (2.5,6);
     \draw (1.5,6.65) -- (1.5,6);
     \draw (1.5,6.65) -- (2,7);
     \draw (2.5,6.65) -- (2.5,6);
     \draw (2.5,6.65) -- (2,7);
     \end{tikzpicture}}} &   
   %\subcaptionbox{${O}_{1}^{3}$\label{OS_13}}
   {
    \scalebox{.53}{\begin{tikzpicture}[myNode/.style = black node]
     
     \node[myNode] (n1) at (1,3) {};
     \node[myNode] (n1) at (2,2) {};
     \node[myNode] (n1) at (1,1) {};
     \node[myNode] (n1) at (4,3) {};
     \node[myNode] (n1) at (3,2) {};
     \node[myNode] (n1) at (4,1) {};
     
     \draw (1,3) -- (2,2);
     \draw (1,1) -- (2,2);
     \draw (1,3) -- (4,3);
     \draw (1,3) -- (1,1);
     \draw (1,1) -- (4,1);
     \draw (2,2) -- (3,2);
     \draw (3,2) -- (4,1);
     \draw (3,2) -- (4,3);
     \draw (4,1) -- (4,3);
     
     \end{tikzpicture}}} &
   %\subcaptionbox{${O}_{2}^{3}$\label{OS_23}}
   {
    \scalebox{.53}{\begin{tikzpicture}[myNode/.style = black node]
     
     \node[myNode] (n1) at (1,3) {};
     \node[myNode] (n1) at (2,3) {};
     \node[myNode] (n1) at (3,3) {};
     \node[myNode] (n1) at (1,1) {};
     \node[myNode] (n1) at (2,1) {};
     \node[myNode] (n1) at (3,1) {};
     
     \draw (1,3) -- (1,1);
     \draw (1,3) -- (2,1);
     \draw (1,3) -- (3,1);
     \draw (2,3) -- (1,1);
     \draw (2,3) -- (2,1);
     \draw (2,3) -- (3,1);
     \draw (3,3) -- (1,1);
     \draw (3,3) -- (2,1);
     \draw (3,3) -- (3,1);
     
     \end{tikzpicture}}} &
   %\subcaptionbox{${O}_{3}^{3}$\label{OS_33}}
   {
    \scalebox{.53}{\begin{tikzpicture}[myNode/.style = black node]
     
     \node[myNode] (n1) at (2,2.6) {};
     \node[myNode] (n1) at (0.875,1.75) {};
     \node[myNode] (n1) at (3.125,1.75) {};
     \node[myNode] (n1) at (1.4,0.625) {};
     \node[myNode] (n1) at (2.6,0.625) {};
     
     \draw (2,2.6) -- (0.875,1.75);
     \draw (2,2.6) -- (3.125,1.75);
     \draw (2,2.6) -- (1.4,0.625);
     \draw (2,2.6) -- (2.6,0.625);
     \draw (0.875,1.75) -- (3.125,1.75);
     \draw (0.875,1.75) -- (1.4,0.625);
     \draw (0.875,1.75) -- (2.6,0.625);
     \draw (3.125,1.75) -- (1.4,0.625);
     \draw (3.125,1.75) -- (2.6,0.625);
     
     \end{tikzpicture}}} \\
     \end{tabular}
  \caption{The set ${\cal L}_{1}$  of the 19  nearly-biconnected minor-obstructions for ${\cal A}_{1}({\cal S})$ that are also obstructions for ${\cal A}_{1}({\cal P})$.}
  \label{bringing}
	
\end{figure}

% \end{figure}\sed{Μπορείτε να βάλετε τα 19 γραφήματα της εικόνας σε  3 σειρές; Υπάρχει χώρος!}

The set ${\bf obs}({\cal A}_{1}({\cal P}))$ consists of 33 graphs and  has been identified in~\cite{Leivaditis18mino}.
The correctness of \autoref{begrudge}
can be verified by exhaustive check, considering all nearly-biconnected graphs in ${\bf obs}({\cal A}_{1}({\cal P}))$ (they are 26) and then filter those 
that belong in ${\bf obs}({\cal A}_{1}({\cal S}))$. For this, one should pick
those that become 
apex sub-unicyclic after the contraction or removal of each of their edges. Notice that the fact that these graphs are not apex-sub-unicyclic follows directly by the fact that they are not apex-pseudoforests  (as members of ${\bf obs}({\cal A}_{1}({\cal P}))$)
and the fact that  ${\cal S}\subseteq {\cal P}$.
The choice of ${\cal L}_{1}$ is justified by the next lemma.

 \begin{lemma}\label{operetta}
If $G$ is a nearly-biconnected graph in ${\bf obs}({\cal A}_{1}({\cal S}))$, then $G\in{\bf obs}({\cal A}_{1}({\cal P}))$.
\end{lemma}

 \begin{proof}
 Let $G$ be a graph satisfying the assumptions of the lemma.
We need to show that $G \notin {\cal A}_{1}({\cal P})$ and that
 for every proper minor $H$ of $G$ it holds that $H\in {\cal A}_{1}({\cal P})$.
Notice that the latter is trivial since ${\cal A}_{1}({\cal S})\subseteq {\cal A}_{1}({\cal P})$
and therefore it remains to show that $G\notin {\cal A}_{1}({\cal P})$. We begin with the following claim:
 
 \medskip
 
 \noindent{\em Claim:} If $x\in V(G)$ is a ${\cal P}$-apex of $G$ then, then $x$ is a cut-vertex of $G$.
 %\aliv{Better as "If $x\in V(G)$ is a ${\cal P}$-apex of $G$ then.."?}
 
 \noindent{\em Proof of Claim:} Consider a vertex $x\in V(G)$.
 Since $G\in{\bf obs}({\cal A}_{1}({\cal S}))$, $x$ is not an ${\cal S}$-apex of $G$ and so there exist two
 cycles $C_{1}, C_{2}$ in $G\setminus x$. If $x$ is not a cut-vertex of $G$,
 then $G\setminus x$ is connected and therefore $C_{1}, C_{2}$ are in the same connected component of $G\setminus x$. Hence, $G\setminus x$ is not a pseudoforest and Claim follows.
 
 \medskip
 
 Suppose, towards a contradiction, that $G\in {\cal A}_{1}({\cal P})$. Then, there exists a vertex $x\in V(G)$ such that $G\setminus x$ is a pseudoforest.
 From the above Claim, $x$ is a
 cut-vertex of $G$ and since $G$ is nearly biconnected, ${\cal C}(G,x)=\{H,K_{3}\}$ where $H$ is a biconnected graph.
 Therefore, $H\setminus x$ is a connected component of $G\setminus x$ while the other connected component of $G\setminus x$ is a single edge.
 Therefore,   $G\setminus x$ contains at most one cycle which implies that $G\setminus x\in {\cal S}$, a contradiction.
 \end{proof}

 We are now ready to prove the main result of this section.
 
 \begin{theorem}
${\bf obs}({\cal A}_{1}({\cal S}))={\cal L}_{0}\cup {\cal L}_{1}.$
 \end{theorem} 
 
 \begin{proof}
 Recall that ${\cal L}_{0}={\cal O}^{0}\cup\{{{O}_{1}^{1}},{{O}_{2}^{1}},{{O}_{3}^{1}},{{O}_{4}^{1}}\}$. Notice that ${\cal L}_{0}\cup {\cal L}_{1}\subseteq {\bf obs}({\cal A}_{1}({\cal S}))$.
Let $G\in  {\bf obs}({\cal A}_{1}({\cal S}))$. If $G$ is disconnected, then, from \autoref{presents},
$G\in{\cal O}^{0}$.  If $G$ is connected and has at least three cut-vertices, then from~\autoref{immature}, $G\cong {{O}_{1}^{1}}$. If $G$ is connected and has exactly two  cut-vertices, then from~\autoref{daylight}, $G\in\{{{O}_{2}^{1}},{{O}_{3}^{1}}\}$. If $G$ is connected with exactly one a cut-vertex and is not  nearly-biconnected
 then, from \autoref{handling},  $G\cong{{O}_{4}^{1}}$.
We just proved that if  $G$ is not nearly-biconnected, then $G\in {\cal L}_{0}$. On the other side, if $G$ is nearly-biconnected, then from \autoref{operetta},  $G\in{\bf obs}({\cal A}_{1}({\cal P}))$, therefore, from \autoref{begrudge}, $G\in{\cal L}_{1}$, as required.
 \end{proof}

 \section{Structural Characterisation of Cactus Obstructions}
 \label{opera}

Recall that a {\em cactus graph} is a graph where all its blocks are either edges or cycles.
Equivalently, a graph is a cactus graph, if it does not contain $K^{-}_{4}$ as a minor.
We denote by ${\cal K}$ the set of all cactus graphs.
In this section we provide a complete characterization of the class of ${\bf obs}({\cal A}_{k}{\cal S})\cap {\cal K}$,
i.e., the obstructions for $k$-apex sub-unicyclic graphs that are cactus graphs.

Given a graph $G$ and two vertices $x$ and $y$ of $G$, we call a pair $(x,y)\in (V(G))^2$  {\em anti-diametrical}
if there is no other pair $(x',y')$, where the distance between $x'$ and $y'$ in $G$ is bigger than
%\aliv{Bigger is correct, right?} 
the distance in $G$ between $x$ and $y$. Notice that if $G$ is a tree, the 
two vertices in any anti-diametrical pair of $G$ are both leaves.

\paragraph{Block-cut-vertex Tree.}
Let $G$ be a connected graph. We denote by ${\cal B}(G)$ the set of its blocks and by $C(G)$, the set of its cut-vertices. 
We define the graph $T_{G}=({\cal B}(G)\cup C(G),E)$
where $E=\{\{B,c\}\mid B\in {\cal B}(G), v\in C(G), v\in V(B)\}$.
Notice that $T_{G}$ is a tree, called the {\em block-cut-vertex tree} of $G$ (or {\em bc-tree} in short). Furthermore, note that all its leafs are blocks of $G$.  We  call a block of $G$ {\em leaf-block} if $B$ is a leaf of $T_{G}$. We call a leaf-block $B$ of $G$ {\em peripheral} if  there is some leaf-block $B'$ of $G$ such that the pair $(B,B')$ is an  anti-diametrical pair of $T_{G}$.

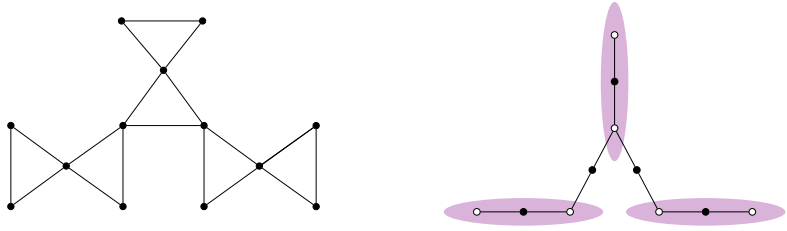
\begin{figure}
\begin{center}
\minipage{0.32\textwidth}
\resizebox{0.8\linewidth}{!}{
\begin{tikzpicture}[myNode/.style = black node]
     
     \node[myNode] (n1) at (10,0) {};
     \node[myNode] (n2) [below=10ex of n1]  {};
     \coordinate (midway1) at ($(n1)!0.5!(n2)$)  {};
     \node[myNode] (n3) [right=7ex of midway1] {};
     \node[myNode] (n4) [above right= 4.8 ex and 7ex of n3]  {};
     \node[myNode] (n5) [below=10ex of n4]  {};
     
     \node[myNode] (n6) [right=10ex of n4] {};
     \node[myNode] (n7) [below=10ex of n6]  {};
     \coordinate (midway2) at ($(n6)!0.5!(n7)$)  {};
     \node[myNode] (n8) [right=7ex of midway2] {};
     \node[myNode] (n9) [above right= 4.8 ex and 7ex of n8]  {};
     \node[myNode] (n10) [below=10ex of n9]  {};
     
     \coordinate (midway3) at ($(n4)!0.5!(n6)$)  {};
     \node[myNode] (n11) [above=7ex of midway3] {};
     \node[myNode] (n12) [above left= 6ex and 5ex of n11]  {};
     \node[myNode] (n13) [right=10ex of n12]  {};

     \draw (n1) -- (n2);
     \draw (n1) -- (n3);
     \draw (n2) -- (n3);
     \draw (n3) -- (n5);
     \draw (n3) -- (n4);
     \draw (n4) -- (n5);
     
     \draw (n6) -- (n7);
     \draw (n6) -- (n8);
     \draw (n7) -- (n9);
     \draw (n8) -- (n10);
     \draw (n8) -- (n9);
     \draw (n9) -- (n10);
     
     \draw (n4) -- (n6);
     \draw (n11) -- (n4);
     \draw (n11) -- (n6);
     \draw (n11) -- (n12);
     \draw (n11) -- (n13);
     \draw (n12) -- (n13);

\end{tikzpicture}}
\endminipage~~~
\minipage{0.35\textwidth}
\resizebox{.8\linewidth}{!}{
\begin{tikzpicture}[myNode/.style = black node]
     
    \node[myNode,fill=white] (n1) at (0,0) {};
    \node[myNode,fill=black] (n2) [below=5ex of n1]  {};
    \node[myNode,fill=white] (n3) [below=5ex of n2]  {};
    
    \node[myNode,fill=white] (n4) [below right =10ex and 5ex of n3]  {};
    \node[myNode,fill=black] (n5) [right = 5ex of n4]  {};
    \node[myNode,fill=white] (n6) [right = 5ex of n5]  {};
    
    \node[myNode,fill=white] (n7) [below left =10ex and 5ex of n3]  {};
    \node[myNode,fill=black] (n8) [left = 5ex of n7]  {};
    \node[myNode,fill=white] (n9) [left = 5ex of n8]  {};
    
    \coordinate (midway1) at ($(n3)!0.5!(n4)$)  {};
    \coordinate (midway2) at ($(n3)!0.5!(n7)$)  {};
    \node[myNode,fill=black] (n10) at (midway1) {};
    \node[myNode,fill=black] (n11) at (midway2) {};
    
    \draw (n1) -- (n2);
    \draw (n2) -- (n3);
    \draw (n3) -- (n10);
    \draw (n3) -- (n11);
    \draw (n10) -- (n4);
    \draw (n11) -- (n7);
    \draw (n4) -- (n5);
    \draw (n5) -- (n6);
    \draw (n7) -- (n8);
    \draw (n8) -- (n9);
    
    \begin{pgfonlayer}{background}
    %label=above:$S$ = for names
    \node[fit=(n7)(n9), fill=violet!30,ellipse] {};
    \node[fit=(n1)(n3), fill=violet!30,ellipse] {};
    \node[fit=(n4)(n6), fill=violet!30,ellipse] {};
    \end{pgfonlayer}
\end{tikzpicture}}
\endminipage
\end{center}

\caption{An example of a graph $G\in{\cal Z}_{3}$ and its block-cut-vertex tree $T_{G}$ with the $P_3$-subgraphs corresponding to the butterflies composing $G$ highlighted.}
\label{uttering}
\end{figure}

\subsection{Characterization of connected cactus-obstructions}

\paragraph{Butterflies and Butterfly-Cacti.}
We denote by $Z$ the butterfly graph. We will frequently refer to graphs isomorphic to $Z$ simply as {\em butterflies}. Given a butterfly $Z$ we call all its four vertices that have degree two, {\em extremal vertices} of $Z$ and the unique vertex of degree four, {\em central vertex} of $Z$.

Let $k$ be a positive integer.
We recursively define the graph class of the 
{\em $k$-butterfly-cacti}, denoted by ${\cal Z}_{k}$, as follows:
We set ${\cal Z}_{1}=\{Z\}$, where $Z$ is the butterfly graph,  
and given a $k\geq 2$ we say that 
$G\in {\cal Z}_{k}$ if there is a graph $G'\in {\cal Z}_{k-1}$ such that 
$G$ is obtained if we take a copy of the  butterfly graph $Z$ and then we identify one of its 
extremal vertices with a non-central vertex of $G'$. The {\em central vertices} of the obtained graph $G$  are the central vertices of $G'$ and the central vertex of $Z$. 
If $G\in{\cal Z}_{k}$, we denote by $K(G)$ the set of all central vertices of $G$.\medskip

We need the following observation.

\begin{observation} 
\label{furrowed}
For every $k\geq 1$ and for every $G\in {\cal Z}_{k}$, $K(G)$ is the unique $k$-apex forest set of $G$.
\end{observation}
\begin{proof}
It is easy to observe that  $K(G)$ is a $k$-apex forest set of $G$. 
To prove that $K(G)$ is unique, suppose to the contrary that $k$ is the minimum number such
that there is a $G\in {\cal Z}_{k}$
and a $k$-apex forest set $S\subseteq V(G)$ where $S\neq K(G)$. Recall that $G$ is obtained by identifying
a non-central vertex of some member $G'$ of ${\cal Z}_{k-1}$ with an extremal vertex of some graph $H$ isomorphic to the butterfly
graph $Z$. Let now $C$ be the cycle of $H$ in $G$ that contains no vertices of $G'$. By the 
minimality of $k$, $S\setminus V(C)=K(G')$ and therefore $S\cap V(C)$ must contain only one vertex, 
namely $x$, which must also belong to the cycle of $H$ different from $C$. This implies 
that $x$ is the central vertex of $Z$, thus $S= K(G)$, a contradiction.
\end{proof}

The objective of this section is to prove the following theorem.

\begin{theorem}
\label{holiness}
For every non-negative integer $k$, the connected graphs in ${\bf obs}({\cal A}_{k}{\cal S})\cap {\cal K}$
are exactly the graphs in ${\cal Z}_{k+1}$.
\end{theorem}

The following lemma proves one direction of \autoref{holiness}.

\begin{lemma}
\label{supports}
If $G\in {\cal Z}_{k+1}, k\geq 0$, then $G \in {\bf obs}({\cal A}_{k}({\cal S}))$. 
%and $G$ is an $n$-apex forest with a unique feedback vertex set composed of its central vertices.
\end{lemma}

\begin{proof}
We proceed by induction on $k$. The lemma clearly holds for $k = 0$. Let $G\in {\cal Z}_{k+1}$
for some $k\geq 1$ and assume that the lemma holds for smaller values of $k$.
%As $K(G)$ is a $(k+1)$-apex forest set of $G$, we  have $G\in {\cal S}^{(k+1)}$.
We argue that $G$ is not $k$-apex sub-unicyclic while all its proper minors are. 
By the construction of $G$, we know that $G$ is the result of the identification
of an extremal vertex of a new copy of the butterfly graph $Z$ and a non-central vertex of some
graph $G'\in {\cal Z}_{k}$. By the induction hypothesis, we have that $G'\in \cactobs{k-1}$.
Let $C$ (resp. $C'$) be the cycle of the new copy of $Z$ in $G$ that is (resp. is not) a leaf-block of $G$.
\medskip

\noindent {\em Claim 1}: $G$ is not $k$-apex sub-unicyclic. \smallskip

\noindent{\em Proof of Claim 1}: Suppose, towards a contradiction, that $G$ is $k$-apex sub-unicyclic and therefore there exists some $k$-apex sub-unicyclic set $S$ of $G$.
\smallskip

\noindent{\em Case 1:} $S\cap V(C)\neq \emptyset$. We set $S'=S \cap V(G')$.
Then $|S'|\leq k-1$ and we observe that $G'\setminus S'$ is sub-unicyclic contradicting the fact that $G'\in \cactobs{k-1}$.\smallskip

\noindent{\em Case 2:} $S\cap V(C)= \emptyset$. Then $S$ is a $k$-apex forest set of $G\setminus V(C)$ that should contain at least one vertex of $C'$. This means that $G'$ contains a $k$-apex forest set that is different from $K(G')$, a contradiction to \autoref{furrowed}.\medskip

\noindent {\em Claim 2}: Every proper minor of $G$ is $k$-apex sub-unicyclic.
\smallskip

\noindent{\em Proof of Claim 2}: 
Consider a minor $H$ of $G$ created by the  contraction (or removal) of some edge $e$ of $G$.
If $e$ is an edge of the copy of $Z$ in $G$, then observe that $K(G')$ is a $k$-apex sub-unicyclic set of $H$ and so the claim is proven. 
Suppose now that $e$ is an edge of $G'$ in $G$ and  let $H'$ be the minor of $G'$ created after contracting (or removing) $e$ in $G'$. 
Since $G'\in \cactobs{k-1}$, there exists a $(k-1)$-apex sub-unicyclic set $S'$ of $H'$. But then $S'$, together 
with the central vertex of $Z$, form a $k$-apex sub-unicyclic set of $H$, as required.
\medskip

From the above two claims, we conclude that $G \in {\bf obs}({\cal A}_{k}({\cal S}))$.  
\end{proof}

The following is a direct consequence of the application of~\autoref{segments} on cacti.
\begin{observation}
\label{conceive}
Let $G \in \cactobs{k} \cap {\cal K},  k\geq 0$.  Then all blocks of $G$ are triangles.
\end{observation}

\begin{lemma}
\label{typifies}
Let $k\geq 1$, $G$ be a connected cactus graph in $\cactobs{k}$ and let $B$ be some peripheral block of $G$. Then the (unique) neighbour $c$ of $B$  in $T_{G}$ has  degree 2.
\end{lemma}

\begin{proof} 
%From \autoref{conceive}, we know that each block of $G$ is a cycle.
Notice that $T_{G}$ has diameter at least $3$ since otherwise $G$ has a unique cut-vertex 
that is an 1-apex forest set, and therefore also an 1-apex sub-unicyclic set,  of $G$,
contradicting the fact that $G\notin{\cal A}_{1}({\cal S})$.
Suppose, towards a contradiction, that $c$ has degree at least three in $T_{G}$. Since $T_{G}$ has
diameter at least $3$ and $B$ is a peripheral leaf, there is exactly one neighbour, say $B'$, 
of $c$ in $T_{G}$ that is not a leaf-block of $G$. Let $e \in E(B')$ be some edge of said neighbour. Since $G\in \cactobs{k}$, we have that $G' = G\setminus e$ contains a $k$-apex sub-unicyclic set $S$. 
If $c\not\in S$, $S$ must contain at least one vertex from a leaf-block of $G$ that contains $c$. This follows
from the assumption that $c$ has at least two neighbours in $T_{G}$ which are leaf-blocks of $G$. But then 
the set $S'$ which is constructed by replacing these vertices with $c$ is also a $k$-apex 
sub-unicyclic set of $G'$. Therefore, we can assume that $c\in S$. But then $S$ is also an $k$-apex sub-unicyclic set 
for $G$, as $c\in V(B')$, a contradiction.
\end{proof}

\begin{lemma}
\label{solution}
Let $k\geq 1$ and $G$ be a connected cactus graph in $\cactobs{k}$. Let also $B$ be a peripheral block of $G$.
Then $T_{G}$ contains a path of length 3 whose one endpoint is $B$
and its internal vertices are of degree 2.
\end{lemma}
\begin{proof}
Let $c$ be the unique neighbour of $B$ in $T_{G}$.
By \autoref{typifies}, there exists a unique block $B'$ of $G$, different from $B$, that is a 
neighbour of $c$ in $T_{G}$. Observe that it suffices to prove that $B'$ has degree $2$ in $T_{G}$.

Suppose, towards a contradiction, that the block $B'$ has 3 neighbours $c,c',c''$ in $T_{G}$. Since $B$ is a  
peripheral leaf, we have that at least one of $c', c''$, say $c''$, is such that all its neighbours in $T_{G}$, except for $B'$, are 
leaf-blocks. Let $B''$ be a neighbour of $c''$ in $T_{G}$ different than $B'$. Consider now some edge 
$e \in E(B')$. Since $G\in \cactobs{k}$, we have that $G' = G\setminus e$ must contain a  $k$-apex 
sub-unicyclic set $S$. We can assume that $S$ contains one of $c,c''$. Indeed, we have
 that $S$ contains a vertex $x\in V(B)\cup V(B'')$. If $x\in V(B)$ then 
the set $S'=(S\setminus \{x\})\cup \{c\}$ is a $k$-apex sub-unicyclic set of $G'$. Respectively, if 
$x\in V(B'')$ then the set $S'=(S\setminus \{x\})\cup \{c''\}$ is a $k$-apex sub-unicyclic set of $G'$.
Assume then that $S$ is a $k$-apex sub-unicyclic set of $G'$ such that either $c$ or $c''$ is in $S$.
Then, $S$ is also a $k$-apex sub-unicyclic set of $G$ since both $c$ and $c''$ are vertices of $B'$, a contradiction.
\end{proof}

\noindent Given a graph $G$ we say that  a subgraph $Q$ of $G$ is a {\em leaf-butterfly of $G$}
if 
\begin{itemize}
	\item $Q$ is an induced subgraph of $G$,
	\item $Q$ is isomorphic to a butterfly graph,
	\item all the vertices of $Q$, except from an extremal one, called the {\em  attachment} of $Q$, have all their neighbours  inside $Q$ in  $G$, and
	\item the block of $Q$ that does not contain its attachment is a peripheral block of $G$.
\end{itemize}

A {\em butterfly bucket} of $G$ is a maximal collection ${\cal Q}=\{Q_{1},\ldots,Q_{r}\}$ of  leaf-butterflies of $G$ with the same attachment $w$ in $G$. If $G= \cupall{\cal Q}$ then we say that ${\cal Q}$ is a {\em trivial butterfly bucket}, otherwise we say that ${\cal Q}$ is a {\em non-trivial butterfly bucket}.
We call $w$ the {\em attachment} of ${\cal Q}$ in $G$.

By considering \autoref{solution} and \autoref{conceive} together, we have the following corollary:
\begin{corollary}
\label{barracks}
Let $k\geq1$, and let $G$ be a connected cactus graph in $\cactobs{k}$.
Then $G$ contains a butterfly bucket.
\end{corollary}

\begin{lemma}
\label{straight}
Let $k\geq 1$ and let ${\cal Q}$ be a non-trivial butterfly
bucket of a connected cactus graph $G$. If $G \in \cactobs{k}$
then there is no leaf-block of $G$ containing the attachment of ${\cal Q}$.
\end{lemma}

\begin{proof}
Suppose to the contrary that there exists a leaf-block $B$ of $G$ containing the attachment $w$ of ${\cal Q}$. Let $Q\in{\cal Q}$, let $c$ be the central vertex of $Q$, and let $A$ and $C$ be the two cycles of $Q$ such that $w$ is a vertex of $A$. Let $e$ be an edge of $A$ and $G'=G\setminus e$.
As $G \in \cactobs{k}$, it follows that
$G'=G\setminus e$ contains a $k$-apex sub-unicyclic set $S$.
If $S\cap V(C)=\emptyset$ then there exists some $x\in S\cap V(B)$ and therefore $S'=(S\setminus \{x\})\cup\{w\}$ is a $k$-apex sub-unicyclic set of $G$, a contradiction.
If there exists some $y\in S\cap V(C)$ then $S'=(S\setminus \{y\})\cup\{c\}$ is a $k$-apex sub-unicyclic set of $G$, again a contradiction.
\end{proof}

\begin{lemma} Let $k\geq 1$, $G$ be a connected cactus graph in $\cactobs{k}$, and ${\cal Q}$ be a non-trivial butterfly bucket of $G$ with attachment $w$. Then the graph $G' = G \setminus (V(\cupall {\cal Q})\setminus \{w\})$  is a connected cactus in $\cactobs{k-r}$ where $r= \lvert {\cal Q} \lvert$.
\label{animates} 
\end{lemma}

\begin{proof}
Let ${\cal Q}=\{Q_{1},\ldots,Q_{r}\}$.
For $i\in[r]$, let $A_i$ and $B_i$ be the two cycles of $Q_i$ such that $w$ is a vertex of $A_i$. Recall that $V(A_i)\cap V(B_i)$ is a singleton consisting of the central vertex, say $c_i$, of $Q_i$.  
Observe that $G'$ is a connected cactus and $w$ is contained in exactly one, say $B^*$, of the blocks of $G'$. This follows from the non-triviality of the butterfly bucket ${\cal Q}$, \autoref{straight}, \autoref{solution}, and the definition of a butterfly bucket.
\begin{figure}[!h]
	\begin{center}
		\includegraphics[width=4cm]{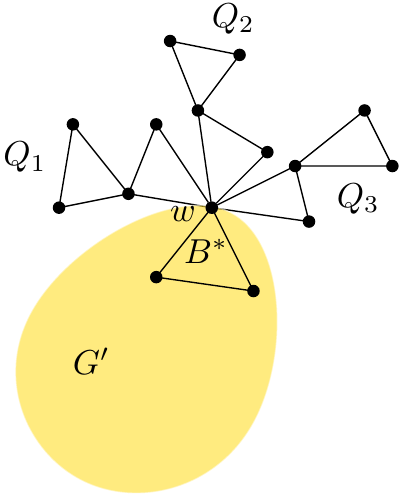}
	\end{center}
	\caption{An example of a graph $G$ and a butterfly bucket ${\cal Q}=\{Q_{1}, Q_{2}, Q_{3}\}$ of $G$ with attachment $w$.
		The graph $G'=G \setminus (V(\cupall {\cal Q})\setminus \{w\})$ is depicted in yellow and $B^{*}$ is the unique block of $G'$ that contains $w$.}
	\label{japanese}
\end{figure}

In what follows, we prove that $G'$  is a member of $\cactobs{k-r}$.\medskip

\noindent{\em Claim 1}: $G'$ is not $(k-r)$-apex sub-unicyclic.\smallskip

\noindent{\em Proof of Claim 1}: Suppose, to the contrary, that $S$ is a $(k-r)$-apex sub-unicyclic set of $G'$. Then $S\cup\{c_{1},\ldots,c_{r}\}$ is a $k$-apex sub-unicyclic set of $G$,
a contradiction as $G\in \cactobs{k}$.\medskip

\noindent {\em Claim 2}: Every proper  minor of $G'$ is $(k-r)$-apex sub-unicyclic.
\smallskip

\noindent{\em Proof of Claim 2}: 
Consider a minor $H'$ of $G'$ created by the contraction (or removal) of some edge $e$ of $G'$. Let  $H$ be the result of the contraction (or removal) of $e$ in $G$.
As $G\in \cactobs{k}$, there is a $k$-apex sub-unicyclic set $S$ in $H$.

We can assume that $\{c_{1},\dots,c_{r}\}\subseteq S$. Indeed, to see this is so, we can distinguish two cases:

\smallskip

\noindent{\em Case 1:} for every $i\in[r], \ S\cap V(B_{i})\neq\emptyset$. Then, for all $i\in[r]$ 
let $x_{i}\in S\cap V(B_{i})$ and observe that the set 
$S'=(S\setminus\{x_{1},\dots,x_{r}\})\cup\{c_{1},\dots c_{r}\}$ is a $k$-apex sub-unicyclic set of $G$.

\smallskip

\noindent{\em Case 2:} there is some  $i\in[r]$ such that $S\cap V(B_{i})=\emptyset$. Without loss of 
generality, we can assume that $i=1$. Then, the only cycle in $G\setminus S$
is $B_{1}$ and therefore for every $j\in[2,r],$ there exist some $x_{j}\in S\cap V(B_{j})$. Observe
that $S'=(S\setminus\{x_{2},\ldots,x_{r}\})\cup\{c_{2},\dots,c_{r}\}$ is a $k$-apex
sub-unicyclic set of $G$ (see \autoref{persuade}). As before, we have that there exists $x\in  S\cap V(A_{1})$. Set 
$S''=(S'\setminus\{x\})\cup\{c_{1}\}$. If $x\neq w$ then $S''$ is a $k$-apex
sub-unicyclic set of $G$. If $x=w$ then since $B_{i}$ is the only cycle in $G\setminus S'$ and 
$B^{*}$ is the only cycle in $G'$ that contains $w$, $S''$ is again a $k$-apex sub-unicycle set 
of $G$. %Also, note that $\{c_{1},\dots,c_{r}\}\subseteq S''$.
\begin{figure}[!h]
	\begin{center}
		\includegraphics[width=4cm]{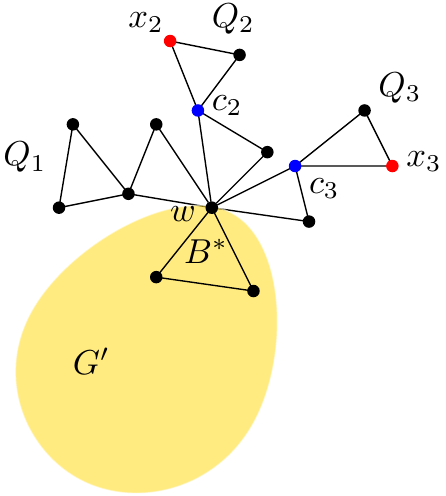}
	\end{center}
	\caption{Following the example in \autoref{japanese}, for every $i\in[2]$, $x_{i}$ is the vertex of $S$ that is in $V(B_{i})$ (depicted in red) and $c_{i}$ is the center of $Q_{i}$ (depicted in blue). The set $S'$ is obtained by replacing in $S$ the red vertices with the blue ones.}
	\label{persuade}
\end{figure}

\smallskip

Now, since $\{c_{1},\ldots,c_{r}\}\subseteq S$, we have that $S\setminus \{c_{1},\ldots,c_{r}\}$
is a $(k-r)$-apex sub-unicyclic set of $H'$ and so the claim follows.

\medskip

Based on the above two claims, we conclude that $G' \in \cactobs{k-r}$. 
\end{proof}

\begin{figure}
\centering
\scalebox{0.18}{

\definecolor{Gray}{gray}{0.9}
\newcolumntype{a}{>{\columncolor{Gray}}c}
\tikzset{every node/.style={black, shape=circle, fill=black, inner sep=0pt, minimum size=0.5cm}}

\begin{tabular}{ a@{\hskip 0.8in}  a@{\hskip 0.8in} a a}
%{\fontsize{50}{60} $n=1$} & {\fontsize{50}{60} $n=2$} & {\fontsize{50}{60} $n=3$} \\
\rule{10pt}{0pt}    
\begin{tikzpicture}[>=latex,line join=bevel,]
\node (node_8) at (71.017bp,-25.543bp) [draw,draw=none] {$8$};
  \node (node_7) at (51.317bp,-142.57bp) [draw,draw=none] {$7$};
  \node (node_6) at (-17.17bp,-142.48bp) [draw,draw=none] {$6$};
  \node (node_5) at (17.711bp,-74.81bp) [draw,draw=none] {$5$};
  \node (node_4) at (15.797bp,143.43bp) [draw,draw=none] {$4$};
  \node (node_3) at (-52.756bp,141.72bp) [draw,draw=none] {$3$};
  \node (node_2) at (-69.725bp,25.116bp) [draw,draw=none] {$2$};
  \node (node_1) at (0.8661bp,0.21434bp) [draw,draw=none] {$1$};
  \node (node_0) at (-17.058bp,74.93bp) [draw,draw=none] {$0$};
  \draw [black,] (node_0) ..controls (-27.454bp,94.38bp) and (-42.386bp,122.32bp)  .. (node_3);
  \draw [black,] (node_5) ..controls (27.498bp,-94.543bp) and (41.555bp,-122.89bp)  .. (node_7);
  \draw [black,] (node_0) ..controls (-33.151bp,59.709bp) and (-53.21bp,40.737bp)  .. (node_2);
  \draw [black,] (node_5) ..controls (34.108bp,-59.655bp) and (54.739bp,-40.588bp)  .. (node_8);
  \draw [black,] (node_6) ..controls (2.264bp,-142.51bp) and (32.281bp,-142.55bp)  .. (node_7);
  \draw [black,] (node_1) ..controls (20.772bp,-7.0944bp) and (51.518bp,-18.384bp)  .. (node_8);
  \draw [black,] (node_1) ..controls (-18.489bp,7.0421bp) and (-50.399bp,18.299bp)  .. (node_2);
  \draw [black,] (node_0) ..controls (-7.6759bp,94.49bp) and (6.1697bp,123.36bp)  .. (node_4);
  \draw [black,] (node_1) ..controls (5.5495bp,-20.645bp) and (12.988bp,-53.776bp)  .. (node_5);
  \draw [black,] (node_3) ..controls (-33.304bp,142.2bp) and (-3.2576bp,142.95bp)  .. (node_4);
  \draw [black,] (node_0) ..controls (-12.074bp,54.157bp) and (-4.1592bp,21.162bp)  .. (node_1);
  \draw [black,] (node_5) ..controls (7.553bp,-94.518bp) and (-7.0368bp,-122.82bp)  .. (node_6);
\end{tikzpicture}
&

\begin{tikzpicture}[>=latex,line join=bevel,]
\node (node_9) at (-84.938bp,-34.6bp) [draw,draw=none] {$9$};
  \node (node_8) at (38.272bp,-66.502bp) [draw,draw=none] {$8$};
  \node (node_7) at (158.97bp,-26.984bp) [draw,draw=none] {$7$};
  \node (node_6) at (137.2bp,-88.644bp) [draw,draw=none] {$6$};
  \node (node_5) at (82.838bp,-27.459bp) [draw,draw=none] {$5$};
  \node (node_4) at (49.267bp,149.54bp) [draw,draw=none] {$4$};
  \node (node_3) at (-15.071bp,158.78bp) [draw,draw=none] {$3$};
  \node (node_2) at (-42.197bp,57.899bp) [draw,draw=none] {$2$};
  \node (node_1) at (-1.2442bp,-2.8459bp) [draw,draw=none] {$1$};
  \node (node_0) at (12.285bp,82.585bp) [draw,draw=none] {$0$};
  \node (node_11) at (-161.56bp,-35.069bp) [draw,draw=none] {$11$};
  \node (node_10) at (-137.02bp,-96.271bp) [draw,draw=none] {$10$};
  \node (node_12) at (-36.807bp,-70.426bp) [draw,draw=none] {$12$};
  \draw [black,] (node_0) ..controls (4.6789bp,103.77bp) and (-7.4014bp,137.42bp)  .. (node_3);
  \draw [black,] (node_5) ..controls (103.42bp,-27.331bp) and (138.65bp,-27.111bp)  .. (node_7);
  \draw [black,] (node_0) ..controls (-4.4745bp,74.991bp) and (-25.56bp,65.437bp)  .. (node_2);
  \draw [black,] (node_9) ..controls (-70.186bp,-45.581bp) and (-55.105bp,-56.806bp)  .. (node_12);
  \draw [black,] (node_5) ..controls (68.127bp,-40.347bp) and (52.903bp,-53.684bp)  .. (node_8);
  \draw [black,] (node_6) ..controls (143.7bp,-70.239bp) and (152.39bp,-45.627bp)  .. (node_7);
  \draw [black,] (node_1) ..controls (10.63bp,-21.974bp) and (26.657bp,-47.791bp)  .. (node_8);
  \draw [black,] (node_1) ..controls (-13.468bp,15.285bp) and (-29.815bp,39.532bp)  .. (node_2);
  \draw [black,] (node_1) ..controls (-22.555bp,-10.932bp) and (-62.8bp,-26.201bp)  .. (node_9);
  \draw [black,] (node_0) ..controls (23.055bp,102.08bp) and (38.524bp,130.09bp)  .. (node_4);
  \draw [black,] (node_1) ..controls (20.166bp,-9.1133bp) and (60.598bp,-20.949bp)  .. (node_5);
  \draw [black,] (node_1) ..controls (-11.601bp,-22.527bp) and (-26.476bp,-50.794bp)  .. (node_12);
  \draw [black,] (node_3) ..controls (3.7928bp,156.07bp) and (31.134bp,152.14bp)  .. (node_4);
  \draw [black,] (node_0) ..controls (8.7065bp,59.99bp) and (2.3991bp,20.16bp)  .. (node_1);
  \draw [black,] (node_9) ..controls (-100.48bp,-53.008bp) and (-121.28bp,-77.624bp)  .. (node_10);
  \draw [black,] (node_10) ..controls (-144.35bp,-78.004bp) and (-154.14bp,-53.575bp)  .. (node_11);
  \draw [black,] (node_5) ..controls (99.065bp,-45.722bp) and (120.76bp,-70.144bp)  .. (node_6);
  \draw [black,] (node_9) ..controls (-104.43bp,-34.72bp) and (-137.1bp,-34.92bp)  .. (node_11);
\end{tikzpicture}

&
\begin{tikzpicture}[>=latex,line join=bevel,]
\node (node_14) at (109.16bp,131.21bp) [draw,draw=none] {$14$};
  \node (node_9) at (-82.219bp,-42.723bp) [draw,draw=none] {$9$};
  \node (node_8) at (65.112bp,-39.73bp) [draw,draw=none] {$8$};
  \node (node_7) at (110.74bp,-133.82bp) [draw,draw=none] {$7$};
  \node (node_13) at (66.956bp,66.184bp) [draw,draw=none] {$13$};
  \node (node_5) at (43.107bp,-86.687bp) [draw,draw=none] {$5$};
  \node (node_4) at (-58.28bp,156.96bp) [draw,draw=none] {$4$};
  \node (node_3) at (-111.2bp,123.56bp) [draw,draw=none] {$3$};
  \node (node_2) at (-56.719bp,33.18bp) [draw,draw=none] {$2$};
  \node (node_1) at (2.0301bp,-0.62883bp) [draw,draw=none] {$1$};
  \node (node_0) at (-39.924bp,81.171bp) [draw,draw=none] {$0$};
  \node (node_6) at (54.449bp,-163.96bp) [draw,draw=none] {$6$};
  \node (node_11) at (-159.98bp,-53.421bp) [draw,draw=none] {$11$};
  \node (node_16) at (70.788bp,16.672bp) [draw,draw=none] {$16$};
  \node (node_10) at (-129.93bp,-109.05bp) [draw,draw=none] {$10$};
  \node (node_15) at (149.42bp,82.451bp) [draw,draw=none] {$15$};
  \node (node_12) at (-33.499bp,-61.364bp) [draw,draw=none] {$12$};
  \draw [black,] (node_0) ..controls (-45.468bp,65.329bp) and (-51.205bp,48.935bp)  .. (node_2);
  \draw [black,] (node_9) ..controls (-67.183bp,-48.476bp) and (-51.652bp,-54.418bp)  .. (node_12);
  \draw [black,] (node_6) ..controls (71.25bp,-154.96bp) and (93.717bp,-142.93bp)  .. (node_7);
  \draw [black,] (node_1) ..controls (-15.624bp,9.5306bp) and (-39.45bp,23.242bp)  .. (node_2);
  \draw [black,] (node_10) ..controls (-139.45bp,-91.438bp) and (-150.72bp,-70.562bp)  .. (node_11);
  \draw [black,] (node_9) ..controls (-96.115bp,-62.04bp) and (-116.07bp,-89.784bp)  .. (node_10);
  \draw [black,] (node_0) ..controls (-59.468bp,92.795bp) and (-91.689bp,111.96bp)  .. (node_3);
  \draw [black,] (node_5) ..controls (62.297bp,-100.06bp) and (91.939bp,-120.72bp)  .. (node_7);
  \draw [black,] (node_14) ..controls (122.45bp,115.11bp) and (136.2bp,98.458bp)  .. (node_15);
  \draw [black,] (node_5) ..controls (50.324bp,-71.288bp) and (57.72bp,-55.505bp)  .. (node_8);
  \draw [black,] (node_13) ..controls (68.229bp,49.734bp) and (69.559bp,32.545bp)  .. (node_16);
  \draw [black,] (node_0) ..controls (-45.028bp,102.24bp) and (-53.133bp,135.71bp)  .. (node_4);
  \draw [black,] (node_13) ..controls (91.271bp,70.981bp) and (124.97bp,77.627bp)  .. (node_15);
  \draw [black,] (node_5) ..controls (46.217bp,-107.87bp) and (51.344bp,-142.8bp)  .. (node_6);
  \draw [black,] (node_1) ..controls (12.894bp,-23.39bp) and (32.045bp,-63.512bp)  .. (node_5);
  \draw [black,] (node_1) ..controls (-19.574bp,-11.423bp) and (-60.731bp,-31.987bp)  .. (node_9);
  \draw [black,] (node_1) ..controls (20.938bp,18.828bp) and (48.095bp,46.775bp)  .. (node_13);
  \draw [black,] (node_13) ..controls (79.639bp,85.724bp) and (96.758bp,112.1bp)  .. (node_14);
  \draw [black,] (node_0) ..controls (-28.585bp,59.062bp) and (-9.1715bp,21.211bp)  .. (node_1);
  \draw [black,] (node_3) ..controls (-94.924bp,133.84bp) and (-74.441bp,146.76bp)  .. (node_4);
  \draw [black,] (node_1) ..controls (20.401bp,-12.016bp) and (46.787bp,-28.371bp)  .. (node_8);
  \draw [black,] (node_1) ..controls (-8.9718bp,-19.436bp) and (-22.943bp,-43.319bp)  .. (node_12);
  \draw [black,] (node_1) ..controls (20.743bp,4.0795bp) and (48.086bp,10.959bp)  .. (node_16);
  \draw [black,] (node_9) ..controls (-102.07bp,-45.454bp) and (-135.51bp,-50.054bp)  .. (node_11);
\end{tikzpicture}

&

\begin{tikzpicture}[>=latex,line join=bevel,]
\node (node_13) at (125.03bp,94.416bp) [draw,draw=none] {$13$};
  \node (node_14) at (144.4bp,168.55bp) [draw,draw=none] {$14$};
  \node (node_9) at (-142.71bp,-73.951bp) [draw,draw=none] {$9$};
  \node (node_8) at (75.865bp,28.486bp) [draw,draw=none] {$8$};
  \node (node_7) at (135.61bp,-93.854bp) [draw,draw=none] {$7$};
  \node (node_6) at (87.096bp,-125.79bp) [draw,draw=none] {$6$};
  \node (node_5) at (68.674bp,-47.133bp) [draw,draw=none] {$5$};
  \node (node_4) at (-83.594bp,119.74bp) [draw,draw=none] {$4$};
  \node (node_3) at (-131.39bp,86.758bp) [draw,draw=none] {$3$};
  \node (node_2) at (-72.094bp,-32.237bp) [draw,draw=none] {$2$};
  \node (node_1) at (2.2916bp,-1.6038bp) [draw,draw=none] {$1$};
  \node (node_0) at (-64.756bp,41.738bp) [draw,draw=none] {$0$};
  \node (node_11) at (-218.56bp,-81.467bp) [draw,draw=none] {$11$};
  \node (node_16) at (150.43bp,33.575bp) [draw,draw=none] {$16$};
  \node (node_10) at (-187.0bp,-139.8bp) [draw,draw=none] {$10$};
  \node (node_15) at (196.69bp,128.86bp) [draw,draw=none] {$15$};
  \node (node_12) at (-85.979bp,-106.29bp) [draw,draw=none] {$12$};
  \draw [black,] (node_0) ..controls (-66.796bp,21.17bp) and (-70.037bp,-11.497bp)  .. (node_2);
  \draw [black,] (node_3) ..controls (-115.93bp,97.426bp) and (-98.819bp,109.24bp)  .. (node_4);
  \draw [black,] (node_6) ..controls (102.79bp,-115.46bp) and (120.16bp,-104.03bp)  .. (node_7);
  \draw [black,] (node_1) ..controls (-17.675bp,-9.8265bp) and (-51.554bp,-23.779bp)  .. (node_2);
  \draw [black,] (node_10) ..controls (-196.71bp,-121.86bp) and (-208.92bp,-99.28bp)  .. (node_11);
  \draw [black,] (node_9) ..controls (-155.6bp,-93.128bp) and (-174.13bp,-120.67bp)  .. (node_10);
  \draw [black,] (node_0) ..controls (-83.536bp,54.425bp) and (-112.28bp,73.847bp)  .. (node_3);
  \draw [black,] (node_5) ..controls (87.537bp,-60.3bp) and (116.41bp,-80.455bp)  .. (node_7);
  \draw [black,] (node_14) ..controls (162.66bp,154.69bp) and (178.16bp,142.92bp)  .. (node_15);
  \draw [black,] (node_2) ..controls (-75.954bp,-52.826bp) and (-82.086bp,-85.528bp)  .. (node_12);
  \draw [black,] (node_5) ..controls (70.674bp,-26.109bp) and (73.849bp,7.285bp)  .. (node_8);
  \draw [black,] (node_13) ..controls (132.61bp,76.256bp) and (142.75bp,51.971bp)  .. (node_16);
  \draw [black,] (node_0) ..controls (-69.921bp,63.126bp) and (-78.437bp,98.387bp)  .. (node_4);
  \draw [black,] (node_9) ..controls (-126.57bp,-83.147bp) and (-106.4bp,-94.646bp)  .. (node_12);
  \draw [black,] (node_13) ..controls (147.17bp,105.06bp) and (173.83bp,117.87bp)  .. (node_15);
  \draw [black,] (node_5) ..controls (73.725bp,-68.7bp) and (82.053bp,-104.26bp)  .. (node_6);
  \draw [black,] (node_8) ..controls (95.184bp,29.805bp) and (126.49bp,31.941bp)  .. (node_16);
  \draw [black,] (node_1) ..controls (21.0bp,-14.435bp) and (49.637bp,-34.076bp)  .. (node_5);
  \draw [black,] (node_2) ..controls (-91.454bp,-43.675bp) and (-123.37bp,-62.531bp)  .. (node_9);
  \draw [black,] (node_13) ..controls (130.41bp,115.03bp) and (138.97bp,147.77bp)  .. (node_14);
  \draw [black,] (node_0) ..controls (-45.861bp,29.523bp) and (-16.937bp,10.826bp)  .. (node_1);
  \draw [black,] (node_8) ..controls (90.183bp,47.686bp) and (110.75bp,75.264bp)  .. (node_13);
  \draw [black,] (node_1) ..controls (22.465bp,6.6465bp) and (55.723bp,20.248bp)  .. (node_8);
  \draw [black,] (node_9) ..controls (-162.36bp,-75.898bp) and (-194.21bp,-79.054bp)  .. (node_11);
\end{tikzpicture}

\\
&
\begin{tikzpicture}[>=latex,line join=bevel,]
\node (node_9) at (98.045bp,-53.079bp) [draw,draw=none] {$9$};
  \node (node_8) at (-107.94bp,35.888bp) [draw,draw=none] {$8$};
  \node (node_7) at (-147.18bp,-96.649bp) [draw,draw=none] {$7$};
  \node (node_6) at (-184.59bp,-41.778bp) [draw,draw=none] {$6$};
  \node (node_5) at (-106.67bp,-32.147bp) [draw,draw=none] {$5$};
  \node (node_4) at (45.679bp,143.21bp) [draw,draw=none] {$4$};
  \node (node_3) at (-15.775bp,149.43bp) [draw,draw=none] {$3$};
  \node (node_2) at (39.151bp,2.0139bp) [draw,draw=none] {$2$};
  \node (node_1) at (-37.803bp,9.8378bp) [draw,draw=none] {$1$};
  \node (node_0) at (7.5638bp,73.404bp) [draw,draw=none] {$0$};
  \node (node_11) at (124.18bp,-124.64bp) [draw,draw=none] {$11$};
  \node (node_10) at (172.18bp,-78.755bp) [draw,draw=none] {$10$};
  \node (node_12) at (113.15bp,13.264bp) [draw,draw=none] {$12$};
  \draw [black,] (node_0) ..controls (1.0748bp,94.543bp) and (-9.2318bp,128.12bp)  .. (node_3);
  \draw [black,] (node_5) ..controls (-118.84bp,-51.529bp) and (-135.27bp,-77.689bp)  .. (node_7);
  \draw [black,] (node_0) ..controls (16.466bp,53.285bp) and (30.092bp,22.487bp)  .. (node_2);
  \draw [black,] (node_2) ..controls (58.323bp,4.9286bp) and (89.396bp,9.6524bp)  .. (node_12);
  \draw [black,] (node_5) ..controls (-107.03bp,-12.719bp) and (-107.57bp,15.952bp)  .. (node_8);
  \draw [black,] (node_6) ..controls (-172.82bp,-59.038bp) and (-159.0bp,-79.301bp)  .. (node_7);
  \draw [black,] (node_1) ..controls (-57.706bp,17.23bp) and (-88.448bp,28.648bp)  .. (node_8);
  \draw [black,] (node_1) ..controls (-17.45bp,7.7685bp) and (18.428bp,4.1209bp)  .. (node_2);
  \draw [black,] (node_2) ..controls (56.848bp,-14.541bp) and (80.733bp,-36.885bp)  .. (node_9);
  \draw [black,] (node_0) ..controls (18.523bp,93.474bp) and (34.843bp,123.36bp)  .. (node_4);
  \draw [black,] (node_1) ..controls (-57.343bp,-2.0755bp) and (-87.526bp,-20.477bp)  .. (node_5);
  \draw [black,] (node_3) ..controls (2.1215bp,147.62bp) and (27.827bp,145.02bp)  .. (node_4);
  \draw [black,] (node_9) ..controls (102.44bp,-33.759bp) and (108.76bp,-6.0089bp)  .. (node_12);
  \draw [black,] (node_0) ..controls (-6.0687bp,54.303bp) and (-24.468bp,28.522bp)  .. (node_1);
  \draw [black,] (node_9) ..controls (117.25bp,-59.732bp) and (148.38bp,-70.513bp)  .. (node_10);
  \draw [black,] (node_10) ..controls (155.92bp,-94.302bp) and (140.02bp,-109.5bp)  .. (node_11);
  \draw [black,] (node_5) ..controls (-127.28bp,-34.694bp) and (-163.6bp,-39.185bp)  .. (node_6);
  \draw [black,] (node_9) ..controls (105.41bp,-73.248bp) and (116.69bp,-104.12bp)  .. (node_11);
\end{tikzpicture}
&
\begin{tikzpicture}[>=latex,line join=bevel,]
\node (node_14) at (134.51bp,162.5bp) [draw,draw=none] {$14$};
  \node (node_9) at (-111.42bp,-28.919bp) [draw,draw=none] {$9$};
  \node (node_8) at (-15.25bp,-93.16bp) [draw,draw=none] {$8$};
  \node (node_7) at (91.421bp,-152.4bp) [draw,draw=none] {$7$};
  \node (node_13) at (102.64bp,92.406bp) [draw,draw=none] {$13$};
  \node (node_5) at (37.048bp,-97.384bp) [draw,draw=none] {$5$};
  \node (node_4) at (-38.752bp,138.8bp) [draw,draw=none] {$4$};
  \node (node_3) at (-88.939bp,111.81bp) [draw,draw=none] {$3$};
  \node (node_2) at (40.216bp,37.092bp) [draw,draw=none] {$2$};
  \node (node_1) at (-18.193bp,-21.865bp) [draw,draw=none] {$1$};
  \node (node_0) at (-26.753bp,58.777bp) [draw,draw=none] {$0$};
  \node (node_6) at (35.11bp,-181.23bp) [draw,draw=none] {$6$};
  \node (node_11) at (-188.25bp,-21.142bp) [draw,draw=none] {$11$};
  \node (node_16) at (114.89bp,29.244bp) [draw,draw=none] {$16$};
  \node (node_10) at (-174.0bp,-83.486bp) [draw,draw=none] {$10$};
  \node (node_15) at (179.15bp,114.28bp) [draw,draw=none] {$15$};
  \node (node_12) at (-73.428bp,-65.322bp) [draw,draw=none] {$12$};
  \draw [black,] (node_0) ..controls (-7.8801bp,52.665bp) and (21.01bp,43.311bp)  .. (node_2);
  \draw [black,] (node_9) ..controls (-98.318bp,-41.473bp) and (-87.43bp,-51.905bp)  .. (node_12);
  \draw [black,] (node_6) ..controls (51.918bp,-172.62bp) and (74.394bp,-161.11bp)  .. (node_7);
  \draw [black,] (node_1) ..controls (-0.64135bp,-4.1486bp) and (23.047bp,19.762bp)  .. (node_2);
  \draw [black,] (node_2) ..controls (59.562bp,35.059bp) and (90.917bp,31.763bp)  .. (node_16);
  \draw [black,] (node_10) ..controls (-178.25bp,-64.878bp) and (-183.94bp,-39.993bp)  .. (node_11);
  \draw [black,] (node_9) ..controls (-129.15bp,-44.379bp) and (-153.69bp,-65.773bp)  .. (node_10);
  \draw [black,] (node_0) ..controls (-44.863bp,74.22bp) and (-70.874bp,96.401bp)  .. (node_3);
  \draw [black,] (node_5) ..controls (53.773bp,-114.31bp) and (74.817bp,-135.6bp)  .. (node_7);
  \draw [black,] (node_14) ..controls (149.25bp,146.58bp) and (164.5bp,130.11bp)  .. (node_15);
  \draw [black,] (node_5) ..controls (20.489bp,-96.047bp) and (0.86367bp,-94.461bp)  .. (node_8);
  \draw [black,] (node_13) ..controls (106.32bp,73.426bp) and (111.29bp,47.81bp)  .. (node_16);
  \draw [black,] (node_0) ..controls (-29.974bp,80.257bp) and (-35.439bp,116.7bp)  .. (node_4);
  \draw [black,] (node_13) ..controls (126.18bp,99.134bp) and (155.79bp,107.6bp)  .. (node_15);
  \draw [black,] (node_5) ..controls (36.535bp,-119.56bp) and (35.632bp,-158.65bp)  .. (node_6);
  \draw [black,] (node_2) ..controls (58.026bp,52.872bp) and (82.906bp,74.918bp)  .. (node_13);
  \draw [black,] (node_1) ..controls (-2.8342bp,-42.862bp) and (21.56bp,-76.211bp)  .. (node_5);
  \draw [black,] (node_1) ..controls (-41.161bp,-23.603bp) and (-88.135bp,-27.157bp)  .. (node_9);
  \draw [black,] (node_13) ..controls (111.81bp,112.56bp) and (125.45bp,142.57bp)  .. (node_14);
  \draw [black,] (node_0) ..controls (-24.456bp,37.131bp) and (-20.557bp,0.40276bp)  .. (node_1);
  \draw [black,] (node_3) ..controls (-73.153bp,120.3bp) and (-54.619bp,130.27bp)  .. (node_4);
  \draw [black,] (node_1) ..controls (-17.364bp,-41.957bp) and (-16.094bp,-72.714bp)  .. (node_8);
  \draw [black,] (node_1) ..controls (-34.504bp,-34.698bp) and (-53.876bp,-49.939bp)  .. (node_12);
  \draw [black,] (node_9) ..controls (-131.03bp,-26.934bp) and (-164.07bp,-23.59bp)  .. (node_11);
\end{tikzpicture}

&

\begin{tikzpicture}[>=latex,line join=bevel,]
\node (node_13) at (176.12bp,101.6bp) [draw,draw=none] {$13$};
  \node (node_14) at (193.51bp,176.0bp) [draw,draw=none] {$14$};
  \node (node_9) at (-148.33bp,-110.03bp) [draw,draw=none] {$9$};
  \node (node_8) at (28.575bp,-51.877bp) [draw,draw=none] {$8$};
  \node (node_7) at (119.29bp,-26.186bp) [draw,draw=none] {$7$};
  \node (node_6) at (123.0bp,45.723bp) [draw,draw=none] {$6$};
  \node (node_5) at (52.105bp,12.911bp) [draw,draw=none] {$5$};
  \node (node_4) at (-122.61bp,101.05bp) [draw,draw=none] {$4$};
  \node (node_3) at (-167.77bp,62.536bp) [draw,draw=none] {$3$};
  \node (node_2) at (-92.462bp,-51.758bp) [draw,draw=none] {$2$};
  \node (node_1) at (-27.382bp,-5.5519bp) [draw,draw=none] {$1$};
  \node (node_0) at (-97.87bp,25.269bp) [draw,draw=none] {$0$};
  \node (node_11) at (-219.92bp,-134.08bp) [draw,draw=none] {$11$};
  \node (node_16) at (198.93bp,33.301bp) [draw,draw=none] {$16$};
  \node (node_10) at (-175.7bp,-184.25bp) [draw,draw=none] {$10$};
  \node (node_15) at (245.72bp,133.89bp) [draw,draw=none] {$15$};
  \node (node_12) at (-85.198bp,-128.54bp) [draw,draw=none] {$12$};
  \draw [black,] (node_0) ..controls (-96.387bp,4.1495bp) and (-93.942bp,-30.67bp)  .. (node_2);
  \draw [black,] (node_3) ..controls (-152.87bp,75.25bp) and (-137.44bp,88.408bp)  .. (node_4);
  \draw [black,] (node_6) ..controls (121.96bp,25.457bp) and (120.35bp,-5.5638bp)  .. (node_7);
  \draw [black,] (node_6) ..controls (142.67bp,42.504bp) and (174.56bp,37.288bp)  .. (node_16);
  \draw [black,] (node_1) ..controls (-45.723bp,-18.574bp) and (-73.798bp,-38.507bp)  .. (node_2);
  \draw [black,] (node_10) ..controls (-189.91bp,-168.13bp) and (-205.49bp,-150.45bp)  .. (node_11);
  \draw [black,] (node_9) ..controls (-155.94bp,-130.67bp) and (-168.03bp,-163.45bp)  .. (node_10);
  \draw [black,] (node_6) ..controls (139.23bp,62.795bp) and (159.46bp,84.076bp)  .. (node_13);
  \draw [black,] (node_0) ..controls (-117.7bp,35.844bp) and (-148.34bp,52.177bp)  .. (node_3);
  \draw [black,] (node_5) ..controls (71.039bp,1.8928bp) and (100.02bp,-14.974bp)  .. (node_7);
  \draw [black,] (node_14) ..controls (211.42bp,161.55bp) and (227.65bp,148.47bp)  .. (node_15);
  \draw [black,] (node_2) ..controls (-90.47bp,-72.809bp) and (-87.187bp,-107.52bp)  .. (node_12);
  \draw [black,] (node_5) ..controls (45.034bp,-6.5574bp) and (35.491bp,-32.833bp)  .. (node_8);
  \draw [black,] (node_13) ..controls (182.63bp,82.094bp) and (192.25bp,53.313bp)  .. (node_16);
  \draw [black,] (node_0) ..controls (-104.75bp,46.339bp) and (-115.68bp,79.805bp)  .. (node_4);
  \draw [black,] (node_9) ..controls (-130.7bp,-115.2bp) and (-106.75bp,-122.22bp)  .. (node_12);
  \draw [black,] (node_13) ..controls (198.01bp,111.75bp) and (223.71bp,123.68bp)  .. (node_15);
  \draw [black,] (node_5) ..controls (71.544bp,21.908bp) and (103.59bp,36.74bp)  .. (node_6);
  \draw [black,] (node_1) ..controls (-6.3589bp,-0.66871bp) and (30.7bp,7.9393bp)  .. (node_5);
  \draw [black,] (node_2) ..controls (-109.65bp,-69.684bp) and (-131.27bp,-92.238bp)  .. (node_9);
  \draw [black,] (node_13) ..controls (180.96bp,122.28bp) and (188.64bp,155.14bp)  .. (node_14);
  \draw [black,] (node_0) ..controls (-78.676bp,16.877bp) and (-47.315bp,3.1639bp)  .. (node_1);
  \draw [black,] (node_1) ..controls (-10.054bp,-19.897bp) and (11.95bp,-38.113bp)  .. (node_8);
  \draw [black,] (node_9) ..controls (-167.15bp,-116.35bp) and (-196.51bp,-126.22bp)  .. (node_11);
\end{tikzpicture}
\\
&
\begin{tikzpicture}[>=latex,line join=bevel,]
\node (node_9) at (-41.65bp,144.46bp) [draw,draw=none] {$9$};
  \node (node_8) at (84.24bp,-114.89bp) [draw,draw=none] {$8$};
  \node (node_7) at (43.918bp,-223.52bp) [draw,draw=none] {$7$};
  \node (node_6) at (-23.318bp,-213.3bp) [draw,draw=none] {$6$};
  \node (node_5) at (21.766bp,-151.19bp) [draw,draw=none] {$5$};
  \node (node_4) at (68.928bp,50.376bp) [draw,draw=none] {$4$};
  \node (node_3) at (2.0497bp,80.52bp) [draw,draw=none] {$3$};
  \node (node_2) at (-45.45bp,-44.12bp) [draw,draw=none] {$2$};
  \node (node_1) at (20.933bp,-73.993bp) [draw,draw=none] {$1$};
  \node (node_0) at (11.149bp,3.3006bp) [draw,draw=none] {$0$};
  \node (node_11) at (-61.179bp,218.12bp) [draw,draw=none] {$11$};
  \node (node_10) at (-112.34bp,174.4bp) [draw,draw=none] {$10$};
  \node (node_12) at (30.957bp,149.84bp) [draw,draw=none] {$12$};
  \draw [black,] (node_0) ..controls (8.6542bp,24.473bp) and (4.5408bp,59.379bp)  .. (node_3);
  \draw [black,] (node_5) ..controls (28.009bp,-171.58bp) and (37.565bp,-202.78bp)  .. (node_7);
  \draw [black,] (node_0) ..controls (-5.7447bp,-10.853bp) and (-28.336bp,-29.781bp)  .. (node_2);
  \draw [black,] (node_9) ..controls (-22.565bp,145.88bp) and (7.2097bp,148.08bp)  .. (node_12);
  \draw [black,] (node_3) ..controls (10.361bp,100.45bp) and (22.739bp,130.13bp)  .. (node_12);
  \draw [black,] (node_6) ..controls (-4.3697bp,-216.18bp) and (24.636bp,-220.59bp)  .. (node_7);
  \draw [black,] (node_5) ..controls (39.96bp,-140.62bp) and (66.091bp,-125.44bp)  .. (node_8);
  \draw [black,] (node_1) ..controls (39.369bp,-85.904bp) and (65.849bp,-103.01bp)  .. (node_8);
  \draw [black,] (node_3) ..controls (-11.082bp,99.735bp) and (-28.805bp,125.67bp)  .. (node_9);
  \draw [black,] (node_1) ..controls (2.2248bp,-65.574bp) and (-26.413bp,-52.687bp)  .. (node_2);
  \draw [black,] (node_0) ..controls (28.395bp,17.352bp) and (51.458bp,36.142bp)  .. (node_4);
  \draw [black,] (node_1) ..controls (21.161bp,-95.16bp) and (21.538bp,-130.06bp)  .. (node_5);
  \draw [black,] (node_3) ..controls (20.897bp,72.025bp) and (49.748bp,59.021bp)  .. (node_4);
  \draw [black,] (node_0) ..controls (13.832bp,-17.892bp) and (18.254bp,-52.832bp)  .. (node_1);
  \draw [black,] (node_9) ..controls (-60.232bp,152.33bp) and (-89.222bp,164.61bp)  .. (node_10);
  \draw [black,] (node_10) ..controls (-94.681bp,189.49bp) and (-78.516bp,203.3bp)  .. (node_11);
  \draw [black,] (node_5) ..controls (8.3094bp,-169.73bp) and (-9.6863bp,-194.52bp)  .. (node_6);
  \draw [black,] (node_9) ..controls (-47.08bp,164.94bp) and (-55.704bp,197.47bp)  .. (node_11);
\end{tikzpicture}

&

\begin{tikzpicture}[>=latex,line join=bevel,]
\node (node_14) at (67.787bp,240.38bp) [draw,draw=none] {$14$};
  \node (node_9) at (-119.27bp,-59.702bp) [draw,draw=none] {$9$};
  \node (node_8) at (-12.445bp,-124.76bp) [draw,draw=none] {$8$};
  \node (node_7) at (107.27bp,-154.15bp) [draw,draw=none] {$7$};
  \node (node_13) at (82.616bp,164.52bp) [draw,draw=none] {$13$};
  \node (node_5) at (42.237bp,-113.21bp) [draw,draw=none] {$5$};
  \node (node_4) at (-15.347bp,96.887bp) [draw,draw=none] {$4$};
  \node (node_3) at (54.34bp,90.862bp) [draw,draw=none] {$3$};
  \node (node_2) at (-48.772bp,9.7099bp) [draw,draw=none] {$2$};
  \node (node_1) at (-26.126bp,-54.493bp) [draw,draw=none] {$1$};
  \node (node_0) at (11.654bp,23.835bp) [draw,draw=none] {$0$};
  \node (node_6) at (58.412bp,-194.81bp) [draw,draw=none] {$6$};
  \node (node_11) at (-195.23bp,-44.273bp) [draw,draw=none] {$11$};
  \node (node_16) at (126.72bp,108.63bp) [draw,draw=none] {$16$};
  \node (node_10) at (-184.56bp,-107.9bp) [draw,draw=none] {$10$};
  \node (node_15) at (132.58bp,222.0bp) [draw,draw=none] {$15$};
  \node (node_12) at (-81.859bp,-103.53bp) [draw,draw=none] {$12$};
  \draw [black,] (node_0) ..controls (-5.8239bp,19.749bp) and (-30.699bp,13.935bp)  .. (node_2);
  \draw [black,] (node_9) ..controls (-106.36bp,-74.832bp) and (-94.537bp,-88.678bp)  .. (node_12);
  \draw [black,] (node_6) ..controls (74.215bp,-181.66bp) and (91.71bp,-167.1bp)  .. (node_7);
  \draw [black,] (node_1) ..controls (-32.931bp,-35.2bp) and (-42.116bp,-9.162bp)  .. (node_2);
  \draw [black,] (node_10) ..controls (-187.77bp,-88.78bp) and (-192.09bp,-62.975bp)  .. (node_11);
  \draw [black,] (node_9) ..controls (-136.91bp,-72.726bp) and (-162.45bp,-91.575bp)  .. (node_10);
  \draw [black,] (node_0) ..controls (24.085bp,43.354bp) and (41.94bp,71.391bp)  .. (node_3);
  \draw [black,] (node_5) ..controls (60.565bp,-124.75bp) and (88.621bp,-142.41bp)  .. (node_7);
  \draw [black,] (node_14) ..controls (88.959bp,234.37bp) and (111.62bp,227.95bp)  .. (node_15);
  \draw [black,] (node_5) ..controls (25.416bp,-116.76bp) and (4.2531bp,-121.23bp)  .. (node_8);
  \draw [black,] (node_13) ..controls (96.58bp,146.82bp) and (113.13bp,125.85bp)  .. (node_16);
  \draw [black,] (node_3) ..controls (62.202bp,111.34bp) and (74.689bp,143.87bp)  .. (node_13);
  \draw [black,] (node_0) ..controls (3.9925bp,44.564bp) and (-7.8416bp,76.582bp)  .. (node_4);
  \draw [black,] (node_13) ..controls (97.985bp,182.2bp) and (117.32bp,204.45bp)  .. (node_15);
  \draw [black,] (node_3) ..controls (73.364bp,95.532bp) and (103.04bp,102.82bp)  .. (node_16);
  \draw [black,] (node_5) ..controls (46.609bp,-135.27bp) and (54.093bp,-173.02bp)  .. (node_6);
  \draw [black,] (node_1) ..controls (-6.7278bp,-71.156bp) and (23.235bp,-96.892bp)  .. (node_5);
  \draw [black,] (node_1) ..controls (-49.075bp,-55.776bp) and (-96.009bp,-58.401bp)  .. (node_9);
  \draw [black,] (node_13) ..controls (78.493bp,185.61bp) and (71.944bp,219.11bp)  .. (node_14);
  \draw [black,] (node_0) ..controls (1.2952bp,2.3582bp) and (-15.783bp,-33.049bp)  .. (node_1);
  \draw [black,] (node_3) ..controls (34.566bp,92.572bp) and (4.0231bp,95.213bp)  .. (node_4);
  \draw [black,] (node_1) ..controls (-22.193bp,-74.695bp) and (-16.335bp,-104.78bp)  .. (node_8);
  \draw [black,] (node_1) ..controls (-42.813bp,-69.175bp) and (-63.024bp,-86.958bp)  .. (node_12);
  \draw [black,] (node_9) ..controls (-138.95bp,-55.704bp) and (-170.85bp,-49.225bp)  .. (node_11);
\end{tikzpicture}
&

\begin{tikzpicture}[>=latex,line join=bevel,]
\node (node_14) at (61.516bp,237.32bp) [draw,draw=none] {$14$};
  \node (node_9) at (44.495bp,-137.86bp) [draw,draw=none] {$9$};
  \node (node_8) at (-122.81bp,7.3852bp) [draw,draw=none] {$8$};
  \node (node_7) at (-193.57bp,-104.53bp) [draw,draw=none] {$7$};
  \node (node_13) at (77.92bp,161.87bp) [draw,draw=none] {$13$};
  \node (node_5) at (-135.9bp,-55.321bp) [draw,draw=none] {$5$};
  \node (node_4) at (-13.524bp,86.388bp) [draw,draw=none] {$4$};
  \node (node_3) at (53.189bp,85.738bp) [draw,draw=none] {$3$};
  \node (node_2) at (13.441bp,-61.683bp) [draw,draw=none] {$2$};
  \node (node_1) at (-57.351bp,-30.24bp) [draw,draw=none] {$1$};
  \node (node_0) at (7.4815bp,16.511bp) [draw,draw=none] {$0$};
  \node (node_6) at (-215.23bp,-42.673bp) [draw,draw=none] {$6$};
  \node (node_11) at (42.265bp,-213.86bp) [draw,draw=none] {$11$};
  \node (node_16) at (123.91bp,108.46bp) [draw,draw=none] {$16$};
  \node (node_10) at (103.74bp,-191.63bp) [draw,draw=none] {$10$};
  \node (node_15) at (126.53bp,221.24bp) [draw,draw=none] {$15$};
  \node (node_12) at (83.894bp,-87.106bp) [draw,draw=none] {$12$};
  \draw [black,] (node_0) ..controls (9.1155bp,-4.9287bp) and (11.81bp,-40.275bp)  .. (node_2);
  \draw [black,] (node_3) ..controls (34.388bp,85.921bp) and (5.608bp,86.202bp)  .. (node_4);
  \draw [black,] (node_6) ..controls (-208.77bp,-61.138bp) and (-200.12bp,-85.83bp)  .. (node_7);
  \draw [black,] (node_1) ..controls (-37.941bp,-38.861bp) and (-5.9398bp,-53.075bp)  .. (node_2);
  \draw [black,] (node_10) ..controls (83.579bp,-198.92bp) and (62.918bp,-206.39bp)  .. (node_11);
  \draw [black,] (node_9) ..controls (61.94bp,-153.69bp) and (84.836bp,-174.47bp)  .. (node_10);
  \draw [black,] (node_0) ..controls (20.623bp,36.415bp) and (40.194bp,66.056bp)  .. (node_3);
  \draw [black,] (node_5) ..controls (-153.11bp,-70.011bp) and (-176.13bp,-89.654bp)  .. (node_7);
  \draw [black,] (node_14) ..controls (82.761bp,232.06bp) and (105.5bp,226.44bp)  .. (node_15);
  \draw [black,] (node_2) ..controls (31.96bp,-68.365bp) and (60.851bp,-78.791bp)  .. (node_12);
  \draw [black,] (node_5) ..controls (-131.99bp,-36.605bp) and (-126.77bp,-11.575bp)  .. (node_8);
  \draw [black,] (node_13) ..controls (92.387bp,145.07bp) and (109.37bp,125.35bp)  .. (node_16);
  \draw [black,] (node_3) ..controls (60.065bp,106.91bp) and (70.986bp,140.53bp)  .. (node_13);
  \draw [black,] (node_0) ..controls (1.442bp,36.602bp) and (-7.5521bp,66.521bp)  .. (node_4);
  \draw [black,] (node_9) ..controls (57.155bp,-121.55bp) and (71.036bp,-103.67bp)  .. (node_12);
  \draw [black,] (node_13) ..controls (92.873bp,180.13bp) and (111.69bp,203.11bp)  .. (node_15);
  \draw [black,] (node_3) ..controls (71.779bp,91.71bp) and (100.78bp,101.03bp)  .. (node_16);
  \draw [black,] (node_5) ..controls (-156.88bp,-51.976bp) and (-193.87bp,-46.079bp)  .. (node_6);
  \draw [black,] (node_1) ..controls (-78.125bp,-36.874bp) and (-114.74bp,-48.567bp)  .. (node_5);
  \draw [black,] (node_2) ..controls (22.075bp,-82.863bp) and (35.789bp,-116.5bp)  .. (node_9);
  \draw [black,] (node_13) ..controls (73.359bp,182.85bp) and (66.115bp,216.17bp)  .. (node_14);
  \draw [black,] (node_0) ..controls (-11.527bp,2.8038bp) and (-39.078bp,-17.063bp)  .. (node_1);
  \draw [black,] (node_1) ..controls (-75.798bp,-19.637bp) and (-104.04bp,-3.4052bp)  .. (node_8);
  \draw [black,] (node_9) ..controls (43.875bp,-158.99bp) and (42.89bp,-192.55bp)  .. (node_11);
\end{tikzpicture}
\\ & & &
\begin{tikzpicture}[>=latex,line join=bevel,]
\node (node_13) at (217.77bp,-3.4243bp) [draw,draw=none] {$13$};
  \node (node_14) at (294.64bp,2.0063bp) [draw,draw=none] {$14$};
  \node (node_9) at (-218.59bp,51.385bp) [draw,draw=none] {$9$};
  \node (node_8) at (19.39bp,-101.46bp) [draw,draw=none] {$8$};
  \node (node_7) at (117.81bp,-105.59bp) [draw,draw=none] {$7$};
  \node (node_6) at (148.2bp,-38.432bp) [draw,draw=none] {$6$};
  \node (node_5) at (70.61bp,-47.206bp) [draw,draw=none] {$5$};
  \node (node_4) at (-79.995bp,84.279bp) [draw,draw=none] {$4$};
  \node (node_3) at (-141.48bp,47.17bp) [draw,draw=none] {$3$};
  \node (node_2) at (-72.056bp,-63.207bp) [draw,draw=none] {$2$};
  \node (node_1) at (-5.6609bp,-31.029bp) [draw,draw=none] {$1$};
  \node (node_0) at (-72.301bp,10.283bp) [draw,draw=none] {$0$};
  \node (node_11) at (-288.1bp,82.754bp) [draw,draw=none] {$11$};
  \node (node_16) at (214.49bp,-75.242bp) [draw,draw=none] {$16$};
  \node (node_10) at (-286.3bp,15.373bp) [draw,draw=none] {$10$};
  \node (node_15) at (261.35bp,59.629bp) [draw,draw=none] {$15$};
  \node (node_12) at (-179.78bp,112.71bp) [draw,draw=none] {$12$};
  \draw [black,] (node_0) ..controls (-72.232bp,-10.57bp) and (-72.124bp,-42.78bp)  .. (node_2);
  \draw [black,] (node_9) ..controls (-207.0bp,69.69bp) and (-191.51bp,94.168bp)  .. (node_12);
  \draw [black,] (node_6) ..controls (139.35bp,-57.991bp) and (126.64bp,-86.083bp)  .. (node_7);
  \draw [black,] (node_6) ..controls (166.12bp,-48.379bp) and (192.04bp,-62.774bp)  .. (node_16);
  \draw [black,] (node_1) ..controls (-24.372bp,-40.098bp) and (-53.015bp,-53.979bp)  .. (node_2);
  \draw [black,] (node_10) ..controls (-286.82bp,34.996bp) and (-287.57bp,63.18bp)  .. (node_11);
  \draw [black,] (node_9) ..controls (-237.02bp,41.584bp) and (-263.94bp,27.263bp)  .. (node_10);
  \draw [black,] (node_6) ..controls (167.14bp,-28.904bp) and (194.8bp,-14.983bp)  .. (node_13);
  \draw [black,] (node_0) ..controls (-91.933bp,20.75bp) and (-122.25bp,36.917bp)  .. (node_3);
  \draw [black,] (node_5) ..controls (85.129bp,-65.167bp) and (103.4bp,-87.764bp)  .. (node_7);
  \draw [black,] (node_14) ..controls (284.4bp,19.732bp) and (271.51bp,42.033bp)  .. (node_15);
  \draw [black,] (node_5) ..controls (54.392bp,-64.384bp) and (35.172bp,-84.741bp)  .. (node_8);
  \draw [black,] (node_13) ..controls (216.85bp,-23.664bp) and (215.43bp,-54.646bp)  .. (node_16);
  \draw [black,] (node_3) ..controls (-161.88bp,48.285bp) and (-197.83bp,50.25bp)  .. (node_9);
  \draw [black,] (node_0) ..controls (-74.441bp,30.857bp) and (-77.838bp,63.533bp)  .. (node_4);
  \draw [black,] (node_13) ..controls (230.87bp,15.523bp) and (248.54bp,41.095bp)  .. (node_15);
  \draw [black,] (node_5) ..controls (91.132bp,-44.886bp) and (127.31bp,-40.795bp)  .. (node_6);
  \draw [black,] (node_1) ..controls (14.954bp,-35.402bp) and (50.246bp,-42.887bp)  .. (node_5);
  \draw [black,] (node_3) ..controls (-152.56bp,66.127bp) and (-168.32bp,93.108bp)  .. (node_12);
  \draw [black,] (node_13) ..controls (240.99bp,-1.784bp) and (270.96bp,0.33342bp)  .. (node_14);
  \draw [black,] (node_0) ..controls (-53.521bp,-1.3596bp) and (-24.772bp,-19.182bp)  .. (node_1);
  \draw [black,] (node_3) ..controls (-123.58bp,57.977bp) and (-97.858bp,73.499bp)  .. (node_4);
  \draw [black,] (node_1) ..controls (1.5417bp,-51.278bp) and (12.268bp,-81.433bp)  .. (node_8);
  \draw [black,] (node_9) ..controls (-237.51bp,59.922bp) and (-265.15bp,72.397bp)  .. (node_11);
\end{tikzpicture}

\end{tabular}

}
\caption{The connected graphs in ${\bf obs}({\cal S}^{k})$ for $n = 1,2,3$ respectively (presented left to right).}
\label{rigorous}
\end{figure}
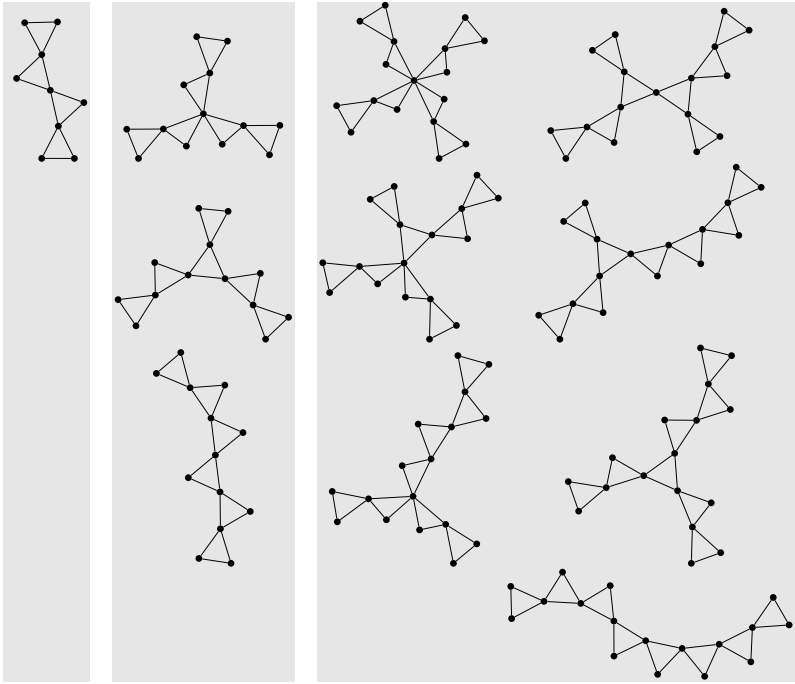

%\aliv{Έβαλα αυτό το observation.}
\begin{observation} \label{machines}
  For every $k\in\mathbb{N}$, $(k+2)K_{3}\in \cactobs{k}$.
\end{observation}

\begin{lemma} 
\label{ruthless}
Let $k\geq 0$ and $G$ be a connected cactus graph in $\cactobs{k}$. Then $G\in {\cal Z}_{k+1}$.
\end{lemma}

\begin{proof}
We proceed by induction on $k$. The base case where $k=0$ is trivial.
Let  $G$ be a connected cactus graph in $\cactobs{k}$ for some $k\geq 1$  and assume that the statement of the lemma holds for smaller values of $k$.

Let ${\cal Q}$ be a butterfly bucket in $G$ that exists because of \autoref{barracks}.  
We first examine the case where  ${\cal Q}$ is trivial.
We claim that if ${\cal Q}$ is trivial, then $\lvert {\cal Q} \lvert =k+1$. Indeed, if $|{\cal Q}|\leq k$, then 
the central vertices of the leaf buckets of ${\cal Q}$ form a $k$-apex sub-unicyclic 
set, contradicting the fact that  $G\in \cactobs{k}$. Also,  if $\lvert {\cal Q} \rvert \geq  k+2$, then $(k+2)K_{3}$ is a minor of $G$, a contradiction as $(k+2)K_{3}\in \cactobs{k}$.
The triviality of ${\cal Q}$ and the fact that $ \lvert {\cal Q} \lvert =k+1$ then imply that $G\in {\cal Z}_{k+1}$.

Suppose now that ${\cal Q}$ is not trivial. By \autoref{animates}, $G' = G \setminus (V(\cupall {\cal Q})\setminus \{w\})$  is a connected cactus in $\cactobs{k-r}$ where $r=|{\cal Q}|$. Since, due to the induction hypothesis, we have $G'\in {\cal Z}_{k-r+1}$, it follows that  $G\in {\cal Z}_{k+1}$, as required.
\end{proof}

\begin{proof}[Proof of \autoref{holiness}]
The proof is an immediate consequence of \autoref{supports} and \autoref{ruthless}.
\end{proof}

\subsection{Characterization of disconnected cactus-obstructions}

The objective of this section is to prove the following theorem.

\begin{theorem}\label{selected}
Let $k\in\Bbb{N}$, let $G$ be a disconnected cactus graph in $\cactobs{k}$, and let $G_{1}, G_{2}, \ldots, G_{r}$ be the connected components of $G$.
Then, one of the following holds:
\begin{itemize}
	\item $G\cong(k+2)K_{3}$
	
	\item there is a sequence $k_{1}, k_{2}, \ldots, k_{r}$ such that for every $i\in [r]$, $G_{i}$ is a graph in ${\cal Z}_{k_i}$ and $\sum_{i\in [r]}k_{i}= k + 1$.
\end{itemize}

\end{theorem}\medskip

We begin with the following Lemma which implies that every obstruction $G \in \cactobs{k}$ not isomorphic to $(k+2) K_3$ is also a $(k+1)$-forest. 

\begin{lemma}
	\label{supplied}
	For every $k\in\Bbb{N}$ and for every cactus graph $G$ it holds that if $(k+2) K_3\not\leq G$ then $G$  contains  a $(k+1)$-apex forest set $S$.
\end{lemma}
\begin{proof}
	Suppose, towards a contradiction, that $k$ is the minimum non-negative integer for which the contrary 
	holds. Let $G$ be a cactus graph with the minimum number of vertices such 
	that $(k+2)K_3\not\leq G$ and that for every apex-forest set $S$ of $G$ it holds that 
	$|S|> k+1$. Observe that $k\geq 1$ and that there exists some connected component $H$ of $G$ that is not isomorphic to a cycle.
	As such, let $B$ be a leaf-block of $H$ and observe that, since $G$ has the minimum number of vertices, every vertex of $G$ has degree at least 2 and therefore
	$B$ is isomorphic to a cycle.
	Since $H$ is not isomorphic to a cycle, there exists a cut-vertex $c\in V(B)$, which is unique since $B$ is a leaf-block of $H$.
	Now, consider the graph $G'=G\setminus c$ and observe that this too is a cactus.
	Observe, also, that $(k+1)K_{3}\not\leq G'$, since otherwise $(k+2)K_{3}\leq G$, a contradiction.
	Thus, by the minimality of $k$, we have that there exists a $k$-apex 
	forest set $S'$ of $G'$. But then, the set $S=S'\cup\{c\}$ is a $(k+1)$-apex forest 
	set of $G$, a contradiction to our assumption for $G$.
\end{proof}

We now proceed with the main lemma of this section.

\begin{lemma}\label{decoding}
	Let $k\geq 1$ and let $G$ be a disconnected cactus graph in $\cactobs{k}$  non-isomorphic to $(k+2)K_{3}$. 
	Let also $\{{\cal C}_{1},{\cal C}_{2}\}$ be a  partition of the connected components of $G$ and $G_{i}=\cupall {\cal C}_{i},i\in[2]$.
	Then $G_{1}\in \cactobs{k_{1} - 1}$ and $G_{2}\in \cactobs{k_{2} - 1}$ for some $k_{1},k_{2}\geq 1$ such that $k_{1}+k_{2}=k+1$.
\end{lemma}
\begin{proof}
	Clearly, since $G$ is a cactus graph, then the same holds for $G_{1}, G_{2}$.
	By \autoref{supplied}, there exists a $(k+1)$-apex forest set $S$ of  
	$G$. Notice that, since $G\not\in {\cal A}_{k}({\cal S})$, we have that $|S|=k+1$.	Also 
	observe that, by \autoref{segments}, neither of $G_{1}, G_{2}$ is a forest.	
	Let $S_{1}=S\cap V(G_{1}), S_{2}=S\cap V(G_{2})$ and let $k_{1}=|S_{1}|, k_{2}=|S_{2}|$. Note that, $k_{1},k_{2}\geq 1$ and, since 
	$V(G_{1})\cap V(G_{2})=\emptyset$, $k_{1}+k_{2}=k+1$. We argue that the following holds:
	
	\medskip
	
	\noindent{\em Claim 1:} For each $i\in[2]$, $G_{i}\not\in \kapexmo{k_{i} - 1}$. \smallskip
	
	\noindent{\em Proof of Claim 1:}
	Suppose, towards a contradiction, that $G_{i}\in \kapexmo{k_{i} - 1}$ for some 
	$i\in[2]$. Then, there exists a $(k_{i}-1)$-apex sub-unicyclic set $X_{i}$ of 
	$G_{i}$. But then, the set $X_{i}\cup S_{j}$, where $j\neq i$, is a $k$-apex 
	sub-unicycle set of $G$, a contradiction to the fact that $G\in \cactobs{k}$. Claim 1 follows.
	
	\medskip
	
	\noindent{\em Claim 2:}	For each $i\in[2]$, it holds that if $H_{i}$ is a proper 
	minor of $G_{i}$ then $H_{i}\in \kapexmo{k_{i} - 1}$.
	\smallskip
	
	\noindent{\em Proof of Claim 2:} Suppose, towards a contradiction, that for some 
	$i\in[2]$ there exists a proper minor $H_{i}$ of $G_{i}$ such that 
	$H_{i}\not\in{\cal A}_{k_{i}-1}({\cal S})$. Let $H= H_{i}\cup G_{j}$, where $j\neq i$. 
	As $G\in \cactobs{k}$, there exists a $k$-apex sub-unicyclic set $X$ of 
	$H$. Let $X_{i}=X\cap H_{i}$ and $X_{j}=X\cap G_{j}$. Then, as 
	$H_{i}\not\in {\cal A}_{k_{i}-1}{\cal S}$, we have that $|X_{i}|\geq k_{i}$ and 
	therefore the fact that $|X|\leq k$ implies that $|X_{j}|=|X|-|X_{i}|\leq 
	k-k_{i}=k_{j}-1$. Hence, the set $X_{j}\cup S_{i}$ is a $k$-apex sub-unicyclic 
	set of $G$, a contradiction to the fact that $G\in \cactobs{k}$. Claim 2 follows.
	
	\medskip
	
	Claim 1 and Claim 2 imply that $G_{1}\in \cactobs{k_{1} - 1}$ and $G_{2}\in \cactobs{k_{2} - 1}$, which concludes the proof of the Lemma.
\end{proof}

\begin{proof}[Proof of \autoref{selected}]
	The proof follows by \autoref{machines} and by repeated applications of \autoref{decoding}, as required.
\end{proof}

\section{Enumeration of cactus obstructions}
\label{stripop}

Let $\mathcal{G}=\textbf{obs}({\mathcal A}_{k}({\mathcal S}))\cap {\mathcal K}$. In this section, we determine the asymptotic growth of $g_k=|\textbf{obs}({\mathcal A}_{k}({\mathcal S}))\cap {\mathcal K}|$ and $z_k=|\mathcal{Z}_k| $.
To that end, we make use of the \emph{Symbolic Method} framework and the corresponding analytic techniques, as developed in~\cite{flajolet2009analytic}. 
%We begin with a brief exposition of the combinatorial framework that we use and notation.
\paragraph{The Symbolic Method.} A \emph{combinatorial class} is a tuple $(\mathcal{A},f)$, where $\mathcal{A}$ is a set and $f$ is a \emph{size} function $f:\mathcal{A}\rightarrow \mathbb{N}^*$. When the nature of $f$ is clear, we refer to $(\mathcal{A},f)$ as $\mathcal{A}$ for convenience. Let $a_k=|\{a\in \mathcal{A}: f(a)=k\}|$. The \emph{generating function}, or simply \emph{GF}, $A(z)$ is defined as the power series $A(z)=\sum _{n\geq 0}a_nz^n$.\footnote{By convention, we denote combinatorial classes by calligraphic uppercase letters, their GFs by the same uppercase letters in plain form, and use subscripted lowercase letters to denote the coefficients of the GFs.}
We also employ the notation $[z^k] A(z) = a_k$ to refer to the $k$-th coefficient of some GF $A(z)$.
Two combinatorial classes $\mathcal{A}, \mathcal{B}$ are \emph{isomorphic}, written as $\mathcal{A} = \mathcal{B}$, if $a_k=b_k$ for all $k$. The simplest combinatorial class is the \emph{atomic} one, denoted by $\mathcal{X}$, with GF $X(x)=x$.

The Symbolic Method allows us to translate operations between combinatorial classes into functional operations between their generating functions. In particular, we shall make use of the operations of \emph{disjoint sum} ($+$), \emph{cartesian product}, \emph{multiset} ($MSET$), \emph{2-multiset} ($MSET_2$), i.e. multisets of two objects. Each of these operations defines a new set upon given ones, in the obvious way. The size function upon the new class is an additive function over the size of the objects that compose the new object. For instance, the size of a multiset $b\in MSET(\mathcal{A})$ equals the sum of the sizes of all objects in $b$. The functional relations corresponding to each of these operations can be seen in~\autoref{hegelian}. We refer to~\cite[Chapter I]{flajolet2009analytic} for details.

\begin{table}[]
	\renewcommand{\arraystretch}{1.5}
	\centering
	\begin{tabular}{cc|c}
		Set Operation & & GF operation\\ \hline
		Disjoint sum & $\mathcal{A}=\mathcal{B}+\mathcal{C}$ & A(z)=B(z)+C(z)\\
		Cartesian Product & $\mathcal{A}=\mathcal{B}\mathcal{C}$ & A(z)=B(z)C(z)\\
		Multiset & $\mathcal{A}=MSET(\mathcal{B})$ & $A(z)=\exp \big(\sum _{k\geq 1}\frac{B(z^k)}{k}\big)$ \\
		2-multiset & $\mathcal{A}=MSET_2(\mathcal{B})$ & $A(z)=\frac{A(z)^2+A(z^2)}{2}$
	\end{tabular}
	\caption{Operations between combinatorial classes and their counterparts in terms of generating functions.}
	\label{hegelian}
\end{table}

\paragraph{Rooted trees.} A \emph{tree} is a connected graph for which it holds that $\lvert V(G) \lvert -1 = \lvert E(G) \lvert$. Let $\mathcal{T}$ be a family of trees. We define the family of trees in $\mathcal{T}$ {\em rooted at a vertex}, denoted by $\mathcal{T}^\bullet$, to be all tuples $(T,v)$, where $T \in \mathcal{T}$ and $v \in V(T)$. We define the family of trees in $\mathcal{T}$ {\em rooted at an edge}, denoted by $\mathcal{T}^{\bullet -\bullet}$, to be all tuples $(T,e),$ where $ T \in \mathcal{T},$ and  $e \in E(G)$. Finally, we define the family of trees in $\mathcal{T}$ {\em rooted at an oriented edge}, denoted by $\mathcal{T}^{\bullet \rightarrow\bullet}$, to be all tuples $(T,\overrightarrow{e})$, where $T \in \mathcal{T}$, $\overrightarrow{e} = (a,b) \in V(T)^2$, and $ab \in E(G)$. We say that two rooted trees $(T_1, r_1), (T_2, r_2)$ are isomorphic when there exists a graph isomorphism between $T_1, T_2$ that maps $r_1$ to $r_2$. When no confusion can arise, we will refer to a rooted tree $(T,r)$ simply as $T$. By the well-known \emph{Dissymmetry Theorem for Trees} (see~\cite{bergeron1998combinatorial}), it holds that
\begin{equation}\label{supplier}
\mathcal{T}+\mathcal{T}^{\bullet \rightarrow \bullet}=\mathcal{T}^\bullet+\mathcal{T}^{\bullet - \bullet}.
\end{equation}

\subsection{A bijection of $\mathcal{Z}$ with a family of trees.} 
We begin by giving a bijection between the combinatorial class $\mathcal{Z}$, i.e., connected graphs in $\textbf{obs}({\mathcal A}_{k}({\mathcal S}))\cap {\mathcal K}$ with size equal to the number of butterfly-subgraphs, and the following family of trees. Let $\mathcal{T}$ be the family of trees having three different types of vertices, namely $\mathsmaller{\square}$-, $\vartriangle$-, and $\circ$-vertices, and meeting the following conditions:
\begin{enumerate}
\item The neighbourhood of a $\mathsmaller{\square}$-vertex consists of two $\vartriangle$-vertices.
\item The neighbourhood of a $\vartriangle$-vertex consists of a $\mathsmaller{\square}$-vertex and two $\circ$-vertices.
\item The neighbourhood of a $\circ$-vertex consists of one or more $\vartriangle$-vertices.
\end{enumerate}

%\begin{figure}[h]
%\centering
%\resizebox{0.2\linewidth}{!}{
%\includegraphics{t1structure.pdf}
%}
%\caption{The smallest member of $\mathcal{T}$.}
%\label{smallestt}
%\end{figure}

%Observe that from the above specification, it follows that no vertex can have a neighbour of the same type. Also observe that $\mathcal{T}$ is non-empty, with its smallest, in terms of vertex-count, element being depicted in \autoref{smallestt}.

Consider the combinatorial class $(\mathcal{T},f)$ where $f$ assigns to a tree size equal to the number of its $\mathsmaller{\square}$-vertices. Then the following holds.

%\begin{figure}
%\begin{center}
%\minipage{0.32\textwidth}
%\resizebox{0.8\linewidth}{!}{
%  \includegraphics{petalouda.pdf}}
%\endminipage~~~
%\minipage{0.35\textwidth}
%\resizebox{\linewidth}{!}{
 % \includegraphics{correspondingTree.pdf}}
%\endminipage
%\end{center}
%\caption{A graph in $\mathcal{Z}$ (left) and its image in $\mathcal{T}$ under the bijection described in \autoref{required} (right).}
%\label{marching}
%\end{figure}
%
\begin{lemma}
\label{required}
The combinatorial classes $\mathcal{Z}$ and $\mathcal{T}$ are isomorphic.
\end{lemma}
\begin{proof}
We shall construct a bijection $\phi: \mathcal{Z} \to \mathcal{T}$ that preserves the size function.
%\gstam{define "preserves the size"} 
Given a graph $G \in \mathcal{Z}$ whose butterfly-subgraphs we denote as $\{Z_1, \dots, Z_k\}$, let $\phi(G) = T \in \mathcal{T}$ be a tree constructed as such:
\begin{itemize}
\item To every central vertex $c_i$ of $Z_i$ we associate a $\mathsmaller{\square}$-vertex $s_i$.
\item To every vertex $v\in G$ which is not a central vertex of some $Z_i$, we associate a $\circ$-vertex $o_v$. 
\item To each $K_3$-subgraph $T_{i_1}, T_{i_2}$ of $Z_i$, we associate a $\vartriangle$-vertex $t_{i_1}, t_{i_2}$, respectively. The neighbourhood of $t_{i_j}$ consists of the $\mathsmaller{\square}$-vertex $c_i$ and the two $\circ$-vertices associated to the vertices of $T_{i_j} \setminus c_i$. 
\end{itemize} 
Observe that the described tree belongs in $\mathcal{T}$ and that $\phi $ is a bijection (see also Figure~\ref{marching}). Also, notice that $\phi$ is also size-preserving, since the number of butterfly-subgraphs of $Z\in \mathcal{Z}$ equals the number of $\mathsmaller{\square}$-vertices in $\phi(Z)$. 
% Observe that the above argument can also be read as a way to construct a $G \in \mathcal{Z}$ from a given $T \in {\cal T}$ such that $\phi(G) = T$, hence $\phi$ is a surjection.
%Let $G, G' \in G$ and suppose that $\phi(G) = \phi(G') = T$. Then, each $\mathsmaller{\square}$-vertex $s_i$ of $T$ corresponds to exactly one central vertex $c_i$ of some $Z_i$, every $\circ$-vertex $o_u$ corresponds to a unique, non-central to any butterfly-subgraph, vertex $u$ of $G$, each triplet $\{ c_i, u, v \}  \in V(G)^3$  forms a triangle $G[c_i,u,v]$ if and only if the corresponding triplet $\{s_i, o_u, o_v\} \in V(T)^3$ is the neighbourhood of $t_{i_{1}}$ or $t_{i_{2}}$, and $G$ has no edges other than those forming these triangles. But the same must hold for $G'$ and therefore we have that $G \cong G'$ and so $\phi$ is injective. 
\end{proof}
\begin{figure}[!h]
\begin{center}
\minipage{0.4\textwidth}
\resizebox{0.8\linewidth}{!}{
  \includegraphics{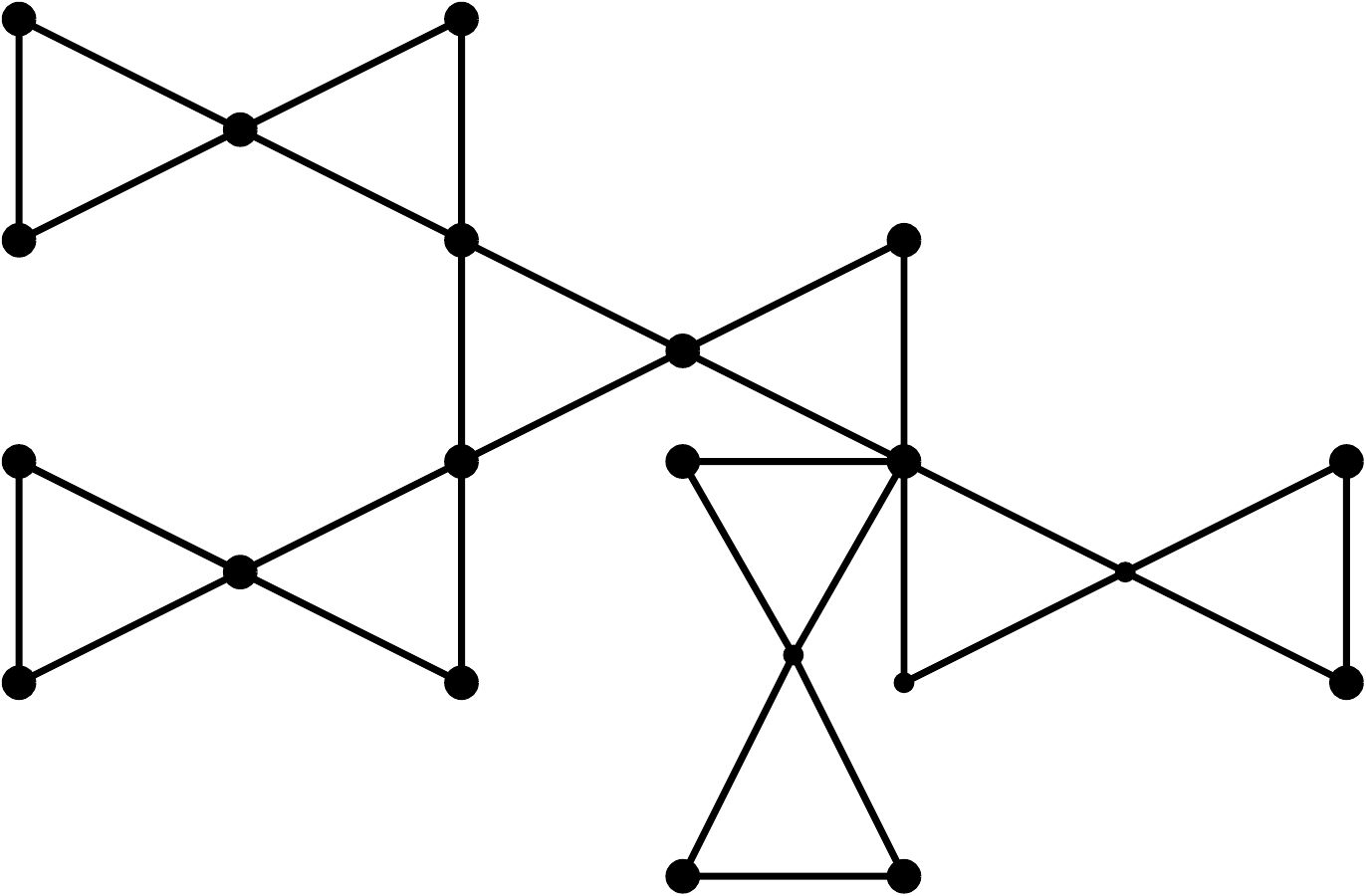}}
\endminipage~~~
\minipage{0.35\textwidth}
\resizebox{1.2\linewidth}{!}{
  \includegraphics{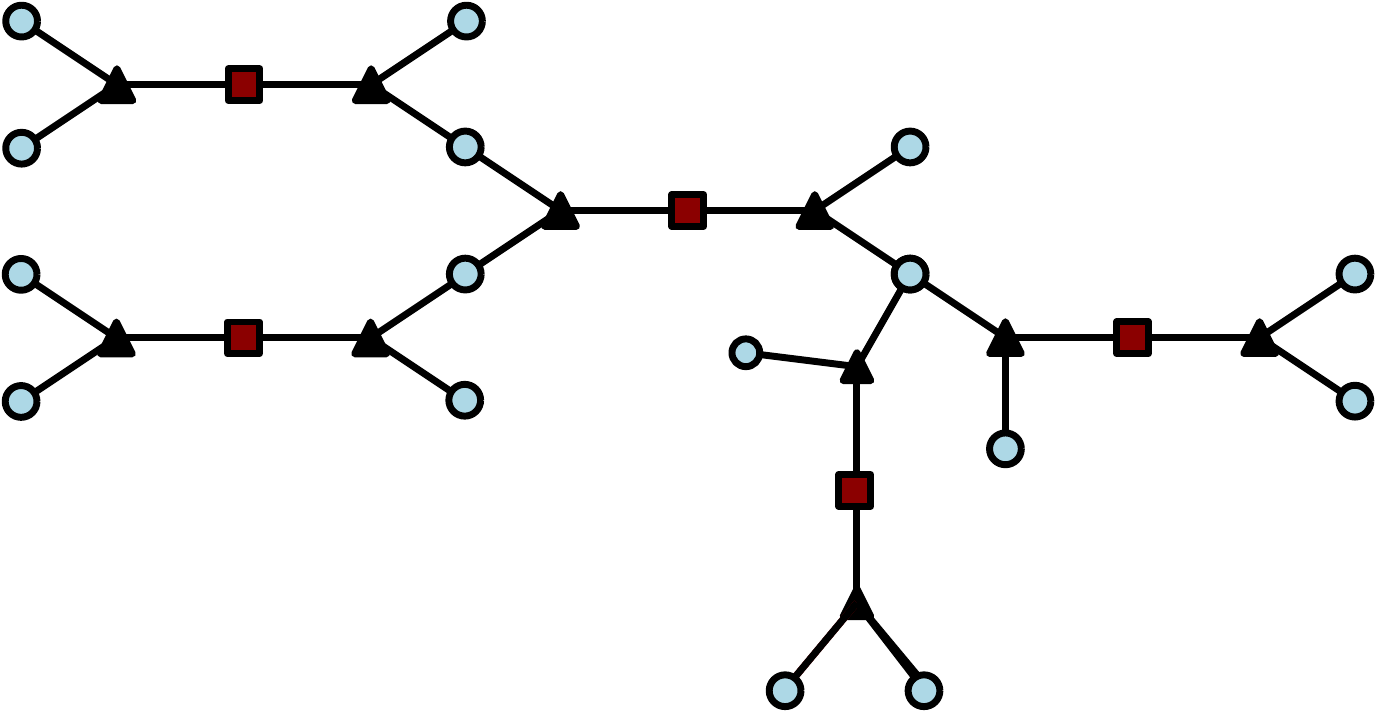}}
\endminipage
\end{center}
\caption{A graph in $\mathcal{Z}$ and its image in $\mathcal{T}$, under the bijection described in \autoref{required}.}
\label{marching}
\end{figure}

By Lemma~\ref{required}, enumerating $\mathcal{Z}$ is equivalent to enumerating $\mathcal{T}$. To that end, we will make use of the combinatorial classes $\mathcal{T}^{\mathsmaller{\mathsmaller{\square}}} ,\mathcal{T}^\vartriangle ,\mathcal{T}^\circ ,  \mathcal{T}^{\circ- \vartriangle} , \mathcal{T}^{\mathsmaller{\mathsmaller{\square}}- \vartriangle}, \mathcal{T}^{\mathsmaller{\mathsmaller{\square}} \rightarrow \vartriangle} , \mathcal{T}^{\vartriangle \rightarrow \mathsmaller{\square}} , \mathcal{T}^{\vartriangle \rightarrow \circ} , \mathcal{T}^{\circ \rightarrow \vartriangle},
$ which correspond to trees in $\mathcal{T}$ that are rooted in the indicated way.

\begin{lemma}\label{shimmers}The following functional relations hold for $T(z)$ and $G(z)$:
\begin{align}\label{reliably}
T(x) & =  T^{\mathsmaller{\mathsmaller{\square}}}(x)+T^\vartriangle(x)+T^\circ(x)-T^{\mathsmaller{\mathsmaller{\square}}\rightarrow \vartriangle}(x)-T^{\vartriangle\rightarrow \circ}(x)\\ 
 G(x) & =  \exp \bigg(\sum _{k\geq 1}\frac{T(x^k)}{k}\bigg). \nonumber
\end{align}
\end{lemma}
\begin{proof}
It is clear that
\begin{align*}
\mathcal{T}^\bullet &= \mathcal{T}^{\mathsmaller{\mathsmaller{\square}}}+\mathcal{T}^\vartriangle+\mathcal{T}^\circ\\
 \mathcal{T}^{\bullet \rightarrow \bullet} &=  \mathcal{T}^{\mathsmaller{\mathsmaller{\square}} \rightarrow \vartriangle}+\mathcal{T}^{\vartriangle \rightarrow \mathsmaller{\square}}+\mathcal{T}^{\vartriangle \rightarrow \circ}+\mathcal{T}^{\circ \rightarrow \vartriangle}\\
 \mathcal{T}^{\bullet - \bullet} &=  \mathcal{T}^{\circ- \vartriangle} +\mathcal{T}^{\mathsmaller{\mathsmaller{\square}}- \vartriangle}\\
\mathcal{T}^{\circ- \vartriangle} &= \mathcal{T}^{\circ \rightarrow \vartriangle}\\ \mathcal{T}^{\mathsmaller{\mathsmaller{\square}}- \vartriangle}&= \mathcal{T}^{\vartriangle\rightarrow \mathsmaller{\mathsmaller{\square}} }
\end{align*} 
Then, the first relation follows by substituting the above in Equation~\ref{supplier} and translating to GFs. The second relation follows by noticing that $\mathcal{G}=MSET(\mathcal{Z})=MSET(\mathcal{T})$.
\end{proof}
To obtain defining systems for $T^{\mathsmaller{\mathsmaller{\square}}} (x),T^{\vartriangle} (x),T^\circ (x),T^{\mathsmaller{\mathsmaller{\square}}\rightarrow \vartriangle}(x),T^{\vartriangle\rightarrow \circ}(x)$, we define the auxiliary combinatorial classes $\mathcal{T}_{\diamond}$ and $\mathcal{T}_{\star}$. $\mathcal{T}_{\diamond}$ contains trees in $\mathcal{T}$ rooted at a leaf and $\mathcal{T}_{\star}$ contains multisets of trees in $\mathcal{T}_{\diamond}$.
\begin{lemma}\label{ulterior}
The generating functions 
$\mathcal{T}_{\star},\mathcal{T}_{\diamond},T^{\mathsmaller{\mathsmaller{\square}}} ,T^\vartriangle ,T^\circ ,T^{\mathsmaller{\mathsmaller{\square}}\rightarrow \vartriangle},T^{\vartriangle\rightarrow \circ}$
are defined through the following system of functional equations.
\begin{align*}
T_{\diamond}(x) &=  \frac{x}{2}\exp \bigg( \sum _{k\geq 1}\frac{T_{\diamond}(x^k)}{k}\bigg) \bigg(\exp \bigg( \sum _{k\geq 1}\frac{2\,T_{\diamond}(x^k)}{k}\bigg)+\exp \bigg( \sum _{k\geq 1}\frac{T_{\diamond}(x^{2k})}{k}\bigg)\bigg)\\
T_{\star} (x)&=  \exp \bigg( \sum _{k\geq 1}\frac{T_{\diamond}(x^k)}{k}\bigg)\\
T^\circ (x) &= \exp\bigg(\sum _{k\geq 1} \frac{T_{\diamond}(x^k)}{k}\bigg)-1\\
T^{\mathsmaller{\mathsmaller{\square}}} (x) &= \frac{x}{8}\,  T_{\star} \left( x \right)  ^{4}+ \frac{x}{4}  T_{\star} ( 
x)  ^{2}T_{\star} (
x^2 )+ \frac{3x}{8}\, T_{\star} ( x^2 )^2
+ \frac{x}{4}\,T_{\star} (x^4)
\\
T^\vartriangle (x) &= \frac{x}{4}\,  T_{\star}  ( x )  ^{4}+ \frac{x}{2}\, T _{\star} ( 
x )  ^{2}T_{\star}  ( {x}^{2} ) + \frac{x}{4} T_{\star} 
( {x}^{2} )  ^{2}
\\
T^{\mathsmaller{\mathsmaller{\square}}\rightarrow \vartriangle} (x) &=  \frac{x}{4}\,  T_{\star}  ( x )  ^{4}+ \frac{x}{2}\, T_{\star}  ( 
x )  ^{2}T_{\star}  ( {x}^{2} ) + \frac{x}{4} T_{\star} 
( {x}^{2} )  ^{2}\\
T^{\vartriangle\rightarrow \circ} (x) &=  \frac{x}{2}\, T_{\star} \left( x \right)   ^{4}+\frac{x}{2}\,T_{\star}(x)^2\,T_{\star} ( {x}^{2}
)
\end{align*}
\end{lemma}
\begin{proof}
We obtain the indicated functional equations by establishing the following combinatorial bijections in the language of the Symbolic method. The result follows by applying their translation to GF relations. 
\begin{align}
\mathcal{T}^\diamond  &=  \mathcal{X}MSET(\mathcal{T}_{\diamond})MSET_2(MSET(\mathcal{T}_{\diamond}))  \label{wrapping} \\
\mathcal{T}_{\star} &= MSET(\mathcal{T}_{\diamond}) \label{residual} \\
\mathcal{T}^\circ &= MSET(\mathcal{T}_{\diamond}  )-1 \label{hundreds} \\
\mathcal{T}^{\mathsmaller{\mathsmaller{\square}}} &= \mathcal{X}MSET_2(MSET_2(\mathcal{T}_{\star} )) \label{fissures} \\
\mathcal{T}^\vartriangle &= \mathcal{X}MSET_2(\mathcal{T}_{\star}  )^2 \label{obtained} \\
\mathcal{T}^{\mathsmaller{\mathsmaller{\square}}\rightarrow \vartriangle} &= \mathcal{X}MSET_2(\mathcal{T}_{\star} )^2 \label{generals} \\
\mathcal{T}^{\vartriangle\rightarrow \circ} &= \mathcal{X}(\mathcal{T}_{\star} )^2MSET_2(\mathcal{T}_{\star}  ). \label{daughter} 
\end{align}
%
%We proceed by proving each of the equations presented above.

\noindent We now justify each of the equations presented above.

\autoref{residual} immediately follows from the definition of $\mathcal{T}_{\star}$.

 Let $T\in\mathcal{T}$ and let $s$ be a $\mathsmaller{\square}$-vertex. We define its {\em extended neighbourhood} to be the set $\{o_1, o_2, o_3, o_4, t_1, t_2\}$ where $\{t_1, t_2\}$ is the neighbourhood of $s$, $o_1, o_2$ are the $\circ$-neighbours of $t_1$, and $o_3, o_4$ are the $\circ$-neighbours of $t_2$. Given some $s$ whose extended neighbourhood contains the $\circ$-vertices $o_i$, $i \in [4]$, we define $S_i = \{G \ \lvert  \ G \in \mathcal{C}(T, o_i) \land s \notin V(G) \}$, $i \in [4]$. Observe that each $S_i$ is a multiset of graphs which, rooted at $o_i$, are elements of $\mathcal{T}_{\diamond}$. Hence, $S_i\in MSET(\mathcal{T}_{\diamond})$, which equals $S_i\in\mathcal{T}_{\star}$. 

\autoref{wrapping}: Let $T \in \mathcal{T}_\diamond$, $o_1$ be its root vertex, and $s$ be the closest $\mathsmaller{\square}$-vertex to $o_1$ in $T$. Consider the extended neighbourhood defined by $s$ and the corresponding multisets $S_i$. Notice that if $S_3$ is exchanged with $S_4$, then the resulting tree remains the same
%\gstam{define "same". Maybe you mean "isomorphic"?}. 
However, this does not hold when $S_2$ is exchanged with $S_3$ or $S_4$. Hence, $T$ defines uniquely (and is defined by) an object of $T_\star$ and a 2-set of objects in $T_\star$. The object $\mathcal{X}$ of the atomic class accounts for the vertex $s$, whose size equals 1.

%We will show that there exists a bijection $f_1$ between the respective sets of each class. Let $T \in \mathcal{T}_\diamond$ be any graph in $\mathcal{T}$ rooted at a leaf $o_1$ and observe that $o_1$ is an $\circ$-vertex with some unique $\vartriangle$-vertex neighbour $T_\diamond$. Denote by $s$ the $\mathsmaller{\square}$-vertex neighbour of $T_\diamond$ and by $t_2$ the $\vartriangle$-vertex neighbour of $s$ other than $T_\diamond$. Furthermore, denote by $o_2$ the neighbour of $T_\diamond$ other than $o_1, s$ and by $o_3, o_4$ the neighbours of $t_2$ other than $s$. Let $S_i = \{G \ \lvert  \ G \in \mathcal{C}(T, o_i) \land T_\diamond \notin V(G) \}$, $i \in \{2,3,4\}$. We then define $f(T) = (s, S_2, MSET_2(S_3, S_4))$. 
%From the above construction, one easily observes that $f$ is an injection. To prove that $f$ is a surjection, we need only observe that for any tuple $t = (s, S_a, MSET_2(S_b, S_c))$ there exists a $H \in \mathcal{T}_\diamond$ which is the preimage of $t$ under $f$, constructed as follows: identify all roots of $S_a$ into some $\circ$-vertex $o_2$, all roots of $S_b$ into some $\circ$-vertex $o_3$, and all roots of $S_c$ into some $\circ$-vertex $o_4$. Then $H$ is the tree obtained by considering a new root $\circ$-vertex $o_1$ and two  new $\vartriangle$-vertices whose neighbourhoods are $\{o_1, o_2, s\}$ and $\{o_3, o_4, s\}$ respectively. 

%Since largely similar arguments can be used to prove that the functions defined in the sequel are all bijections, we will omit the details and only describe the required constructions.

\autoref{hundreds}: Let $o_1$ be the root of $T\in \mathcal{T}^\circ$. Observe that $\mathcal{C}(T,o_1) \neq \emptyset$ (this fact corresponds  to the term -1 in \autoref{hundreds}) and all elements of $\mathcal{C}(T,o_1)$, rooted at $o_1$, belong in $T_{\diamond}$. Hence, $T$ is uniquely defined by (and defines) a multiset of elements in $T_{\diamond}$.
%Let $T \in \mathcal{T}^\circ$ be an element of $\mathcal{T}$ rooted at a $\circ$-vertex $o_1$. Note that $\mathcal{C}(T,o_1)$ is always non-empty, since even when $o_1$ is not a cut-vertex of $T$, there exists at least element in $\mathcal{C}(T,o_1)$ (namely $T$). We define $f_4(T) = \mathcal{C}(T,o_1)$ and observe that it provides the desired bijection.

\autoref{fissures}: Let $T \in \mathcal{T}^{\mathsmaller{\mathsmaller{\square}}}$ with root $s$. Consider the multisets $S_i$, $i \in [4]$, as defined using the extended neighbour of $s$. Observe that $S_1, S_2$ can be exchanged to give the same tree, and the same holds for $S_3, S_4$, so we can consider them as two 2-multisets. Moreover, one can exchange the pair $S_1, S_2$ with the pair $S_3, S_4$ to obtain the same tree. Therefore, the desired relation holds, where $\mathcal{X}$ accounts for $s$.

%Let $T \in \mathcal{T}^{\mathsmaller{\mathsmaller{\square}}}$ be an element of $\mathcal{T}$ rooted at a $\mathsmaller{\square}$-vertex $s$. Let, also, $t_1, t_2$ be the two neighbours of $s$, $o_1, o_2$ be the $\circ$-vertex neighbours of $T_\diamond$, and $o_3, o_4$ be the $\circ$-vertex neighbours of $t_2$. We define $S_i = \{G \ \lvert  \ G \in \mathcal{C}(T, o_i) \land T_\diamond \notin V(G) \}$, $i \in [4]$. We then let $f_2(T) = (s, MSET_2(MSET_2(S_1, S_2), MSET_2(S_3,S_4)))$ and observe that $f$ indeed provides the desired bijection.

%\autoref{obtained}: Let $T \in \mathcal{T}^\vartriangle$ be an element of $\mathcal{T}$ rooted at a $\vartriangle$-vertex $T_\diamond$. Denote by $s$ be the $\mathsmaller{\square}$-vertex neighbour of $T_\diamond$ and by $t_2$ be the $\vartriangle$-vertex neighbour of $s$ other than $T_\diamond$. As above, let $o_1, o_2$, be the $\circ$-vertex neighbours of $T_\diamond$ and $o_3, o_4$  be the $\circ$-vertex neighbours of $t_2$.  Define $S_i = \{G \ \lvert  \ G \in \mathcal{C}(T, o_i) \land T_\diamond \notin V(G) \}$, $i \in [4]$. Note that care must be taken to distinguish the multisets $S_1, S_2$ of subtrees rooted at children of the root $T_\diamond$ from those multisets $S_3, S_4$ of subtrees rooted at the other two $\circ$-vertices $o_3, o_4$. We then let $f_3(T) = (s, MSET_2(S_1, S_2), MSET_2(S_3, S_4))$ and observe that $f$ indeed provides a bijection which proves that.

\autoref{obtained}: Let $T \in \mathcal{T}^\vartriangle$ and $s$ the $\mathsmaller{\square}$-vertex connected to the root. Consider the extended neighbourhood of $s$ and the sets $S_i$, $i \in [4]$. $S_1$ and $S_2$ are exchangeable, as well as $S_3$ and $S_4$. However, as pairs, they cannot be exchanged to give the same graph. Hence, $\mathcal{T}^{\mathsmaller{\vartriangle}}$ is equivalent to the cartesian product of two 2-multisets of objects in $\mathcal{T}_\star$. $\mathcal{X}$ accounts for the vertex $s$. 

\autoref{generals}: Holds by arguments similar to the ones used to prove \autoref{obtained}.

%Let $T \in \mathcal{T}^{\mathsmaller{\mathsmaller{\square}} \rightarrow \vartriangle}$ be an element of $\mathcal{T}$ rooted at an edge whose endpoints are a $\mathsmaller{\square}$-vertex $s$ and a $\vartriangle$-vertex $T_\diamond$. We then construct a bijection $f_3$ in a manner identical to the one given in the case of $\mathcal{T}^\vartriangle$. Again, care must be taken to distinguish between the 2-multisets containing elements of $\mathcal{T}_{\star}$ rooted at the $\circ$-neighbours of $T_\diamond$, from those rooted at the $\circ$-neighbours of $t_2$, where $t_2$ is the neighbour of $s$ other than $T_\diamond$.

\autoref{daughter}: Let $T \in \mathcal{T}^{\vartriangle \rightarrow \circ}$, $(t_1, o_1)$ be its root, and $s$ be the $\mathsmaller{\square}$-vertex closest to $t_1$. Consider the multisets $S_i$, $i \in [4]$, as defined using the extended neighborhood of $s$. Note that one may exchange $S_3, S_4$ to obtain the same tree $T$. Note, also, that one may not exchange $S_1$ with $S_2$, since then one obtains a different tree.%\gstam{Define different?} 
The desired relation then follows, with the $\mathcal{X}$ factor accounting for $s$.\end{proof}
%The statement then follows by the applying the translation relations dictated by the Symbolic Method.
%
\noindent By the defining systems of $T(x)$ and $G(x)$, we can obtain the first terms of the series:
\begin{align*}
T(x) &= x+{x}^{2}+3\,{x}^{3}+7\,{x}^{4}+25\,{x}^{5}+88\,{x}^{6}+366\,{x}^{7}+
1583\,{x}^{8}+7336\,{x}^{9}+34982\,{x}^{10}+\cdots\\
G(x) &= 1+z+2\,{x}^{2}+5\,{x}^{3}+13\,{x}^{4}+41\,{x}^{5}+143\,{x}^{6}+558\,{
x}^{7}+2346\,{x}^{8}+10546\,{x}^{9}+49397\,{x}^{10}+\cdots 
\end{align*}

\subsection{Asymptotic Analysis} 

Having set up a system of functional equations for the generating functions $Z(x)$ and $G(x)$, we can determine the asymptotic growth of $z_k$ and $g_k$ via the process of \emph{Singularity Analysis}. We briefly mention the main tools we will use and refer to~\cite{flajolet2009analytic} for details. 

We call \emph{dented domain} at $x=\rho $ a set of the form $\{x\in \mathbb{C}\mid |x|<R,\,\arg(x-\rho)\notin [-\theta,\theta]\}$, for some $R>\rho$ and $0<\theta<\pi /2$.
%\gstam{Maybe a figure showing a dented domain? }
%\vv{Not necessary, since we do not have to think of this notion anywhere! Drmota gives us everything for free.}
% and we denote it by $\Delta(\theta,R)$. 
 Let $f(x)=\sum _{k\geq 0}f_kx^k$ a GF analytic in a dented domain at $x=\rho $ that satisfies an expansion of the form %\gstam{Maybe first " We call singular expansion..." and then "Let f... that satisfies a singular expansion"?}
\begin{equation*}
f(x) = F_0 +F_1X+F_2 X^2 + F_3 X^3 + \cdots  + F_{2k}X^{2k} + F_{2k+1}X^{2k+1} + O\left(X^{2k+2}\right)
\end{equation*}
locally around $\rho $, where $X = \sqrt{1 - x/\rho}$. We call \emph{singular exponent} the smallest odd exponent of $X$ divided by two, and denote it by $\alpha$. If $f_k>0$ for all $k$ big enough, then we can apply the so-called  \emph{Transfer Theorems} of singularity analysis~\cite[Corrollary VI.1, Theorem VI.4]{flajolet2009analytic} and obtain \begin{equation}\label{fanfares}
[x^n]f(x) \sim c \cdot n^{-\alpha-1}\cdot \rho^{-n},
\end{equation}
%\gstam{instead of $[z^n]f(x)$, maybe $[x^n]f(x)$ or $F_{n}$?}
where $c = \frac{F_{2\alpha}}{\Gamma(-\alpha)}$ and $\Gamma$ is the standard Gamma function. To obtain such expansions, we will use the following Theorem.
\begin{theorem}[{\hspace{-.02cm}\cite[Proposition 1, Lemma 1]{drmota1997systems}}]\label{ridicule}
Suppose that $F(x, y)$ is an analytic function in $x,y$ such that $F(0, y)\equiv 0$, $F(x, 0)\not\equiv 0$, and all Taylor coefficients of $F$ around $0$ are real and nonnegative. Then, the unique solution $y=y(x)$ of the functional equation $y=F(x,y)$ with $y(0)=0$ is analytic around $0$ and has nonnegative Taylor coefficients $y_k$ around $0$.   Assume that the region of convergence of $F(x,y)$ is large enough such that there exist nonnegative solutions $x=x_0$ and $y=y_0$ of the system of equations 
\begin{align} y &= F(x,y),\\ 1 &= F_y(x,y),\end{align}
where $F_x(x_0,y_0)\neq0$ and $F_{yy}(x_0,y_0)\neq 0$.\footnote{Here, and in the sequel, subscripts will denote partial differentiation with respect to the subscripted variable(s).} Assume also that $y_k>0$ %\gstam{What is $y_{k}$?}
 for large enough $k$. Then, $\rho $ is the unique singularity of $f$ on its radius of convergence and there exist functions $q(x), h(x)$ which are analytic around $x=x_0$, such that $y(x)$ is analytically continuable in a dented domain at $\rho $
%$in the region $|x|\leq |\rho|+\epsilon$, $\arg (x-\rho)\neq 0$ 
and, locally around $x=\rho$, it has a representation of the form 
\begin{equation} y(x)=
q(x)+h(x)\sqrt{\bigg( 1-\frac{x}{\rho} \bigg)}.\label{etiology} \end{equation}
\end{theorem}
In the proof of the latter Theorem, an explicit way is given to compute the coefficients $q_i,h_i$.
%\gstam{maybe $q_{i}, h_{i}$?} 
Using a computer algebra program like \texttt{Maple}, we can easily obtain:
\begin{eqnarray}\label{analysts}
h_0 &=& \sqrt{\frac{2\rho F_x(x_0,y_0)}{F_{yy}(x_0,y_0)}}, \quad h_1 \;\;  = \;\; \; \frac{1}{6} \frac{-F_{yyy}(x_0,y_0) h_0^2 + 6 F_{xy}(x_0,y_0) \rho}{2F_{yy}(x_0,y_0)},\\  
q_1 &=&  -\frac{1}{24} \frac{F_{yyyy}(x_0,y_0) h_0^4 -12 F_{xyyy}(x_0,y_0) h_0^2 \rho + 12 F_{yyy} (x_0,y_0)h_1 h_0^2 }{F_{yy}(x_0,y_0) h_0}+ \\
& & +\frac{ 12 F_{xx}(x_0,y_0) \rho^2 - 24 F_{xy}(x_0,y_0) h_1 \rho + 12 F_{xx}(x_0,y_0) h_1^2}{F_{yy}(x_0,y_0) h_0}. \nonumber\label{chrysler}
\end{eqnarray}

\begin{lemma}\label{stressed}
The generating functions $T_\diamond,T^{\mathsmaller{\mathsmaller{\square}}}, T^\vartriangle, T^{\circ} , T^{\mathsmaller{\mathsmaller{\square}}\rightarrow \vartriangle}, T^{\vartriangle\rightarrow \circ}$ have a unique singularity of smallest modulus, at the same positive number $\rho<1$.  Moreover, they are analytic in a dented domain at $\rho $ and satisfy expansions of the form
\begin{equation*}
A_0+\sum _{k\geq 1}A_kX^k,\quad \text{where}\quad X= \sqrt{1 - x/\rho},
\end{equation*}
locally around $\rho$. The coefficients $A_i$ and $\rho $ are computable; in particular, $\rho\approx 0.15926$.\end{lemma}
 \begin{proof}
 Let $\rho _\diamond <1$ the positive radius of convergence of $T_\diamond$ (it is easy to see combinatorially that $0<\rho  _\diamond<1$). All functions $T^{\mathsmaller{\mathsmaller{\square}}}, T^\vartriangle, T^{\circ} , T^{\mathsmaller{\mathsmaller{\square}}\rightarrow \vartriangle}, T^{\vartriangle\rightarrow \circ}$ can be defined with respect to $T_{\diamond}$, as indicated in Lemma~\ref{ulterior}. In particular, they depend on $T_{\diamond}$ in three different ways: by composing $T_{\diamond}(x)$ with either a polynomial having positive coefficients or the exponential function, by performing a change of variables from $x$ to $x^k$, and by the operator $\exp\big(\frac{1}{k}\sum _{k\geq 2}T_\diamond (x^k)\big)$. We observe that all three of them preserve the number and nature of singularities, hence these are determined solely by the behaviour of $T_\diamond$. In the case of composition with polynomials or exponentials, it is trivial to see.
  %\asingh{Έλεγε "In the cases of blabla it is trivial" χωρίς εξήγηση...} 
In the case of variable change, observe that $T_\diamond(x^k)$ has radius of convergence $\sqrt[k]{\rho _\diamond}>\rho _\diamond$. In the case of $\exp\big(\frac{1}{k}\sum _{k\geq 2}T_\diamond (x^k)\big)$, it is enough to notice that in $|x|<\rho _\diamond$ it holds that
\[\sum _{k\geq 2}T_{\diamond}(x^k)\leq T_{\diamond}(x^2)+\sum _{k\geq 3}x^{k-2}T_{\diamond}(x^2)=\frac{T_\diamond(x^2)}{1-x}.\] 
Therefore, it is enough to prove the claimed properties for $T_{\diamond}(z)$. 

To analyse $T_{\diamond}(z)$, we will use Theorem~\ref{ridicule}. Let
 \begin{equation}
F(x,y)  =  \frac{x}{2}\exp \bigg(y+ \sum _{k\geq 2}\frac{T_{\diamond}(x^k)}{k}\bigg) \bigg(\exp \bigg(2y+ \sum _{k\geq 2}\frac{2T_{\diamond}(x^k)}{k}\bigg)+\exp \bigg( \sum _{k\geq 1}\frac{T_{\diamond}(x^{2k})}{k}\bigg)\bigg).
\label{military}\end{equation}
The system $\{y = F(x,y),1=    F_y(x,y)\}$ can be solved numerically, using truncations of the functions $T_\diamond(z^k)$. We find a solution $(x_0,y_0)$, where $x_0\approx  0.15926 $ and $y_0\approx 0.41738$. Clearly, the rest of the requirements of Theorem~\ref{ridicule} are met and the coefficients of the desired expansion can be computed by Equations~\ref{analysts},~\ref{chrysler}. The coefficients for the expansions of $T^{\mathsmaller{\mathsmaller{\square}}} (x)$, $T^\vartriangle (x)$, $T^\circ (x)$, $T^{\mathsmaller{\mathsmaller{\square}}\rightarrow \vartriangle}(x)$, $ T^{\vartriangle\rightarrow \circ}(x)$ can be computed straightforwardly by the coefficients of $T_{\diamond}(z)$. Notice that Theorem~\ref{ridicule} guarantees $A_1\neq 0$ in all cases.
 \end{proof}

%\begin{lemma}
%	The asymptotic expansion of the generating function $T(x)$, at $\rho$, is of the form
%	\begin{equation*}
%	t_0 + \sum\limits_{k \geq 2} t_k X^k,
%	\end{equation*}
	%and in particular, $t_1 = 0$.
%\end{lemma}
%\begin{proof}
%	Consider the equations given in the statement of \autoref{ulterior}, where we have replaced $x$ and $T_{\star}(x)$ with their asymptotic expansions at $\rho$ and $T_{\star}(x^k)$, $k \geq 2$, with its analytic expansion at $\rho$. Then by applying the dissymmetry theorem, one obtains the asymptotic expansion of $T(x)$ at $\rho$ and in particular we have the following equation for $t_1$

%\end{proof}
 
\begin{lemma}\label{adopting}
The generating functions $Z(x),G(x)$ have a unique singularity of smallest modulus at the same positive number $\rho<1$. Moreover, they are analytic in a dented domain at $\rho $ and satisfy expansions 
\begin{equation*}
Z(x)=Z_0+\sum _{k\geq 2}Z_kX^k,\quad G(x)=G_0+\sum _{k\geq 2}G_kX^k,\quad \text{where}\quad X= \sqrt{1 - x/\rho},
\end{equation*}
locally around $\rho$. The coefficients $Z_i,G_i$, and $\rho $ are computable; in particular, $\rho\approx 0.15926$ (the same as in Lemma~\ref{stressed}).
\end{lemma}
\begin{proof}
By Equation~\ref{reliably}, the singular behaviour of $Z(x)$ depends entirely on the functions $T_i(x)$ that were studied in~Lemma~\ref{stressed} (recall that $Z(x)=T(x)$). In particular, $Z(x)$ has a unique positive singularity of minimum modulus at the same point $\rho$ and the same holds for $G(x)$. 

The coefficients of the expansions are directly computable by the coefficients of $T_i(x)$. In particular, we can show that the coefficient $Z_1$ vanishes identically and $Z_3\neq 0$. Let $A_0+A_1X+...$ be the expansion given by Lemma~\ref{stressed} for $T_\star (x)$ and notice that $A_0=T_{\star}(\rho )$. Then, $Z_1$ is equal to the following expression, which can be easily obtained on computational software such as \texttt{Maple}:
	\begin{equation*}
	Z_1=A_1 \left( \frac{3\rho  A_0^3}{2} + 
	\frac{\rho  A_0  C_0}{2} - 1\right),
	\end{equation*}
	where $C_0=T_{\star}(\rho ^2)$. Recall the function $F$ in \autoref{military} and the system $\{y = F(x,y),1=    F_y(x,y)\}$. The latter has solution $(\rho, y_0)$ and thus it holds that:
	\begin{align*}
	0 &= F_y(\rho, y_0)-1 \\
	&= \frac{3\rho}{2} \exp \bigg( y_0 +
		\sum_{i \geq 2} \frac{T_{\diamond}(\rho^{i})}{i} \bigg)^3+
	\frac{\rho}{2} \exp \bigg( y_0+ \sum _{k\geq 2}\frac{T_{\diamond}(x^k)}{k}\bigg) \exp \bigg( \sum _{k\geq 1}\frac{T_{\diamond}(\rho ^{2k})}{k}\bigg) - 1.
	\end{align*}
This  is equal to $\frac{1}{A_1} Z_1$, since $T_{\star}(\rho )=\exp\bigg( y_0+\sum_{i \geq 2} \frac{T_{\diamond}(\rho^{i})}{i}\bigg) $ and $T_{\star}(\rho ^2)=\exp \bigg( \sum _{k\geq 1}\frac{T_{\diamond}(\rho ^{2k})}{k}\bigg)$.
Thus, $Z_1=0$. This is a typical behaviour after applying the Dissymmetry Theorem (see~\cite{Bodirsky2007enum},~\cite{RueST13asym}). 

To see that $Z_3$ does not vanish, it is enough to argue combinatorially. First, observe that $t_n^{\bullet} \sim  \frac{c\rho ^{-n}}{n^{3/2}}$ by Lemma~\ref{stressed} and the Transfer Theorem (see Equation~\ref{fanfares} and the related account). If $Z_3$ vanished, then the singular exponent would be bigger than $3/2$. Consequently, by the Transfer Theorem we would obtain $n\cdot z_k=n\cdot t_k=o(  \frac{c\rho ^{-n}}{n^{3/2}})$ for large $n$, a contradiction to the asymptotic growth of $t_n^\bullet$.
\end{proof}

\begin{corollary}
The coefficients of $Z(x),G(x)$ satisfy an asymptotic growth of the form
\[cn^{-\frac{5}{2}}\rho^{-n},\] where $c$ is equal to $ \frac{Z_3}{\Gamma (-3/2)}\approx 0.27160$ and $\frac{G_3}{\Gamma (-3/2)}\approx 0.33995$, respectively, and $\rho ^{-1}\approx 6.27888$.
\end{corollary}
\begin{proof}
It follows by Lemma~\ref{adopting} and the Transfer Theorem. The computations are straightforward and can be easily confirmed on computational software such as \texttt{Maple} (see 
\begin{center}
\href{http://www.cs.upc.edu/~sedthilk/osmc/apexmo.mw}{http://www.cs.upc.edu/\~{}sedthilk/osmc/apexmo.mw}
\end{center}
 for the detailed calculations).
\end{proof}

 \newcommand{\bibremark}[1]{\marginpar{\tiny\bf#1}}
  \newcommand{\biburl}[1]{\url{#1}}

%
%	\bibliographystyle{plain}
%	 \bibliography{/Users/sedthilk/Dropbox/bib/complete}
%%	 \bibliography{complete}
%

\end{document}